\documentclass[reqno,11pt]{amsart}
\usepackage{amsfonts}
\usepackage{amsfonts,amsmath, amssymb}
\usepackage{hyperref}
\usepackage{marginnote}
\usepackage{color}
\usepackage{dsfont}
\usepackage{xcolor}
\usepackage{mathtools}
\usepackage{enumitem}
\usepackage[numbers,sort&compress]{natbib}

\numberwithin{equation}{section}

\newcommand{\R}{\mathbb{R}}

\newtheorem{theorem}{Theorem}[section]
\newtheorem{corollary}[theorem]{Corollary}
\newtheorem{lemma}[theorem]{Lemma}

\newtheorem{remark}[theorem]{Remark}
\theoremstyle{definition}
\newtheorem{definition}[theorem]{Definition}

\newtheorem{case}{Case}

\def\v{\varepsilon}

\def\d{{\rm d}}

\def\r{\rho}

\def\dd{{\, \rm d}}
\def\M{{\mathcal{M}}}

\def\bp{\boldsymbol{\psi}}
\def\supp{{\rm supp}}
\newcommand{\longrightharpoonup}{\rm - \!\!\! \rightharpoonup}
\usepackage{tikz}

\begin{document}
\title[Nonlinear Stability and Instability]{Nonlinear Stability and Instability of\\
Finite-Energy Solutions of the Compressible Euler-Riesz Equations with General Pressure Laws}

\author[J.A. Carrillo]{Jos$\acute{\mathrm{E}}$ A. Carrillo}
\address{J.A. Carrillo:\,Mathematical institute, University of Oxford, Oxford OX2 6GG, UK}
\email{carrillo@maths.ox.ac.uk}

\author[S.R. Charles]{Samuel R. Charles}
\address{S.R. Charles:\, Department of Mathematics, National University of Singapore, Singapore 119076, Singapore}
\email{samuel.charles@nus.edu.sg}

\author[G.-Q. Chen]{Gui-Qiang G. Chen}
\address{Gui-Qiang G. Chen:\, Mathematical Institute, University of Oxford, Oxford OX2 6GG, UK;
School of Mathematical Sciences, Fudan University, Shanghai 200433, China;
AMSS, Chinese Academy of Sciences, Beijing 100190, China}
\email{gui-qiang.chen@maths.ox.ac.uk}

\author[D.~F. Yuan]{Difan Yuan}
\address{Difan Yuan: \,School of Mathematical Sciences,  Beijing Normal University and Laboratory of Mathematics and Complex Systems,
Ministry of Education, Beijing 100875, China}
\email{yuandf@amss.ac.cn}

\begin{abstract}
The compressible Euler-Riesz equations arise in the modeling of a wide range of physical phenomena, including stellar dynamics, plasma physics, and mathematical biology. In this paper, we investigate the nonlinear stability and 
instability of steady states for the multidimensional compressible Euler-Riesz equations 
under general pressure laws.
In the polytropic case, we establish the nonlinear instability of steady states 
in the mass-supercritical regime for attractive potentials; this is achieved by analyzing 
the concavity of the free energy along mass-preserving dilations. 
At the mass-critical exponent, we show that, for any steady state, there exist solutions that start arbitrarily 
close to it, but develop growing support. 
For general pressure laws, we employ a concentration-compactness approach to 
prove the existence of energy minimizers and establish the nonlinear stability of steady states. 
Moreover, we quantify the finite-time stability by deriving a relative entropy 
bound for finite-energy solutions, without requiring uniform pointwise upper and lower bounds on the density.
We further exploit the convexity of the second moment to obtain quantitative growth estimates for solutions with positive energy, thereby proving the local nature of the stability result. Finally, we prove the global existence of finite-energy weak solutions to the compressible Euler–Riesz equations with spherical symmetry for general pressure laws via the compensated compactness method, thereby yielding unconditional stability around steady states within the class of weak solutions. The approach developed in this paper should be useful for solving other nonlinear partial differential 
equations involving similar difficulties.
\end{abstract}

\date{\today}

\maketitle

\thispagestyle{empty}

\section{Introduction}
We are concerned with the nonlinear stability and instability of 
the multidimensional (M-D) compressible Euler-Riesz equations (CEREs) in $\R^n:$
\begin{align}\label{1.1}
\begin{cases}
\displaystyle
\partial_t \rho+\mbox{div} \mathcal{M}=0,\\
\displaystyle
\partial_t \mathcal{M}+\mbox{div}\Big(\frac{\mathcal{M}\otimes \mathcal{M}}{\rho}\Big)+\nabla p+\kappa\rho\nabla\mathcal{W}_\alpha=\boldsymbol{0},\\
\end{cases}
\end{align}
for $(t, \boldsymbol{x})\in\R_+\times\mathbb{\R}^n$ with $n \geq 2,$
where $\r:\R_+\times\mathbb{\R}^n \to \mathbb{R}_+$ is the density, the pressure is prescribed by a barotropic constitutive law $p=p(\rho)$, $\mathcal{M}: \R_+\times\mathbb{\R}^n \to \mathbb{\R}^n$ represents the momentum,
and $\mathcal{W}_\alpha : \mathbb{R}^n \to \mathbb{R}$ is the Riesz potential for $\alpha > 0$ 
and the logarithmic potential for $\alpha = 0$.
In the analysis of the stability and instability of steady states of \eqref{1.1},
 we focus on the region $\alpha \in (0,n)$. 
 However, for the existence of solutions, we consider the larger region $\alpha \in (-1,n)$, so that
 $\mathcal{W}_\alpha$ is given by $\mathcal{W}_\alpha=\Phi_\alpha *\rho$ and 
 \[
 \nabla \mathcal{W}_\alpha(\boldsymbol{x}) =
 \begin{cases}
    \displaystyle
     (\nabla \Phi_\alpha * \rho) (\boldsymbol{x}) & \text{for } \alpha \in (-1,n-1), \\[1mm]
     \displaystyle
     \int_{\mathbb{R}^n} \nabla \Phi_\alpha(\boldsymbol{x} - \boldsymbol{y}) (\rho(\boldsymbol{y}) - \rho(\boldsymbol{x})) \dd \boldsymbol{y} \quad & \text{for } \alpha \in [n-1,n),
 \end{cases}
 \]
 where $\Phi_\alpha: \mathbb{R}^n \to \mathbb{R}$ is a Riesz kernel for $\alpha \in (0,n)$ 
 and a logarithmic kernel for $\alpha = 0$; that is, the kernel given by
 \[
 \Phi_\alpha(\boldsymbol{x}) =
 \begin{cases}
     -\frac{|\boldsymbol{x}|^{-\alpha}}{\alpha} &\quad \text{for } \alpha \neq 0, \\
     \log|\boldsymbol{x}| &\quad \text{for } \alpha = 0.
 \end{cases}
 \]
 Here, the logarithmic kernel can be obtained as a limit of the Riesz kernels, as $\alpha \to 0$, 
 in the sense that
 \begin{equation*}
    \lim_{\alpha \to 0} \big(\Phi_\alpha (\boldsymbol{x})+ \frac1{\alpha}\big) 
    = \lim_{\alpha \to 0} \frac{1-|\boldsymbol{x}|^{-\alpha}}{\alpha}
    = \log |\boldsymbol{x}|.
\end{equation*}
Whenever $\alpha$ appears, we use the convention that $\alpha = 0$ represents the logarithmic case. Here $\kappa > 0$ represents the attractive potential, and $\kappa < 0$ represents a repulsive potential. 
By scaling (without loss of generality), we always assume that $\kappa$ 
lies in the following set $\kappa \in \{-1,1\}$. 
The Riesz interactions arise in many physical settings, 
for example, plasma physics, stellar and galactic dynamics, condensed matter, random matrices, 
swarming, among others;
see \cite{Campa_2014,Lewin_2017,Cotar_2019,Lewin_2022,Carrillo_2021,Carrillo_2017}.
The Newtonian case corresponds to $\alpha=n-2$ (for $n\ge 3$) 
and, for $\alpha>n-2$, the kernel is more singular and can be viewed (up to constants) 
as the inverse of a fractional Laplacian $(-\Delta)^{\frac{n-\alpha}{2}}$.

\smallskip
For $\alpha = n-2$, the interaction reduces to the Newtonian/Coulomb potential. 
Coulomb gases serve as classical toy models, neglecting quantum effects; 
for example, Gamow’s liquid drop model ({\it cf}. \cite{Serfaty:2024ojp}) for atomic nuclei, electrons, 
and atoms reduces in certain regimes to a Coulomb-interacting particle system (see \cite{Alberti_2009}). 
Since the 1970s, Coulomb gases have been central to statistical mechanics ({\it cf}. \cite{Lieb_1969}) 
and density functional theory ({\it cf}. \cite{Lieb_2010,Lewin_2017,Cotar_2013,Cotar_2019}), 
particularly through the study of indirect Coulomb energy via the Lieb–Oxford inequality 
and optimal transport with Coulomb cost. 
For $n\ge3$, the critical case $\alpha=n-1$ corresponds to the Manev potential and exhibits similar interesting phenomena as in Vlasov--Riesz equations; see \cite{Bobylev,Cercignani,Illner}.

\smallskip
In two dimensions ($n=2$), the Newtonian/Coulomb interaction coincides with the logarithmic potential, leading to the log gas (also known as the 2-D one-component plasma, jellium, or Dyson gas), a standard model in 2-D plasmas, the fractional quantum Hall effect, and random matrix theory.

\smallskip
We consider a general pressure law $p: [0,\infty) \to [0,\infty)$, $p \in C^2([0,\infty))$ under the conditions that
\begin{equation}\label{Press}
    p'(\rho) > 0, \quad 2p'(\rho) + \rho p''(\rho) > 0,
\end{equation}
for $\rho > 0$, where the latter is the genuine nonlinear condition. 
For the stability of steady states, if the pressure is given by a power law
\[
p(\rho) = a_0 \rho^\gamma,
\]
we require that $\gamma > \frac{n + \alpha}{n}$. 
The value $\gamma = \frac{n+\alpha}{n}$ is known as the mass-critical exponent, 
arising from the balance between the internal and potential energies.
Another important scaling is the energy-critical scaling  $\gamma = \frac{2n}{2n - \alpha}$,
which arises from the scaling of the potential energy.
Classically, for $n=3$ and $\alpha = n-2 = 1$, the range $\gamma\in(\frac{6}{5},\frac{4}{3})$ 
is called the mass-supercritical case and the range $\gamma > \frac{4}{3}$ is called the mass-subcritical case. 
In the context of CEREs, these ranges become $\gamma\in(\frac{2n}{2n - \alpha}, \frac{n + \alpha}{n})$ 
and $\gamma > \frac{n + \alpha}{n}$, respectively. 
However, there are physical phenomena that do not have a corresponding pressure given by a power law. 
In the case of a white dwarf star where $n=3$ and $\alpha = n-2 = 1$, the pressure is of the form
\begin{equation}\label{white}
    p(\rho) = A \int^{B\rho^\frac{1}{3}}_0 \frac{\sigma^4}{\sqrt{D + \sigma^2}} \dd \sigma
\end{equation}
for $\rho > 0$ and $A,B,D > 0$, which has the following behavior for small and large densities
\begin{equation}\label{asymp}
    \begin{cases}
        p(\rho) \sim \frac{A B^5 }{5\sqrt{D}}\rho^\frac{5}{3} \quad& \text{ as } \rho \to 0, \\[2mm]
        p(\rho) \sim \frac{A B^4 }{4}\rho^\frac{4}{3} \quad& \text{ as } \rho \to \infty.
    \end{cases}
\end{equation}

\smallskip

We consider solutions of the Cauchy problem for \eqref{1.1} with initial data:
\begin{equation}\label{1.1-2}
(\rho, \M)|_{t=0}=(\rho_0, \M_0)( \boldsymbol{x})\longrightarrow (0, \boldsymbol{0})
\qquad\,\,\,\,\mbox{as $| \boldsymbol{x}|\to \infty$}.
\end{equation}
In \eqref{1.1-2}, the initial far-field velocity is assumed to be zero
without loss of generality due to the Galilean invariance of system \eqref{1.1}. 
Since global solutions of CEREs may contain vacuum states 
$\{(t,\boldsymbol{x}) \in [0,\infty) \times \mathbb{R}^n\, : \, \rho(t,\boldsymbol{x}) = 0\}$, 
we consider the momentum $\mathcal{M}$, instead of the velocity $u = \frac{\mathcal{M}}{\rho}$ 
which may not be well-defined in these states. 
Under these assumptions, the energy of CEREs is given by
\begin{equation*}
     E(\rho,\mathcal{M})(t): 
     =\int_{\mathbb{R}^n}\Big( \frac{1}{2} \Big|\frac{\mathcal{M}}{\sqrt{\rho}}\Big|^2(t, \boldsymbol{x}) + \rho(t, \boldsymbol{x}) e(\rho(t, \boldsymbol{x}))  
     + \frac{\kappa}{2}\rho(t,\boldsymbol{x})(\Phi_\alpha \ast \rho)(t, \boldsymbol{x})\Big)\,\dd\boldsymbol{x},
\end{equation*}
where
\[
    e(\rho) = \int^\rho_0 \frac{p(s)}{s^2}\,\dd s.
\]

\smallskip

CEREs may be thought of as a generalization of the compressible Euler-Poisson equations (CEPEs), 
where the potential is given by a solution of the Poisson equation $\Delta \mathcal{W} = \kappa \rho$, 
that is, the Newtonian/Coulomb case ($\alpha = n-2$). 
It was proven in Chen-He-Wang-Yuan \cite{Chen2021} that, under the assumption of radially symmetric initial data, 
there exist weak global-in-time finite-energy solutions to the polytropic CEPEs. 
This was further extended to more general pressure laws 
in Chen {\it et al.}  
\cite{Chen_2024}. An interesting direction concerns global-in-time dynamics in the
physical-vacuum free-boundary setting. Hadzi\'c and Jang constructed
global-in-time solutions of the 3-D compressible Euler equations
without symmetry assumptions near expanding compactly supported affine motions
\cite{HadzicJang2018Euler}. They
subsequently obtained a class of global-in-time compactly supported solutions
to the 3-D CEPEs, again without symmetry
assumptions, in both the gravitational and plasma cases
\cite{HadzicJang2019EP}. For these solutions, the density tends to zero and
the support expands linearly in time.
\smallskip

In the general Riesz setting, it was seen in a series of papers \cite{Chi_2025,Choi_2022_2,Choi_2025_2,Choi_2025,Danchin_2022} that, 
if the initial data is small enough or close enough to some background state, 
then there exist global-in-time classical solutions of CEREs for $\alpha \in (n-2,n)$. 
However, since CEREs are strongly hyperbolic and nonlinear, it cannot be expected, in general, 
to admit global smooth solutions for all time. 
In fact, for large initial data, smooth solutions may break down in finite time; see Choi–Jeong \cite{Choi_2022}. 
Such a breakdown can occur through mechanisms including the formation of shock waves in the solution 
and the formation of cavitation or concentration of the density. 
Consequently, for general initial data, it is more natural to study weak solutions of CEREs. 
This was carried out in Carrillo–Charles–Chen–Yuan \cite{Carrillo2025} for the polytropic CEREs 
where, under the assumption of radially symmetric initial data and $\alpha \in (0,n-1)$, 
the existence of global-in-time finite-energy weak solutions was proved.

\smallskip
It is of physical interest to consider the nonlinear stability and instability of 
steady states of \eqref{1.1}, that is, solutions of the equation:
\[
    \nabla p(\rho) + \kappa \rho \nabla \mathcal{W}_\alpha = \boldsymbol{0}.
\]
The notion of being a steady state of \eqref{1.1} is equivalent to the notion of being a steady state 
of the aggregation-diffusion equation:
\[
\partial_t \rho=\nabla\cdot(\nabla p(\rho) + \kappa \rho \nabla \mathcal{W}_\alpha).
\]
This connection motivates the steady-state problem within the variational theory
of aggregation--diffusion equations. In particular, radial symmetry, compact
support, and long-time asymptotics of equilibria were studied in
\cite{Carrillo_2019}, while the uniqueness and nonuniqueness mechanisms for
fixed-mass steady states under attractive interactions were investigated in
\cite{DelgadinoYanYao2022}; see also \cite{Carrillo_2018} for related
ground-state results. For a broader overview of aggregation--diffusion
dynamics, asymptotics, and singular limits, we refer to \cite{CCY19}.
Here we consider steady states nonlinearly stable if any solution that starts ``close'' 
enough at $t=0$ in some sense (usually in the sense of a normed vector space) 
to the steady state remains ``close'' to it for all time $t>0$. 
Such steady states are usually given as minimizers of the free energy functional:
\[
\mathcal{G}(\varrho)=\int_{\mathbb{R}^n} \Big(\varrho(\boldsymbol{x}) e(\varrho(\boldsymbol{x}))+\frac{1}{2}\varrho(\boldsymbol{x})(\Phi_\alpha\ast \varrho)(\boldsymbol{x})\Big)\,\dd\boldsymbol{x}.
\]
In the mass-critical polytropic case $\gamma=\frac{n+\alpha}{n}$, for any fixed $\alpha$, there exists a unique steady solution $\rho_s$ modulo mass normalization and translations, which is referred to as
the Lane-Emden-Fowler type solution. The steady state
$\rho_s(|\boldsymbol{x}|)$ has compact support and is determined by the equation:
\begin{equation*}
\nabla p_s + \kappa \rho_s\,\nabla \Phi_{\alpha} * \rho_s=\boldsymbol{0},
\end{equation*}
where $p_s(|\boldsymbol{x}|)=a_0(\rho_s(|\boldsymbol{x}|))^{\frac{n + \alpha}{n}}$.
In particular, for $\alpha = n-2$, the steady solution is the Lane-Emden solution. 

Recent studies of the attractive CEPEs have also addressed genuinely
time-dependent stellar configurations, including expanding and collapsing
solutions. In the radially symmetric mass-critical case, Hadzi\'c-Jang \cite{HadzicJang2018Stars}
proved the nonlinear stability of compactly supported expanding star
solutions. In the mass-supercritical regime
$\gamma\in(1,\frac{4}{3})$, Guo-Hadzi\'c-Jang \cite{GuoHadzicJang2021} constructed an
infinite-dimensional family of continued-collapse solutions, while Guo-Hadzi\'c-Jang-Schrecker \cite{GuoHadzicJangSchrecker2022}
constructed smooth radially symmetric self-similar imploding solutions
throughout the same range.

We now turn to the stability and instability theory for steady states. In the context of the attractive CEPEs {\rm (}$\kappa = 1${\rm )}, Rein \cite{Rein_2003} proved that, 
under the assumption of the existence of solutions to the 3-D attractive CEPEs and the uniqueness of the minimizer, 
the minimizer is nonlinearly stable for a given class of pressure laws, which corresponds to 
$\gamma > \frac{4}{3}$ in the polytropic case. Subsequently, Luo-Smoller 
\cite{LuoSmoller2009} established the existence of rotating-star minimizers with prescribed 
total mass and angular momentum and, assuming the uniqueness up to vertical translations, proved nonlinear dynamical stability under perturbations having the same total mass and symmetry 
as the steady state; see also  \cite{Luo_2008,LinWangZhu2025}. For the attractive polytropic CEREs, Carrillo-Charles-Chen-Yuan \cite{Carrillo2025} proved 
the existence and unconditional stability of minimizers. However, Deng-Liu-Yang-Yao \cite{Deng_2002} proved for the polytropic CEPEs for $\gamma > \frac{4}{3}$, the linear 
growth of the support of solutions with positive energy over time, thereby proving 
the locality of the stability result.  We also refer to Jang \cite{Jang2008,Jang2014} for the nonlinear instability 
of the polytropic stars $(n=3,\alpha=n-2)$ for $\gamma\in[\frac{6}{5},\frac{4}{3})$. In Cheng-Cheng-Lin \cite{Cheng_2025}, it was proven that steady states of the attractive 
polytropic CEPEs for $\gamma\in(\frac{6}{5},\frac{4}{3}]$ are nonlinearly unstable, 
which was based on a virial identity. The use of such methods date back to Sideris \cite{Sideris_1985}, who used
averaged moment quantities (a virial-type
argument) to prove finite-time breakdown of $C^1$ solutions of the
3-D compressible Euler equations under suitable
assumptions on the initial data.

\smallskip
Even in the classical setting of CEPEs, the dynamical stability theory faces several fundamental difficulties. 
The system is hyperbolic and strongly nonlinear, so smooth solutions typically break down in finite time 
due to behaviors such as shock formation, cavitation, and density concentration.
This motivates the study of weak solutions. 
In addition, the physically relevant steady states are compactly supported and therefore 
involve vacuum regions, which introduce additional degeneracy at the vacuum interface. 
Moreover, the stability and instability of steady states are closely tied to the delicate 
variational structure of the system and 
the competition between the internal and potential energies, 
with the mass-critical exponent $\gamma=\frac{4}{3}$ for $n=3$ playing a central role. 
For example, under suitable assumptions, the nonlinear stability of minimizers 
was established by Rein \cite{Rein_2003} for $\gamma>\frac{4}{3}$ in the polytropic case, 
while the nonlinear instability in the astrophysically relevant range was studied in \cite{Jang2008,Jang2014,Cheng_2025}. 
A further essential difficulty lies in quantifying the dynamical behavior beyond the mere orbital stability: 
for CEPEs in $n=3$, one can prove the linear growth in time of the support for positive energy solutions \cite{Deng_2002}, which shows the inherently local nature of variational stability.
 
\smallskip
Passing from CEPEs to CEREs introduces genuinely new analytic and mathematical challenges. One of the main challenges is the loss of local elliptic structure and the presence of a broad range of singularity types 
at the origin of the potential.
For general $\alpha$, the potential is nonlocal, and $\alpha>n-2$ corresponds to a fractional inverse. 
More delicate non-local estimates near the origin are needed in order 
to obtain local higher integrability estimates of the density required for the compensated 
compactness method, fundamental in proving the existence of solutions. 

\smallskip
Another challenge arises from the different scaling competition compared with CEPEs.
The balance between the internal energy and the Riesz potential energy 
yields the mass-critical exponent:
$$
\gamma=\frac{n+\alpha}{n},
$$
which depends on $\alpha$ and determines whether the minimizing sequences are tight (stability) 
or can lower the energy by dilation (instability).
Moreover, unlike the treatment of CEPEs in \cite{Cheng_2025}, 
the convexity or concavity of the internal energy and the Riesz potential 
energy under the mass-preserving dilations depends {\it separately} on $(\alpha,\gamma)$. 
This leads to new {\it mixed regimes}, in which the two energy components compete under scaling. 
In addition, for the analysis of the stability and instability of steady states, 
the weak convergence of minimizing sequences is not sufficient to 
prove the convergence to a minimizer of the nonlocal energy. 
In the variational approach, the minimizing sequences are typically bounded only 
in $L^1_+\cap L^\gamma (\mathbb{R}^n)$, so the weak convergence alone does not  
imply the convergence of the interaction term:
$$
\int_{\mathbb{R}^n} \varrho_m\Phi_\alpha*\varrho_m \dd \boldsymbol{x}.
$$ 
Therefore, additional mechanisms are required to pass to the limit 
in the nonlocal term. 
To address this issue, we quantify the stability estimates via a relative entropy method. 
Such estimates require a careful analysis of the bound given by the difference 
of the Riesz potentials. 
Previous works relied on strong assumptions, such as $L^\infty$ bounds 
on the density and the momentum, 
and the restriction $\gamma \leq 2$ (see \cite{Luo_2008}), in order to control these terms. 
More recently, the result of Alves-Carrillo-Choi \cite{Alves_2024} assumes that 
the steady state is bounded from above and below by a positive constant, 
which remains a rather restrictive assumption about the steady state. 
In contrast, establishing the stability estimates for weak solutions under weaker assumptions
is considerably more delicate and requires new analytical tools.

\smallskip
Finally, another new challenge in the paper is to establish the linear growth of 
positive-energy solutions 
for $\alpha > 2$. In \cite{Deng_2002}, 
such growth of solutions was proved for CEPEs in the 3-D case, 
where $\alpha = 1$.
The approach involves second moments, which are key to proving the growth of the support, 
and uses a virial identity to obtain a positive lower bound on the second derivative 
of the second moment of the density. 
This property is crucial for proving the growth of the second moment and, consequently, 
due to the conservation of total mass of the fluid, the expansion of the support. 
However, in the present setting, the competition between the potential energy 
and the kinetic energy leads to a fundamentally different situation. 
Using the approach of \cite{Deng_2002}, 
we obtain a lower bound for the second derivative of the second moment that involves the term
\[
    (2-\alpha)\int_{\Omega(t)}\Big| \frac{\mathcal{M}}{\sqrt{\rho}} \Big|^2\dd \boldsymbol{x},
\] 
in the proof of Lemma \ref{E>0,general}. 
When $\alpha>2$, this term becomes negative, and hence it is no longer possible to deduce that 
the second derivative of the second moment is bounded below by a positive constant. 
As a result, the classical argument breaks down.
This difficulty does not arise in the existing literature, which focuses on the Newtonian potential 
in three dimensions, corresponding to $\alpha=1$. In contrast, for $\alpha>2$, including 
the Newtonian/Coulomb interaction in higher dimensions ($n\ge 5$), a direct extension 
of the growth argument in \cite{Deng_2002} to CEREs fails.

\smallskip
In this paper, we establish the nonlinear instability in the mass-supercritical regime.
For the attractive polytropic CEREs, we prove the nonlinear instability of steady states for
$$
\frac{2n}{2n-\alpha}<\gamma<\frac{n+\alpha}{n},
$$
by exploiting the concavity of the free energy along the mass-preserving dilations; 
see Theorem \ref{smallgamma} and Section \ref{instability}. 
Moreover, through a careful analysis of the asymptotic behavior of key critical scalings,
we extend the instability result to the regimes not covered in \cite{Cheng_2025}. 
Furthermore, we establish a new compactness result for the potential term 
in the supercritical range in the radial case, which is crucial for proving 
the existence of steady states. This argument relies on the local compact 
embeddings of fractional Sobolev spaces. 
While insights have often flowed from the theory of aggregation–diffusion equations 
to the study of CEREs, the instability of steady states in the corresponding parameter range remains poorly
understood for the aggregation-diffusion equations.
The current literature primarily establishes 
finite-time blow-up of solutions with large initial data for the Newtonian potential, 
as in Chen-Wang \cite{Chen_2014}. Related finite-time blow-up dynamics for aggregation equations with degenerate
diffusion and power-law attractive potentials were analyzed by Yao-Bertozzi
\cite{YaoBertozzi2013}. Our results suggest that analogous instability phenomena, driven by different mechanisms, 
should also hold for aggregation–diffusion equations.

\smallskip
We establish the nonlinear stability of minimizers in the subcritical regime for general pressure 
laws. 
Under suitable conditions on $p$ including \eqref{Press} and appropriate asymptotic behavior, 
we prove the existence of minimizers via concentration-compactness \cite{Lions_1984} 
and establish the nonlinear orbital dynamical stability of these steady states; 
see Theorem \ref{ns} and Section \ref{stability}. 
Such conditions on the pressure are more general than those in previous results, 
allowing for differing behaviors near the vacuum and at infinity. 

We also prove a quantitative stability result by developing a new
relative-entropy method for finite-energy solutions.  This part of the analysis
goes beyond the variational orbital stability argument and addresses the
dynamical control of perturbations around compactly supported steady states.
The main difficulty comes from the nonlocal interaction terms generated in the
relative entropy inequality, in particular the terms:
\[ \int_{\mathbb R^n} \rho
 \left| \nabla\Phi_\alpha*(\rho-\bar\rho) \right|^2\,\dd\boldsymbol{x}.
\]
Such terms are delicate to estimate because the steady density $\bar\rho$ is compactly
supported and may vanish at the free boundary.  Hence, one cannot rely on a
uniform positive lower bound for $\bar\rho$, which is often used in
relative-entropy arguments to control the density perturbation.

A key point of our approach is to exploit the variational structure of the
steady state in the construction of the relative entropy.  More precisely, the
Euler--Lagrange condition for the minimizer yields
\[ G_\mu=-h'(\bar\rho) \quad\text{on } \{\bar\rho>0\}, \qquad\,\,\, G_\mu\ge0
\quad\text{on } \{\bar\rho=0\},
\]
where $G_\mu=\mu+\bar{\mathcal W}_\alpha.$
This allows us to work with a modified relative internal energy of the form:
\[ h_\mu(\rho\mid\bar\rho) = h(\rho)-h(\bar\rho)+G_\mu(\rho-\bar\rho),
\]
which coincides with the standard relative internal energy on the support of
the steady state, while  it gives
\[ h_\mu(\rho\mid0) =\frac{a_0}{\gamma-1}\rho^\gamma+G_\mu\rho
\]
on the vacuum region.
The additional nonnegative term $G_\mu\rho$ provides useful information on
the density perturbations outside the support of $\bar\rho$.  This is a key point that replaces the usual lower-bound assumption on the steady state
density.

Using this modified relative entropy, we establish new estimates for the density
difference $\rho-\bar\rho$. The key coercivity estimate is given by Lemma \ref{coer}, which relies on the variational relation of the
steady state \eqref{property}. 
The resulting density estimates are obtained by separating the support of $\bar\rho$, the
bounded vacuum region, and the far-field region.  On the support of
$\bar\rho$, the convexity of the internal energy yields the local coercivity. In the
vacuum region, the variational inequality for $G_\mu$ gives additional control,
while in the exterior region the decay of the steady nonlocal potential term provides an
$L^1$-type estimate.  Combining these bounds with the higher integrability
assumptions on $\rho$ and $\mathcal M$, we succeed in controlling the relative non-local term
through the Hardy-Littlewood-Sobolev and H\"older inequalities; see Lemma \ref{relativenonlocal} and, in particular,
\eqref{relativeestimate}. This yields a relative entropy stability estimate without assuming that the
steady density is bounded from below by a positive constant; see
Lemma \ref{effrelativeei} and Theorem \ref{thm:stability}.  Compared with
\cite{Alves_2024}, the argument therefore applies to compactly supported
steady states with vacuum free boundaries.  Different from
\cite{Luo_2016_1}, our method avoids the strong assumptions
$\rho,\mathcal M\in L^\infty,$ 
and instead uses the finite-energy bounds together with natural higher
integrability conditions. In \cite{Luo_2016_1}, the steady solutions (density and momentum) are required to belong to $W^{1,\infty}_{{\rm loc}}.$  In this paper, we define the steady solutions in \eqref{definesteady}, which have lower regularity. The 
relative-entropy estimate is obtained in the range $\gamma\le2$,
where the modified relative internal energy controls the small vacuum
perturbations at least quadratically.  This identifies the precise role of the
vacuum degeneracy in the relative entropy method and is one of the main
technical novelties of the paper.
\smallskip

We establish linear growth of the support for positive-energy solutions, 
including the regime $\alpha>2$.
More precisely, we derive lower bounds on the asymptotic growth rate of the radius of 
the support for global classical solutions with positive energy. In the case 
$\alpha > 2$, we overcome the presence of a term with an unfavorable sign by showing that, 
under a suitable mass threshold, the internal energy dominates the potential energy. 
This yields a uniform positive lower bound on the second derivative of the second moment; 
see Theorem \ref{growth}.

\smallskip
Finally, the global existence of spherically symmetric finite-energy solutions is proved for 
CEREs with general pressure laws.
In particular, our results apply to a \textit{broad} range of 
the Riesz potential $\alpha\in(-1,n-1)$, 
leading to \emph{unconditional} stability results; 
see Theorem \ref{existence} and Section \ref{existy}.

\smallskip
The structure of the paper is as follows: 
In Section \ref{section2}, we introduce the variational approach, 
the steady states, and the main theorems.
In Section \ref{section3}, we establish the growth estimates for positive-energy solutions.
Section \ref{instability} is devoted to the instability results in the supercritical regime.
In Section \ref{stability}, we prove the existence and stability of minimizers
and develop the quantitative finite-time stability estimates via the relative entropy.
Section \ref{existy} addresses the existence of global finite-energy solutions to
CEREs with general pressure laws.
The appendices collect several auxiliary results, including convolution inequalities, 
decay and regularity estimates for radial nonlocal potentials, 
and background material on fractional Sobolev spaces.

\section{Preliminaries and Main Theorems}\label{section2}
In this section, we present the preliminaries and main theorems 
on the stability and instability of steady states, as well as 
the existence of global finite-energy weak solutions to CEREs with general pressure laws.

\subsection{Steady States of CEREs}
One of the important results of this paper is the nonlinear stability and instability
of the minimizers to the energy functional relating to CEREs.

\subsubsection{Steady states for general and polytropic power laws with $\gamma > \frac{n + \alpha}{n}$}
For $\alpha \in (0,n)$, $\kappa =1$ and $0<M<\infty$, define the set of admissible
functions{\rm:}
\begin{equation*}
X_M := \Big\{ \varrho \in L^1_+(\mathbb{R}^n) \cap L^\frac{n + \alpha}{n}(\mathbb{R}^n)
\, : \, \int_{\mathbb{R}^n} \varrho(\boldsymbol{x}) \, \dd \boldsymbol{x} = M \Big\}.
\end{equation*}
For $\varrho\in X_M$, define the energy functional $\mathcal{G}$:
\begin{equation*}
\mathcal{G}(\varrho)=\int_{\mathbb{R}^n} \Big(\varrho(\boldsymbol{x}) e(\varrho(\boldsymbol{x}))+\frac{1}{2}\varrho(\boldsymbol{x})(\Phi_\alpha\ast \varrho)(\boldsymbol{x})\Big)\,\dd\boldsymbol{x}.
\end{equation*}
For $\varrho, \bar{\rho} \in X_M,$  define the  relative free-energy functional:
\begin{equation*}
\begin{split}
    d(\varrho,\bar{\rho}) & := \int_{\mathbb{R}^n}\big((\varrho e(\varrho)-\bar{\rho}e(\bar{\rho}))+(\varrho-\bar{\rho})(\Phi_\alpha\ast\bar{\rho})\big)\,\dd \boldsymbol{x}\\
    & = \int_{\mathbb{R}^n}\big((\varrho e(\varrho)-\bar{\rho}e(\bar{\rho}))+(\varrho-\bar{\rho})(\mu + \Phi_\alpha\ast\bar{\rho})\big)\,\dd \boldsymbol{x}
\end{split}
\end{equation*}
for any $\mu > 0$. If $\bar{\rho}$ is a minimizer, then $\mu$ can be taken to be the Lagrange multiplier given by Theorem \ref{lmul}. Given a solution $(\rho(t),\mathcal{M}(t))$ of the attractive CEREs {\rm (}$\kappa = 1${\rm )} with mass $M$ and minimizer of $\bar{\rho} \in X_M$ of $\mathcal{G}$ over $X_M$, we define the relative energy
\begin{equation}\label{Energydiff}
     E(\rho,\mathcal{M}\, | \, \bar{\rho}, \boldsymbol{0})(t):= E(\rho,\mathcal{M})(t) - E(\bar{\rho},\boldsymbol{0}) = \mathcal{G}(\rho(t)) - \mathcal{G}(\bar{\rho})
    + \frac{1}{2}\int_{\mathbb{R}^n} \Big|\frac{\mathcal{M}}{\sqrt{\rho}}\Big|^2(t,\boldsymbol{x})\,\dd\boldsymbol{x}.
\end{equation}
In order for a minimizer of $\mathcal{G}$ over $X_M$ to exist, we require
\begin{equation}\label{lower}
    M > \bigg ( \frac{2n \limsup_{\rho \to 0} p(\rho) \rho^{-\frac{n +\alpha}{n}}}{ C_*} \bigg )^\frac{n}{n-\alpha},
\end{equation} 
where $C_* > 0$ is the optimal constant given by the Hardy-Littlewood-Sobolev (HLS) inequality as in Lemma \ref{HLSineq}.
Such a condition ensures that the potential energy dominates the internal energy at low densities. Consequently, by applying Lions' concentration–compactness argument, one can rule out both vanishing and dichotomy for any minimizing sequence
of any minimizing sequences of $\mathcal{G}$ over $X_M$.

We also require that
\begin{equation}\label{higherpress}
    M < \bigg ( \frac{2n \liminf_{\rho \to \infty} p(\rho) \rho^{-\frac{n+\alpha}{n}}}{C_*} \bigg )^\frac{n}{n-\alpha},
\end{equation}
which ensures that the internal energy dominates the potential energy at large densities, a necessary condition for the existence of minimizers of $\mathcal{G}$ over $X_M$. 

\begin{remark}
For a polytropic fluid with the pressure law
$p(\rho) = a_0 \rho^\gamma$, if $\gamma > \frac{n + \alpha}{n}$, conditions \eqref{lower} and \eqref{higherpress} reduce to the condition that $0<M<\infty$.
\end{remark}

\begin{remark}
    When $\gamma = \frac{n+\alpha}{n}$, 
    the lower and upper bounds for the mass in \eqref{lower} and \eqref{higherpress} 
    coincide, both reducing to
    the same value $\big( \frac{2na_0}{C_*} \big )^\frac{n}{n-\alpha}$.
    In this case, there exists a unique steady state {\rm (}up to translation and mass-preserving dilation{\rm )} 
    at this critical mass, and no steady states exist for other values of the mass{\rm ;} 
    see {\rm \cite[Theorem 3]{Calvez_2021}}.
\end{remark}

Condition \eqref{higherpress} is analogous to that of the Chandrasekhar limit mass. 
In the polytropic case with $\alpha = n-2$ and $\gamma = \frac{2(n-1)}{n}$, 
this limit mass is given by 
$$
M_{\rm ch}:= \Big( \frac{2 n a_0}{C_*} \Big )^\frac{n}{2}
$$ 
which characterizes steady gaseous stars; see \cite{Chandrabook}. 
This can be understood by considering a white dwarf star, for which the pressure 
is given by \eqref{white}. 
For $n=3$, conditions \eqref{lower} and \eqref{higherpress} reduce 
to $$
0<M<\Big (\frac{3AB^4}{2 C_*} \Big )^\frac{3}{2}.
$$ 
Moreover, it follows from \eqref{asymp} that $p(\rho) \sim a_0 \rho^\frac{4}{3}$, 
where $a_0 = \frac{AB^4}{4}$. 
Substituting this expression shows that the upper bound coincides exactly 
with the Chandrasekhar limit mass $M_{\rm ch}$. 
From a physical perspective, this agreement is expected, as it reflects the classical
phenomenon that white dwarf stars collapse once their mass exceeds the Chandrasekhar limit mass; see \cite{Chandrabook}. As in \cite[Theorem 5]{Carrillo_2018}, we have

\begin{theorem}\label{lmul}
Suppose that there exists an energy minimizer $\bar{\rho}$ of the energy 
functional $\mathcal{G}$ in $X_M,$ and define
\[
 \Gamma:= \{ \boldsymbol{x} \in \mathbb{R}^n \, : \, \bar{\rho}(\boldsymbol{x}) > 0 \}.
\]
Then there exists a constant $\mu \geq 0$ such that
\begin{align}\label{station}
\begin{cases}
\frac{\dd}{\dd \rho}(\rho e(\rho))|_{\rho=\bar{\rho}}+\Phi_\alpha\ast\bar{\rho}= - \mu 
\quad & \mbox{for $\boldsymbol{x}\in \Gamma$},\\[1mm]
\Phi_\alpha\ast\bar{\rho}( \boldsymbol{x})\geq - \mu 
& \mbox{for $\boldsymbol{x}\in \R^n\setminus\Gamma$}.
\end{cases}
\end{align}
\end{theorem}

\subsubsection{Steady states for the polytropic CEREs with
$\frac{2n}{2n - \alpha} < \gamma < \frac{n + \alpha}{n}$}
Note that \eqref{station} is derived under the assumptions 
that \eqref{lower} and \eqref{higherpress} hold, which are  
equivalent to $\gamma > \frac{n+\alpha}{n}$ in the polytropic case.
For $\gamma > \frac{2n}{2n-\alpha}$, it was shown in \cite{Chan_2020} and \cite{Flucher_1998} 
that, for $a > 0$, $\mu > 0$, $1 \leq p < \frac{2n - \alpha}{\alpha}$, 
and $\alpha \in (0,n) \cap [n-2,n)$, there exists a unique spherically 
symmetric $v \in L^\infty(\mathbb{R}^n)$ satisfying
    \[
   (-\Delta)^\frac{n-\alpha}{2} v = a (v - \mu)_+^p.
    \]
Moreover, for $\rho \in L^1 \cap L^\gamma(\mathbb{R}^n)$, 
the fractional Laplacian $(-\Delta)^\frac{n-\alpha}{2}$ is invertible, 
with the inverse given (up to a multiplicative constant) by the Riesz potential $\Phi_\alpha$. 
In particular, if
    \[ (-\Delta)^\frac{n-\alpha}{2} v = c_{n,\alpha} \rho
    \]
in the sense of distributions, where
    \[c_{n,\alpha} = \frac{2^{n-\alpha} \pi^\frac{n}{2} \Gamma\left (\frac{n-\alpha}{2} \right )}{\alpha \Gamma\left (\frac{\alpha}{2} \right )},
    \]
then $v = -\Phi_\alpha * \rho$. 
Consequently, choosing $a = c_{n,\alpha} \Big ( \frac{\gamma - 1}{a_0 \gamma} \Big )^\frac{1}{\gamma - 1}$
and $p = \frac{1}{\gamma - 1}$, 
we deduce the existence of a unique spherically symmetric 
$\bar{\rho}_\mu \in L^1 \cap L^\gamma(\mathbb{R}^n)$ with $\bar{\rho}_\mu \geq 0$ {\it a.e.} and $\bar{\rho}_\mu \not \equiv 0$ such that $v = - \Phi_\alpha * \bar{\rho}_\mu$ and
    \[  \bar{\rho}_\mu = \Big ( \frac{\gamma - 1}{a_0 \gamma} \Big )^\frac{1}{\gamma - 1} (- \Phi_\alpha * \bar{\rho}_\mu - \mu)_+^\frac{1}{\gamma - 1}.
    \]
Hence, there is a unique steady state (up to translation and mass-preserving dilation) 
for $\mu >0$, as described 
in \eqref{station}. 
Moreover, Chan-Gonz\'alez-Huang-Mainini-Volzone showed that $\bar{\rho}_\mu$ has compact support 
and is radially decreasing. 
We will see later that both properties can also be derived from our analysis.
For $\alpha \in (0,n)$, let $\bar{\rho}_1$ be a spherically symmetric solution 
of \eqref{station} corresponding to $\mu = 1$.
Consider the scaling 
$\rho_\mu(\boldsymbol{x}) := 
\mu^\frac{1}{\gamma - 1}\bar{\rho}_1(\mu^\frac{2-\gamma}{(\gamma - 1)(n-\alpha)} \boldsymbol{x})$. 
Then 
    \[  (\Phi_\alpha * \rho_\mu)(\boldsymbol{x}) 
= \mu (\Phi_\alpha * \bar{\rho}_1)(\mu^\frac{2-\gamma}{(\gamma - 1)(n-\alpha)}\boldsymbol{x}).
    \]
Consequently, we obtain 
\[
\begin{split}
\rho_\mu(\boldsymbol{x}) & 
= \Big ( \frac{\gamma - 1}{a_0 \gamma} \Big )^\frac{1}{\gamma - 1} (- \mu \Phi_\alpha * \bar{\rho}_1 - \mu)_+^\frac{1}{\gamma - 1}(\mu^\frac{2-\gamma}{(\gamma - 1)(n-\alpha)}\boldsymbol{x}) \\
 & = \Big ( \frac{\gamma - 1}{a_0 \gamma} \Big )^\frac{1}{\gamma - 1} (- \Phi_\alpha * \rho_\mu - \mu)_+^\frac{1}{\gamma - 1}(\boldsymbol{x}),        \end{split}
\]
so that $\rho_\mu$ is a solution of \eqref{station} for $\mu > 0$. 
For $\alpha \in [n-2,n)$, $\rho_\mu$ remains spherically symmetric and satisfies \eqref{station}.
By uniqueness of steady states, we conclude that 
$\bar{\rho}_\mu(\boldsymbol{x}) 
= \mu^\frac{1}{\gamma - 1}
\bar{\rho}_1 (\mu^\frac{2-\gamma}{(\gamma - 1)(n-\alpha)} \boldsymbol{x})$. 

Similarly, the existence of steady states can be established via a variational approach by
considering the energy functional:
\[
S_\mu(\varrho) := \frac{a_0}{\gamma - 1} \int_{\mathbb{R}^n} \varrho^\gamma \dd \boldsymbol{x} + \frac{1}{2} \int_{\mathbb{R}^n} \varrho \Phi_\alpha * \varrho \dd \boldsymbol{x} + \mu \int_{\mathbb{R}^n} \varrho \dd \boldsymbol{x}.
\]
This functional is considered over the admissible set:
\[
\mathcal{K} := \big\{ \varrho \in L^1 \cap L^\gamma(\mathbb{R}^n) \,\, : \, Q(\varrho) = 0, \, \varrho \geq 0 \text{ {\it a.e.,} } \varrho \not\equiv 0\big\},
\]
where 
\[
Q(\varrho) := n a_0 \int_{\mathbb{R}^n} \varrho^\gamma \dd \boldsymbol{x} + \frac{\alpha}{2} \int_{\mathbb{R}^n} \varrho \Phi_\alpha * \varrho \dd \boldsymbol{x}.
\]
The functional $Q(\varrho)$ corresponds to the derivative of $S_\mu(\varrho)$ 
in the scaling parameter defining mass-preserving dilations. 
Moreover, as in Theorem \ref{lmul}, we show in Section \ref{instability} that
 
\begin{theorem}\label{minform}
Suppose that there exists an energy minimizer $\bar{\rho}$ of the energy functional $S_\mu$ in $\mathcal{K},$ and define
\[ \Gamma := \{ \boldsymbol{x} \in \mathbb{R}^n \,\,: \, \bar{\rho}(\boldsymbol{x}) > 0 \}.
\]
Then $\bar{\rho}$ is a steady state and satisfies \eqref{station}. 
Moreover, for $\alpha \in (0,n) \cap [n-2,n)$, $\bar{\rho}$ coincides, up to translation 
and mass-preserving dilation,
with $\bar{\rho}_\mu$.
\end{theorem}

However, as will be shown in Section \ref{instability}, uniqueness is not essential for our approach; 
rather, the key ingredient is the scaling structure.
In particular, for any $\alpha \in (0,n)$, 
if there exists a unique steady state $\bar{\rho}_{\mu_0}$ 
satisfying \eqref{station} for any given $\mu_0 > 0$, 
then, by the scaling argument above, 
there exists a unique steady state $\bar{\rho}_\mu$ 
satisfying \eqref{station} for each $\mu > 0$, given by
\[
\bar{\rho}_\mu(\boldsymbol{x})
= \mu^\frac{1}{\gamma - 1}\bar{\rho}_1(\mu^\frac{2-\gamma}{(\gamma - 1)(n-\alpha)} \boldsymbol{x}).
\]
Let
\[  M_1 := \int_{\mathbb{R}^n} \bar{\rho}_1 \dd \boldsymbol{x},
    \]
denote the mass of the unique steady state for $\mu = 1$.
Then, for any $\mu>0$, the mass of $\bar{\rho}_\mu$ is given by
    \[ \int_{\mathbb{R}^n} \bar{\rho}_\mu \dd \boldsymbol{x} = \mu^\frac{n\gamma - (n + \alpha)}{(\gamma-1)(n-\alpha)} M_1.
    \]
Consequently, for any prescribed mass $M>0$, choosing
    \[ \mu_M := \Big(\frac{M}{M_1}\Big)^\frac{(\gamma - 1)(n - \alpha)}{n\gamma - (n + \alpha)} > 0,
    \]
 we obtain that $\bar{\rho}_{\mu_M}$ is the unique steady state with mass $M>0$. 
Thus, the correspondence between steady states parametrized by $\mu$ 
 and those with prescribed mass $M>0$ is a bijection. 
 Although the admissible sets and energy functionals considered in different formulations may differ,
 their respective minimization problems are expected to yield the same class of steady states; 
 see Sections \ref{instability}--\ref{stability}.

\subsubsection{Steady states in the general setting}
We define a steady state as follows:
\begin{definition}\label{definesteady}
$\bar{\rho} \in L^1_+(\mathbb{R}^n) \cap L^\infty(\mathbb{R}^n)$ is 
called a steady state of CEREs \eqref{1.1} 
if $p(\bar{\rho}) \in {W}^{1,1}_\text{\rm loc}(\mathbb{R}^n)$, $\mathcal{\tilde{W}}_\alpha=\Phi_\alpha\ast \bar{\rho}$ with
$\nabla \mathcal{\tilde{W}}_\alpha \in L^1_\text{\rm loc}(\mathbb{R}^n)$, and
\[
\nabla p(\bar{\rho}) + \kappa \bar{\rho} \nabla \mathcal{\tilde{W}}_\alpha = \boldsymbol{0}
\]
is satisfied in the sense of distributions{\rm ;} in addition, 
when $\alpha \in [n-1,n)$, 
$\bar{\rho} \in C^{0,\beta}(\mathbb{R}^n)$ is required for some $\beta \in (1 + \alpha - n,1)$.
\end{definition}

In particular, let $\Pi:[0,\infty) \to [0,\infty)$ be defined 
by $\Pi(\tau) = (\tau e(\tau))_{\tau}$. Then 
$$
\Pi'(\tau) = \frac{p'(\tau)}{\tau} > 0\qquad \mbox{for $\tau > 0$}, 
$$
so that $\Pi$ is strictly increasing and hence invertible. 
Extending $\Pi^{-1}$ to $\Pi^{-1}_+:\mathbb{R} \to [0,\infty)$ by setting it equal 
to zero on $(-\infty, 0)$,
we see that \eqref{station} is equivalent to
\begin{equation}\label{newlm}
 \bar{\rho} = \Pi_+^{-1}(-\Phi_\alpha * \bar{\rho} - \mu).
\end{equation}
In the case of a polytropic pressure law $p(\rho) = a_0 \rho^\gamma$, 
equation \eqref{newlm} reduces to 
\begin{equation*}
\bar{\rho} = \Big ( \frac{\gamma - 1}{a_0 \gamma} \Big )^\frac{1}{\gamma - 1} 
(- \Phi_\alpha * \bar{\rho}- \mu)_+^\frac{1}{\gamma - 1}.
\end{equation*}

\subsection{Growth of Solutions to the Attractive CEREs}
While we will later establish the stability of minimizers for both general and polytropic pressure
laws with $\gamma > \frac{n + \alpha}{n}$, we also show that solutions with positive initial energy 
exhibit growth of their support. 
In some cases, we derive a lower bound for the linear growth of the support. 
We assume that the gas occupies a bounded spatial domain $\Omega(t)\subset\mathbb{\R}^n$ 
at time $t$, with $\rho|_{\partial\Omega(t)}=0$ and $\rho>0$ in $\Omega(t)$. 
For steady solutions, the domain is time-independent, that is, 
$\Omega(t)=\Omega$. We define the radius of the domain $\Omega(t)$ by
\begin{equation}\label{2.5a}
 R(t)=\max_{\boldsymbol{x}\in\Omega(t)}\{|\boldsymbol{x}|\}.
\end{equation}
To generalize the behavior of polytropic pressure laws to more general pressure functions, 
we introduce the following exponent:
 \begin{equation}\label{gammainf}
  \gamma_{\inf} := 1 + \inf_{\varrho \in \mathbb{R}_+} \Big( \frac{p(\varrho)}{\varrho e(\varrho)} \Big).
\end{equation}
The main results concerning the growth of the support are summarized in the following theorem:

\begin{theorem}\label{growth}
Suppose that $(\rho,\mathcal{M})$ is a global classical solution of 
the attractive CEREs {\rm (}$\kappa = 1${\rm )} \eqref{1.1} with energy $E = E(\rho,\mathcal{M})>0$. 
Then the following hold{\rm :}
 \begin{enumerate}
 \item[{\rm(a)}] If $\alpha\in(0,2]\cap (0,n-1)$ and $\gamma_{\inf} \geq \frac{n + \alpha}{n}$, then
  $$
  \liminf_{t\rightarrow\infty}\Big(\frac{R(t)}{t}\Big)\geq\Big(\frac{\alpha E}{M}\Big)^{\frac{1}{2}};
  $$
 \item[{\rm(b)}] If $\alpha\in(2,n-1),$ $\gamma_{\inf} > \frac{n + 2}{n}$, 
 and $\lim_{\rho\to\infty}p(\rho)\rho^{-\frac{n+\alpha}{n}} = \infty$, then there exists 
 $\tilde{M} \in (0,\infty]$ such that, if $M < \tilde{M}$, 
 there is $\Lambda_{\min}(M) > 0$ so that
 $$
 \liminf_{t\rightarrow\infty}\Big(\frac{R(t)}{t}\Big)
 \geq\Big(\frac{ \Lambda_{\min}(M)}{M}\Big)^{\frac{1}{2}}.
 $$
 \end{enumerate}
 Moreover, if the pressure is given by a power law with 
 $\gamma > \frac{n+\alpha}{n}$, $\alpha\in(0,2]\cap (0,n-1)$, 
 and $E = E(\rho,\mathcal{M})=0$, 
 then there exists a sequence $(t_m)_m \subset [0,\infty)$ 
 such that $R(t_m) \to \infty$ as $m \to \infty$.
 \end{theorem}

 \begin{remark}
 The faster growth rate for larger $\alpha$, 
 as indicated in {\rm Theorem \ref{growth}(a)}, 
 is consistent with the weaker long-range attractive interaction 
 for more singular kernels.
 \end{remark}

\subsection{Instability of Steady States of the Attractive CEREs for \texorpdfstring{$\frac{2n}{2n - \alpha} < \gamma < \frac{n + \alpha}{n}$}{2n/(2n-\alpha) < \gamma < (n+\alpha)/n}}

In the polytropic case, when  $\frac{2n}{2n-\alpha} < \gamma < \frac{n + \alpha}{n}$, 
we establish the instability of steady states. 
The homogeneity of the pressure plays a crucial
role in the analysis here. The main result is stated in the following theorem:

\begin{theorem}\label{smallgamma}
Suppose that $(\rho(t),\mathcal{M}(t))$ is a classical solution of the attractive 
CEREs {\rm (}$\kappa = 1${\rm )} \eqref{1.1}, for $\alpha \in (0,n)$ and 
$\frac{2n}{2n - \alpha} < \gamma < \frac{n + \alpha}{n}$, 
with initial conditions $(\rho_0,\mathcal{M}_0)$ satisfying $Q(\rho_0) > 0$ and
    \begin{equation}\label{energymasscond}
  \int_{\mathbb{R}^n} \rho_0 \dd \boldsymbol{x} < \Big ( \frac{n + \alpha - n\gamma}{(2n-\alpha)\gamma - 2n} \Big )^\frac{n + \alpha - n\gamma}{(2n-\alpha)\gamma - 2n} \Big ( \ell_1 \frac{(2n-\alpha)\gamma - 2n}{(n - \alpha)(\gamma - 1)} \Big )^\frac{(n - \alpha)(\gamma - 1)}{(2n-\alpha)\gamma - 2n} E(\rho_0,\mathcal{M}_0)^\frac{n\gamma - n - \alpha}{(2n - \alpha)\gamma - 2n}.
    \end{equation}
Then there exists a sequence $t_m \to \infty$ as $m \to \infty$ such that
$$
R(t_m) \to \infty \qquad \mbox{as $m \to \infty$}.
$$
\end{theorem}

\smallskip
\subsection{Stability of Steady States for the Attractive CEREs}

We extend our previous result \cite[Theorem 2.6]{Carrillo2025} to the case of general pressure laws. 
In this context, we assume that both \eqref{lower} and \eqref{higherpress} are satisfied;
that is, the pressure exhibits the asymptotic behavior comparable to a polytropic power law with 
exponent $\gamma > \frac{n+\alpha}{n}$ when the density is near both zero and infinity. 
Our main result on the stability of steady states under general pressure laws is stated as
follows:

\begin{theorem}[Nonlinear Stability of Steady States]\label{ns}
Suppose that $\bar{\rho}$ is a unique minimizer {\rm(}up to translation{\rm)} 
of functional $\mathcal{G}$ in $X_M$. 
Suppose that $\alpha \in (0,n)$ and $p \in C^2([0,\infty))$ satisfying \eqref{Press}
and \eqref{lower}--\eqref{higherpress}. 
Let $(\rho,\mathcal{M})(t,\boldsymbol{x})$ be a global weak solution of 
the attractive CEREs {\rm (}$\kappa = 1${\rm )} in the sense of {\rm Definition {\ref{definition}}} for all $T>0$ with
$$\int_{\R^n} \rho_0(\boldsymbol{x})\,\dd \boldsymbol{x}=\int_{\R^n}  \bar{\rho}(\boldsymbol{x})\,\dd \boldsymbol{x}=M.$$
Then, for any $\v>0$, there exists $\delta>0$ such that, if
\begin{align*}
d(\rho_0,\bar{\rho})+ \| \rho_0 - \bar{\rho} \|_{L^\frac{2n}{2n - \alpha}}^2+\frac{1}{2}\int_{\R^n} \Big|\frac{\mathcal{M}_0}{\sqrt{\rho_0}}\Big|^2\,\dd \boldsymbol{x}<\delta,
\end{align*}
there exists a translation function $\boldsymbol{y}=\boldsymbol{y}(t)\in \R^n$ such that
\begin{align*}
d(\rho(t),T^{\boldsymbol{y}}\bar{\rho})+ \| \rho(t) - T^{\boldsymbol{y}}\bar{\rho} \|_{L^{\frac{2n}{2n - \alpha}}}^2+\frac{1}{2}\int_{\R^n} \Big|\frac{\mathcal{M}}{\sqrt{\rho}}\Big|^2\,\dd\boldsymbol{x}<\v
\end{align*}
for almost every $t>0$, where $T^{\boldsymbol{y}}\bar{\rho}(\boldsymbol{x})=\bar{\rho}(\boldsymbol{x}+\boldsymbol{y})$.
\end{theorem}

To quantify the stability of steady states, we employ 
a relative entropy type method to obtain a stability estimate. In this context, 
a pair of functions $(\eta(\rho,m),q(\rho,m))$ is called an entropy pair of the $1$-D 
Euler system if they satisfy
\begin{equation*}
\partial_t \eta(\rho(t,r),m(t,r)) + \partial_r q(\rho(t,r),m(t,r)) = 0
\end{equation*}
for any smooth solution $(\rho,m)$ of the Euler system. 
Moreover, $\eta$ is called a weak entropy if
\begin{equation*}
\eta|_{\rho = 0}=0 \qquad \text{ for all fixed $u = \frac{m}{\rho}$}.
\end{equation*}
These considerations naturally lead us to the notion of 
(weak) entropy solutions to \eqref{1.1}, which we define below.

\begin{definition}\label{globalweak}
    A finite-energy solution $(\rho,\mathcal{M})$, as defined in Definition \ref{definition} below,
    is called an entropy solution if
    \[
 \partial_t \eta(\rho,\mathcal{M}) + \nabla \cdot q(\rho,\mathcal{M}) \leq -\kappa \mathcal{M}\cdot\nabla \mathcal{W}_\alpha
    \]
    in the sense of distributions for the mechanical entropy
    \[ \eta(\rho, \mathcal{M})
  = \frac{|\mathcal{M}|^{2}}{2\rho} + h(\rho),
  \qquad
  q(\rho,\mathcal{M})
  =\frac12 \mathcal{M}\frac{|\mathcal{M}|^{2}}{\rho^{2}}+\mathcal{M}h'(\rho),
\]
where $h(\rho) = \rho e(\rho)$.
\end{definition}
We will show a growth estimate on the relative entropy
\begin{equation*}
    \mathcal{E}(t) := \int_{\mathbb{R}^n} \eta(\rho,\mathcal{M} \,|\, \bar{\rho}) \dd \boldsymbol{x},
\end{equation*}
where 
\[
 \eta(\rho,\mathcal{M} \,|\, \bar{\rho}) := \eta(\rho,\mathcal{M})-\eta(\bar\rho,\boldsymbol{0})
 -\partial_{\rho}\eta(\bar\rho,\boldsymbol{0})(\rho-\bar\rho)
 -\nabla_{\mathcal{M}}\eta(\bar\rho,\boldsymbol{0})\cdot \mathcal{M}.
\] 

Recall that 
$\bar{\mathcal W}_\alpha=\Phi_\alpha*\bar\rho$, there is a constant $\mu>0$ such that
\[ h'(\bar\rho)+\bar{\mathcal W}_\alpha=-\mu \quad\hbox{on }\{\bar\rho>0\},
\qquad
h'(0)+\bar{\mathcal W}_\alpha\ge -\mu
 \quad\hbox{on }\{\bar\rho=0\}.
\]
Denote that, for polytropic pressure laws, 
\[
 G_\mu(\boldsymbol{x}):=\mu+\bar{\mathcal W}_\alpha(\boldsymbol{x}),
 \qquad
 h(\rho)=\frac{a_0}{\gamma-1}\rho^\gamma,
\]
and define the effective relative entropy
\[
\mathcal E_\mu(t):= \int_{\mathbb R^n}\Big( \frac{|\mathcal M|^2}{2\rho} +h(\rho)-h(\bar\rho)+G_\mu(\rho-\bar\rho)\Big)(t,\boldsymbol{x})\,\dd\boldsymbol{x}.
\]

We remark that this is a modified relative entropy adapted to the
Euler--Lagrange equation for the steady state. To see this, we define
\[
    H(t) := \int_{\mathbb R^n}\big( h(\rho)-h(\bar\rho)-h'(\bar\rho)(\rho-\bar\rho)\big)(t,\boldsymbol{x})\,\dd\boldsymbol{x}, \qquad  H_\mu(t) := \int_{\mathbb R^n}\big( h(\rho)-h(\bar\rho)+G_\mu(\rho-\bar\rho)\big)(t,\boldsymbol{x})\,\dd\boldsymbol{x},
\]
such that
\[
    \mathcal{E}(t) = \int_{\mathbb R^n} \frac{|\mathcal M|^2}{2\rho}(t,\boldsymbol{x})\,\dd\boldsymbol{x} + H(t), \qquad \mathcal{E}_\mu(t) = \int_{\mathbb R^n} \frac{|\mathcal M|^2}{2\rho}(t,\boldsymbol{x})\,\dd\boldsymbol{x} + H_\mu(t).
\]
Indeed,  $G_\mu=-h'(\bar\rho)$ on $\Gamma=\{\bar\rho>0\}$, and therefore
$h_\mu(\rho\mid\bar\rho)$ coincides with the standard relative
internal energy.  On the vacuum region $\Gamma^{\rm c}$, however,
$G_\mu\ge0=-h'(0)$, so that
\[
 h_\mu(\rho\mid0)=h(\rho)+G_\mu\rho .
\]
Thus, the modified relative entropy contains an additional nonnegative
term in the vacuum region, which is essential for controlling the density
perturbations outside the support of the steady state. More specifically, we see that
\begin{align*}
    H_\mu(t) & = \int_{\Gamma}\big( h(\rho)-h(\bar\rho)+G_\mu(\rho-\bar\rho)\big)(t,\boldsymbol{x})\,\dd\boldsymbol{x} + \int_{\Gamma^{\rm c}}\big( h(\rho)+G_\mu\rho\big)(t,\boldsymbol{x})\,\dd\boldsymbol{x} \\
    & \geq \int_{\Gamma}\big( h(\rho)-h(\bar\rho)-h'(\bar\rho)(\rho-\bar\rho)\big)(t,\boldsymbol{x})\,\dd\boldsymbol{x} + \int_{\Gamma^{\rm c}} h(\rho)(t,\boldsymbol{x})\,\dd\boldsymbol{x} \\
    & = H(t).
\end{align*}
Hence, $0 \leq \mathcal{E}(t) \leq \mathcal{E}_\mu(t)$, meaning that not only is the effective relative entropy nonnegative, but it controls the relative entropy, which in turn has a control on the difference of $\rho(t)$ and $\bar\rho$. In addition, we can prove the following finite-time relative entropy lemma:

\begin{theorem}[Weighted $L^2$--Stability Around Steady States]\label{thm:stability}
Let $\alpha\in(0,n-1)$, and let $(\bar\rho,\boldsymbol{0})$ be a compactly supported polytropic steady state characterized by \eqref{station}.
Let $(\rho,\mathcal M)$ be a global entropy solution of 
the attractive CEREs {\rm(}$\kappa=1${\rm)} \eqref{1.1} as in
{\rm Definition \ref{globalweak}}, and let
\[
\max \Big\{\frac{n + \alpha}{n},\frac{3n}{3n - 2(\alpha + 1)} \Big\}<\gamma\le2.
\]
Suppose that $\mathcal E_\mu(0)<\infty$ and
\[
(\rho^{\gamma + \theta}, \,\frac{|\mathcal{M}|^{3}}{\rho^{2}})
\in L^1(0,T;L^1(\mathbb{R}^n))
\qquad \mbox{for $\theta = \frac{\gamma - 1}{2}$}.
\]
Then, for every fixed $T>0$, there exists a constant
$C_T=C_T(\mathcal E_\mu(0),\bar{\rho},T)>0$ such that
\begin{equation}\label{finitetimerelative}
 \mathcal E_\mu(t)
 \le C_T\mathcal E_\mu(0)
 \qquad \hbox{for {\it a.e.} }t\in[0,T].
\end{equation}
\end{theorem}

\medskip
\begin{remark}
The condition $\gamma > \frac{3n}{3n - 2(\alpha + 1)}$ is also required in 
 {\rm\cite{Carrillo2025}} in order to prove the existence of global spherically symmetric energy decreasing
 solutions to the attractive CEREs.
\end{remark}

\begin{remark}
An important earlier use of the relative entropy as an $L^2$--type stability
functional for multidimensional compressible flow is in {\rm Chen-Chen \cite{ChenChen2007}}. They employed this idea to study the stability
of rarefaction waves and vacuum states for the multidimensional Euler
equations in a broad class of bounded entropy solutions.
\end{remark}

Denote $$f:=\rho-\bar\rho.$$ We note that the relative energy can be written as 
\begin{equation*}
E(\rho,\mathcal{M}\, | \, \bar{\rho}, \boldsymbol{0})(t)=\mathcal E_\mu(t)+\mathcal Q_\alpha(t),
\end{equation*}
where $$ \mathcal Q_\alpha(t)
 :=\frac12\int_{\mathbb R^n} f(t,\boldsymbol{x})\,\Phi_\alpha*f(t,\boldsymbol{x})\,\dd\boldsymbol{x}.$$
 This can be seen by using the fact that the mass of $\rho(t)$ and $\bar\rho$ are the same and the following computation:
 \begin{align*}
    \mathcal{E}_\mu(t)+\mathcal Q_\alpha(t) & = \int_{\mathbb R^n} \frac{|\mathcal M|^2}{2\rho}(t,\boldsymbol{x})\,\dd\boldsymbol{x} + \int_{\mathbb R^n}\big(h(\rho)-h(\bar\rho)+(\mu + \Phi_\alpha * \bar\rho)(\rho-\bar\rho)\big)(t,\boldsymbol{x})\,\dd\boldsymbol{x} \\
    & \qquad + \frac12\int_{\mathbb R^n} f(t,\boldsymbol{x})\,\Phi_\alpha*f(t,\boldsymbol{x})\,\dd\boldsymbol{x} \\
    & = \int_{\mathbb R^n} \frac{|\mathcal M|^2}{2\rho}(t,\boldsymbol{x})\,\dd\boldsymbol{x} + \int_{\mathbb R^n}\big(h(\rho)-h(\bar\rho)+(\rho-\bar\rho)\Phi_\alpha * \bar\rho\big)(t,\boldsymbol{x})\,\dd\boldsymbol{x} \\
    & \qquad + \frac{1}{2} \int_{\mathbb R^n} \big (\rho\,\Phi_\alpha*\rho - 2 \rho\,\Phi_\alpha*\bar \rho + \bar\rho\,\Phi_\alpha*\bar\rho \big ) \,\dd\boldsymbol{x} \\
    & = \int_{\mathbb R^n} \Big ( \frac{|\mathcal M|^2}{2\rho} + h(\rho)+\frac{1}{2}\rho\,\Phi_\alpha * \rho\Big)(t,\boldsymbol{x})\,\dd\boldsymbol{x} - \int_{\mathbb R^n} \Big ( h(\bar\rho)+\frac{1}{2}\bar\rho\,\Phi_\alpha * \bar\rho\Big)(\boldsymbol{x})\,\dd\boldsymbol{x} \\[1mm]
    & = E(\rho,\mathcal{M}\, | \, \bar{\rho}, \boldsymbol{0})(t).
 \end{align*}

Thus, owing to the fact that the solutions of \eqref{1.1} are non-increasing in energy, we have the following corollary:
\begin{corollary}[Weighted $L^2$--Stability Around Steady States II]
Let $(\rho,\mathcal M)$ be a global entropy solution of CEREs {\rm(}$\kappa=1${\rm)} \eqref{1.1} as in
{\rm Definition \ref{globalweak}}, and let
\[
\max \Big\{\frac{n + \alpha}{n},\frac{3n}{3n - 2(\alpha + 1)} \Big\}<\gamma\le2.
\]
Let $(\bar\rho,\boldsymbol{0})$ be a compactly supported polytropic steady state satisfying
\eqref{station} with mass given by the initial mass of $\rho_0$, and let $(\rho,\mathcal M)$ be a finite-energy entropy solution of
\eqref{1.1} with $\kappa=1$ in the sense of {\rm Definition \ref{globalweak}} such that
\begin{equation*}
 \rho^{\gamma+\theta}+\frac{|\mathcal M|^3}{\rho^2}
 \in L^1((0,T)\times\mathbb R^n).
\end{equation*}
If $E(\rho_0,\mathcal{M}_0\, | \, \bar{\rho}, \boldsymbol{0})<\infty$, then the relative-energy 
dissipation inequality holds{\rm :}
\begin{equation*}
E(\rho,\mathcal{M}\, | \, \bar{\rho}, \boldsymbol{0})(t) \le E(\rho_0,\mathcal{M}_0\, | \, \bar{\rho}, \boldsymbol{0})
 \qquad\text{for a.e. }t\in(0,T).
\end{equation*}
\end{corollary}

\begin{remark}
    While it is proven in our previous work {\rm \cite[Section 5]{Carrillo2025}} that the relative energy is non-increasing in time, it is still unknown whether the relative functional
    \[
    H_\mu(t) +\mathcal Q_\alpha(t) = \int_{\mathbb R^n} \Big ( h(\rho)+\frac{1}{2}\rho\,\Phi_\alpha * \rho\Big)(t,\boldsymbol{x})\,\dd\boldsymbol{x} - \int_{\mathbb R^n} \Big ( h(\bar\rho)+\frac{1}{2}\bar\rho\,\Phi_\alpha * \bar\rho\Big)(\boldsymbol{x})\,\dd\boldsymbol{x},
    \]
    can be used to control the distance between $\rho(t)$ and $\bar\rho$ in some suitable norms{\rm ;} 
    see {\rm\cite{YY20}} for related functional inequalities.
\end{remark}

\begin{remark}
The methods and results developed in this paper are expected to extend to general pressure laws and to CEREs in rotationally symmetric settings, including the rotating-star configurations{\rm ;} these extensions will be addressed elsewhere.
\end{remark}

\subsection{Weak Solutions of CEREs}

We now define global finite-energy solutions of \eqref{1.1}--\eqref{1.1-2}.

For the existence of finite-energy solutions, we consider a general pressure of the form:
\begin{enumerate}
    \item[(a)] The pressure function $p \in C^4(\mathbb{R}_+)$.
    \item[(b)] There exists a constant $\rho_* > 0$ such that
    \begin{align}\label{1.2}
  p(\rho) = \kappa_1 \rho^{\gamma_1}(1 + \mathcal{P}_1(\rho)) \qquad \text{ for } \rho \in [0,\rho_*),
    \end{align}
    with $\gamma_1 \in (\frac{2n}{2n - \alpha},3)$, $\kappa_1 > 0$, and a function $\mathcal{P}_1 \in C^4(\mathbb{R}_+)$ satisfying
    \begin{align}\label{1.3}
 |\mathcal{P}^{(j)}_1(\rho)| \leq C_\#\rho^{\gamma_1 - 1 - j},
    \end{align}
    for $\rho \in (0,\rho_*)$ and $j \in \{0,\dots,4\}$, where $C_\# > 0$ is a constant depending only on $\rho_*$.
    \item[(c)] There exists a constant $\rho^* > \rho_* > 0$ such that
    \begin{align}\label{1.4}
  p(\rho) = \kappa_2 \rho^{\gamma_2}(1 + \mathcal{P}_2(\rho)) \qquad \text{ for } \rho \in [\rho^*,\infty),
    \end{align}
    for $\gamma_2 \in (\frac{2n}{2n - \alpha},\gamma_1]$, $\kappa_2 >0$, and a function $\mathcal{P}_2 \in C^4(\mathbb{R}_+)$ satisfying
    \begin{align}\label{1.5}
 |\mathcal{P}^{(j)}_2(\rho)| \leq C^\#\rho^{- \varepsilon - j},
    \end{align}
    for $\rho \in [\rho^*,\infty)$, $\varepsilon > 0$ (we can take $\varepsilon < \gamma_2 - 1$ for simplicity), and $j \in \{0,\dots,4\}$, where $C^\# > 0$ is a constant depending only on $\rho^*$.
\end{enumerate}

\smallskip
Given a solution $(\rho,\mathcal{M})$ of \eqref{1.1}, the energy of the solution is given by
\begin{equation}\label{energies}
E(\rho,\mathcal{M})(t)
:= \int_{\mathbb{R}^n}\Big((\rho e(\rho))(t,\boldsymbol{x})
+ \frac{1}{2} \Big|\frac{\mathcal{M}}{\sqrt{\rho}}\Big|^2(t,\boldsymbol{x}) + \frac{\kappa}{2}\rho(t,\boldsymbol{x})(\Phi_\alpha \ast \rho)(t,\boldsymbol{x})\Big)\,\dd\boldsymbol{x}.
\end{equation}

We first assume that the given initial data function $(\rho_0,\M_0)(\boldsymbol{x})$ has finite total energy and
total mass:
\begin{align}
&E_0 := \int_{\mathbb{R}^n}\rho_0 \Big( \frac{1}{2} \Big| \frac{\M_0}{\rho_0}\Big|^2
+ e(\rho_0) + \frac{\kappa}{2} (\Phi_\alpha * \rho_0) \Big)(\boldsymbol{x})
\dd \boldsymbol{x} < \infty,\label{0.7}\\[2mm]
&M := \int_{\mathbb{R}^n} \rho_0(\boldsymbol{x}) \dd \boldsymbol{x} = \omega_n \int^\infty_0 \rho_0(r) r^{n-1} \dd r < \infty.\nonumber
\end{align}
When
$\alpha \in (-1,0]$, we also assume that $(\rho_0,\M_0)(\boldsymbol{x})$ has
finite $(-\alpha)$-moment:
\begin{equation}\label{0.81}
\int_{\mathbb{R}^n} \rho_0(\boldsymbol{x}) k_\alpha (1 + |\boldsymbol{x}|^2) \dd \boldsymbol{x} = \omega_n \int^\infty_0 \rho_0(r) k_\alpha (1 + r^2) r^{n-1} \dd r < \infty,
\end{equation}
where
\begin{equation*}
    k_\alpha(r) = \begin{cases}
  r^{- \frac{\alpha}{2}} \quad &\mbox{for $\alpha \in (-1,0)$}, \\
  \log(r) &\mbox{for $\alpha = 0$},
    \end{cases}
\end{equation*}
and $\omega_n$ represents the surface area of the unit sphere in $\mathbb{R}^n$.
Here we require the initial data so that the convolution $\Phi_\alpha*\rho_0$ is well-defined and differentiable. 

\begin{definition}\label{definition}
A measurable vector-valued function $(\rho,\M)(t, \boldsymbol{x})$ is a finite-energy weak solution 
of the Cauchy problem \eqref{1.1}--\eqref{1.1-2} if
$(\rho,\M)(t, \boldsymbol{x})$ satisfies the following conditions for all $T>0${\rm :}
\begin{enumerate}
\item[(a)] $\rho(t,\boldsymbol{x}) \geq 0$ {\it a.e.} and $(\M, \frac{\M}{\sqrt{\rho}})(t,\boldsymbol{x}) = \boldsymbol{0}$ {\it a.e.} on the vacuum states $\{ (t,\boldsymbol{x}) \, : \, \rho(t,\boldsymbol{x}) = 0 \}$,
\begin{align*}
\rho\in L^{\infty}(0,T; L^1(\mathbb{R}^n)), \quad \,\, 
\rho e(\rho)\in L^{\infty}(0,T; L^1(\mathbb{R}^n)), \quad\,\,
 \frac{\M}{\sqrt{\rho}}\in L^{\infty}(0,T; L^2(\mathbb{R}^n)).
    \end{align*}
 Moreover, for $\alpha \in (0,n)$,
    \[
\Phi_\alpha * \rho\in L^\infty(0,T;L^{\frac{2n}{\alpha}}(\mathbb{R}^n) ),\quad
\nabla\Phi_\alpha * \rho\in L^\infty(0,T;L^{\frac{2n}{\alpha + 2}}(\mathbb{R}^n)),
  \]
 with, for $\alpha \in [n-1,n)$,
 \[
 \rho \in L^\infty(0,T;C^{0,\beta}(\mathbb{R}^n)),
  \]
    for some $\beta \in (1 + \alpha - n,1)$; and, for $\alpha \in (-1,0]$,
    \[
 \Phi_\alpha * \rho\in L^\infty(0,T;L^\infty_{\rm loc}(\mathbb{R}^n) ),\quad
\nabla\Phi_\alpha * \rho\in L^\infty(0,T;L_{\rm loc}^{\frac{2n}{\alpha + 2}}(\mathbb{R}^n)).
 \]
\item[(b)] For {\it a.e.} $t>0$, the total energy is finite{\rm :}

\smallskip
\begin{itemize}
 \item For $\kappa = -1$,
\begin{equation*}
 \int_{\mathbb{R}^n} \Big( \frac{1}{2} \Big| \frac{\M}{\sqrt{\rho}}\Big|^2
 + \rho e(\rho) - \frac{1}{2} \rho (\Phi_\alpha * \rho) \Big)(t,\boldsymbol{x})
 \dd \boldsymbol{x} \leq E_0.
\end{equation*}
\item For $\kappa = 1$,
\begin{align*}
\begin{cases}
\displaystyle
\int_{\mathbb{R}^n} \Big( \frac{1}{2} \Big| \frac{\M}{\sqrt{\rho}}\Big|^2
+ \rho e(\rho) - \frac{1}{2} \rho (\Phi_\alpha * \rho) \Big)(t,\boldsymbol{x})
\dd \boldsymbol{x} \leq C(E_0,M), \\[4mm]
\displaystyle
\int_{\mathbb{R}^n} \Big( \frac{1}{2}\Big|\frac{\M}{\sqrt{\rho}}\Big|^2
+ \rho e(\rho) + \frac{1}{2} \rho (\Phi_\alpha * \rho) \Big)(t,\boldsymbol{x})
\dd \boldsymbol{x} \leq E_0.
\end{cases}
\end{align*}
\end{itemize}
\item[(c)] For any $\zeta \in C^1_0(\mathbb{R}_+ \times \mathbb{R}^n)$,
\begin{equation*}
\int_{\mathbb{R}_+} \int_{\mathbb{R}^n} \big(\rho \zeta_t
+ \M \cdot \nabla \zeta\big) \dd \boldsymbol{x} \dd t
+ \int_{\mathbb{R}^n} \rho_0(\boldsymbol{x})\,\zeta(0,\boldsymbol{x}) \dd \boldsymbol{x} = 0.
\end{equation*}
\item[(d)] For any $\boldsymbol{\psi} \in (C^1_0(\mathbb{R}_+ \times \mathbb{R}^n))^n$,
    \begin{align*}
 & \int_{\mathbb{R}_+} \int_{\mathbb{R}^n}
  \Big(\M \cdot \boldsymbol{\psi}_t
 + \frac{\M}{\sqrt{\rho}} \cdot \big( \frac{\M}{\sqrt{\rho}} \cdot \nabla \big)
 \boldsymbol{\psi} + p(\rho)\nabla \cdot \boldsymbol{\psi} \Big)(t,\boldsymbol{x})
 \dd \boldsymbol{x} \dd t
  + \int_{\mathbb{R}^n} \M_0(\boldsymbol{x}) \cdot \boldsymbol{\psi}(0,\boldsymbol{x}) \dd \boldsymbol{x} \\
  & = \int_{\mathbb{R}_+} \int_{\mathbb{R}^n} \kappa(\rho \nabla \mathcal{W}_\alpha \cdot \boldsymbol{\psi})(t,\boldsymbol{x}) \dd \boldsymbol{x} \dd t.
    \end{align*}
\end{enumerate}
\end{definition}

Furthermore, defining constant $B_\beta$ by
    \begin{align*}
  B_\beta := \frac{C_{\alpha,n}}{2 \alpha} \omega_n^\frac{\alpha - n(\gamma_2 - 1)}{n(\gamma_2 - 1)} (C_{\text{max}}(\beta) )^{\frac{(2n-\alpha)\gamma_2 - 2n}{n(\gamma_2 - 1)}}
    \end{align*}
    for the sharp HLS constant $C_{\alpha,n} > 0$ given in Lemma \ref{simplehls}, 
    the $\beta$--dependent critical mass $M_{\rm c}^{\alpha}(\beta) > 0$ is given 
    by the unique solution of
    \begin{align*}
  \begin{cases}
  \displaystyle
  M_{\rm c}^{\alpha}(\beta) = B_\beta^{-\frac{n}{n - \alpha}} 
 & \mbox{for $\gamma_2 = \frac{n + \alpha}{n}$}, \\
\displaystyle
\frac{\alpha - n(\gamma_2 - 1)}{\alpha} 
\left ( \frac{\alpha B_\beta M_{\rm c}^{\alpha}(\beta)}{n(\gamma_2 -1)} 
 \right)^{-\frac{n(\gamma_2 - 1)}{\alpha - n(\gamma_2 -1)}}
 -  \omega_n^{-1} \beta M_{\rm c}^{\alpha}(\beta) = E_0 
 \qquad & \mbox{for $\gamma_2 \in ( \frac{2n}{2n - \alpha}, \frac{n + \alpha}{n} )$},
  \end{cases}
    \end{align*}
    and the critical mass $M_{\rm c}^{\alpha}$ by
    \begin{align*}
  M_{\rm c}^{\alpha} = \sup_{\beta > 0} M_{\rm c}^{\alpha}(\beta).
    \end{align*}
    Then, as outlined in Section \ref{existy}, 
    we can prove the existence of finite-energy solutions. 

\begin{theorem}[Existence of Spherically Symmetric Solutions of CEREs]\label{existence}
Consider problem \eqref{1.1}--\eqref{1.1-2} with the pressure $p$ satisfying \eqref{Press} and \eqref{1.2}--\eqref{1.5}, $\alpha \in (-1,n-1)$,
and the initial data of spherical symmetry. Assume that $(\rho_0,\M_0)$ satisfy \eqref{0.7}--\eqref{0.81}, $\gamma_2 > \frac{1}{n-1-\alpha}$, and
\begin{itemize}
    \item When $\kappa = -1$,  $\,\,\rho_0 \in L^{\frac{2n}{2n - \alpha}}(\mathbb{R}^n)$ and
    \begin{itemize}
  \item [{\rm(a)}] $\gamma_2 > 1$ for $\alpha \in (-1,n-2]$,
\item[{\rm(b)}] $\gamma_2 \geq \frac{3n}{3n - 2(1+\alpha)}$ for $\alpha \in (n-2,n-1)${\rm ;}
    \end{itemize}
    \item When $\kappa = 1$, $\,\,\gamma_2 \geq \frac{3n}{3n - 2(1+\alpha)}$
when $\alpha \neq n-2$ and    
    \begin{itemize}
 \item[{\rm(c)}]  $\gamma_2 > 1$ for $\alpha \in (-1,0]$,
\item [{\rm(d)}] $\gamma_2 > \frac{n + \alpha}{n}$ for $\alpha \in (0,n-1)$,
\item[{\rm(e)}] $M < M_{\rm c}^\alpha$ for
 $\gamma_2 \in (\frac{2n}{2n - \alpha}, \frac{n + \alpha}{n}]$ and $\alpha \in (0,n-1)$.
    \end{itemize}
\end{itemize}
Then there exists a global spherically symmetric finite-energy solution
$(\rho,\M)(t,\boldsymbol{x})$ of problem \eqref{1.1}--\eqref{1.1-2}
in the sense of {\rm Definition \ref{definition}}.
\end{theorem}

With this theorem and the previous stability result of Theorem \ref{ns}, we have the following corollary:

\begin{corollary}
Suppose that the conditions of {\rm Theorem \ref{existence}} hold, 
with $\gamma_2 \geq \frac{n+\alpha}{n}$, such that 
both \eqref{lower} and \eqref{higherpress} hold. 
Then, for any $\v>0,$ there exists $\delta>0$ such that, 
if the initial data $(\rho_0,\mathcal{M}_0)$ satisfy
\begin{align*}
d(\rho_0,\bar{\rho})+ \| \rho_0 - \bar{\rho} \|_{L^\frac{2n}{2n - \alpha}}^2+\frac{1}{2}\int_{\R^n} \Big|\frac{\mathcal{M}_0}{\sqrt{\rho_0}}\Big|^2\,\dd \boldsymbol{x}<\delta,
\end{align*}
then there exists a finite-energy solution $(\rho(t),\mathcal{M}(t))$ of the attractive 
CEREs {\rm (}$\kappa = 1${\rm )} \eqref{1.1} corresponding to $(\rho_0,\mathcal{M}_0)$ in the sense of {\rm Definition \ref{definition}} 
and $\boldsymbol{y}=\boldsymbol{y}(t)\in \R^n$ such that
\begin{align*}
d(\rho(t),T^{\boldsymbol{y}}\bar{\rho})+ \| \rho(t) - T^{\boldsymbol{y}}\bar{\rho} \|_{L^{\frac{2n}{2n - \alpha}}}^2+\frac{1}{2}\int_{\R^n} \Big|\frac{\mathcal{M}}{\sqrt{\rho}}\Big|^2\,\dd\boldsymbol{x}<\v,
\end{align*}
for almost every $t>0$, where $T^{\boldsymbol{y}}\bar{\rho}(\boldsymbol{x})=\bar{\rho}(\boldsymbol{x}+\boldsymbol{y})$.
\end{corollary}

\section{Growth of Solutions for General and Polytropic Pressures with $\gamma \geq \frac{n + \alpha}{n}$}\label{section3}

Throughout this section, we assume that $(\rho,\mathcal{M})$ is a classical solution
of CEREs with compactly supported density
and sufficiently regular free boundary, so that all integrations by parts are justified.
In the polytropic case, as shown in \cite{Carrillo2025}, we obtain a form of
local Lyapunov stability of steady states when $\gamma > \frac{n + \alpha}{n}$.
On the other hand, the question of global stability has remained open.
In fact, for a solution with positive initial energy, the radius of the support of 
the solution grows linearly in time.

However, in the critical case $\gamma=\frac{n+\alpha}{n}$, 
there exists a unique steady state (up to translation and mass-preserving dilation) 
corresponding to 
a critical mass $M=M_{\rm c}$; see \cite[Theorem 3]{Calvez_2021}. 
In this case, the total energy of the steady state is zero. 
Consequently, any small perturbation $(\rho_0,\mathcal{M}_0)$ of the steady state $(\bar{\rho},\boldsymbol{0})$ has positive energy, unless it is merely a translation 
of the steady state. 
Since the steady state has compact support, 
the linear growth of the radius of the support of the solution implies 
the instability of the steady state. 

To extend this argument to more general pressure laws, we use 
$\gamma_{\inf}$ defined in \eqref{gammainf}.
In the polytropic case, if $\gamma \geq \frac{n + \alpha}{n}$, 
then $\gamma_{\inf} = \gamma \geq \frac{n + \alpha}{n}$. 
For global classical solutions $(\rho,\mathcal{M})$ of \eqref{1.1} with compact support, 
the total energy $E(\rho(t),\mathcal{M}(t))$ is conserved for all $t\geq 0$. 
Accordingly, for the rest of this section, we denote 
$E := E(\rho_0,\mathcal{M}_0)$. 

Different from \cite{Deng_2002}, where the analysis is carried out for polytropic CEPEs, 
we can extend their approach to the case of general pressure laws and the nonlocal Riesz potential. 
We assume that the gas occupies a bounded spatial domain $\Omega(t)\subset\mathbb{\R}^n$ at time $t$, 
with $\rho=0$ on $\partial\Omega(t)$ and $\rho>0$ in $\Omega(t)$.  
For steady solutions, the domain is time-independent, {\it i.e.}, $\Omega(t)=\Omega$.

\begin{lemma}\label{E>0,general}
Let $\alpha\in(0,2]\cap (0,n-1)$ and $\gamma_{\inf} \geq \frac{n + \alpha}{n}$, and 
let $(\rho,\mathcal{M})$ be a global classical solution of the attractive 
CEREs {\rm (}$\kappa = 1${\rm )} \eqref{1.1}. 
If the total energy $E>0,$ then
$$\liminf_{t\rightarrow\infty}\Big(\frac{R(t)}{t}\Big)\geq\Big(\frac{\alpha E}{M}\Big)^{\frac{1}{2}},$$
with $R(t)=\max_{\boldsymbol{x}\in\Omega(t)}\{|\boldsymbol{x}|\}$ as defined in \eqref{2.5a}.
\end{lemma}

\begin{proof}
    Taking derivatives of the second moment of inertia:
\begin{equation}\label{3.1a}
H(t)=\int_{\Omega(t)}\rho|\boldsymbol{x}|^2\dd \boldsymbol{x}.
\end{equation}
Then, using $\eqref{1.1}_1$, integration by parts, 
and $\rho |_{\partial \Omega(t)} \equiv 0$, we have
 $$H'(t)= \int_{\partial \Omega(t)} |\boldsymbol{x}|^2 \mathcal{M} \cdot \nu \dd S(\boldsymbol{x}) + \int_{\Omega(t)}\rho_t|\boldsymbol{x}|^2\dd \boldsymbol{x} = 2\int_{\Omega(t)}\mathcal{M} \cdot \boldsymbol{x}\dd \boldsymbol{x},$$
where $\nu(\boldsymbol{x})$ is the normal to surface $\partial \Omega(t)$ 
at $\boldsymbol{x} \in \partial \Omega(t)$. 
Similarly, using $\eqref{1.1}_2$, we obtain
 \begin{equation}\label{2.2}
 \begin{split}
 H''(t)& =2\int_{\partial \Omega(t)} \Big( \frac{\mathcal{M} \otimes \mathcal{M}}{\rho} : \boldsymbol{x}\otimes \nu \Big) \dd S(\boldsymbol{x}) + 2\int_{\Omega(t)}\mathcal{M}_t \cdot \boldsymbol{x}\dd \boldsymbol{x} \\
 & = 2\Big(\int_{\Omega(t)}\Big| \frac{\mathcal{M}}{\sqrt{\rho}} \Big|^2\dd \boldsymbol{x}
 +n\int_{\Omega(t)}p\dd \boldsymbol{x}-\int_{\Omega(t)}\rho\nabla\Phi_\alpha\ast\rho\cdot \boldsymbol{x}\dd \boldsymbol{x}\Big).
 \end{split}
 \end{equation}
A simple calculation shows 
 \begin{align*}\nonumber
 I&:=\int_{\mathbb{R}^n}\rho\,\nabla\Phi_\alpha\ast\rho\cdot \boldsymbol{x}\,\dd \boldsymbol{x}\\
&=\int_{\mathbb{R}^n}\rho(\boldsymbol{x})\int_{\mathbb{R}^n}
\frac{\rho(\boldsymbol{y})(\boldsymbol{x}-\boldsymbol{y})\cdot \boldsymbol{x}}
{|\boldsymbol{x}-\boldsymbol{y}|^{\alpha+2}}
\,\dd \boldsymbol{y}\,\dd \boldsymbol{x}\\
&=\int_{\mathbb{R}^n}\rho(\boldsymbol{x})\int_{\mathbb{R}^n}
\frac{\rho(\boldsymbol{y})|\boldsymbol{x}-\boldsymbol{y}|^2}
{|\boldsymbol{x}-\boldsymbol{y}|^{\alpha+2}}
\,\dd \boldsymbol{y}\,\dd \boldsymbol{x}
+\int_{\mathbb{R}^n}\rho(\boldsymbol{x})\int_{\mathbb{R}^n}
\frac{\rho(\boldsymbol{y})(\boldsymbol{x}-\boldsymbol{y})\cdot \boldsymbol{y}}
{|\boldsymbol{x}-\boldsymbol{y}|^{\alpha+2}}
\,\dd \boldsymbol{y}\,\dd \boldsymbol{x}\\
&=-\alpha\int_{\mathbb{R}^n}\rho\,\Phi_\alpha\ast\rho\,\dd \boldsymbol{x}
+\int_{\mathbb{R}^n}\rho(\boldsymbol{y})\int_{\mathbb{R}^n}
\frac{\rho(\boldsymbol{x})(\boldsymbol{y}-\boldsymbol{x})\cdot \boldsymbol{x}}
{|\boldsymbol{x}-\boldsymbol{y}|^{\alpha+2}}
\,\dd \boldsymbol{x}\,\dd \boldsymbol{y}\\
&=-\alpha\int_{\mathbb{R}^n}\rho\,\Phi_\alpha\ast\rho\,\dd \boldsymbol{x}-I.
 \end{align*}
 Then we have 
 $$
 I=-\frac{\alpha}{2}\int_{\Omega(t)}\rho\Phi_\alpha\ast\rho\dd \boldsymbol{x}.
 $$
 Since
 $$
 E=\int_{\mathbb{R}^n}\Big( \frac{1}{2} \Big|\frac{\mathcal{M}}{\sqrt{\rho}}\Big|^2(t,\boldsymbol{x})
 + (\rho e(\rho))(t,\boldsymbol{x}) + \frac{1}{2}\rho(t,\boldsymbol{x})(\Phi_\alpha \ast \rho)(t,\boldsymbol{x})\Big)\,\dd\boldsymbol{x},
 $$
 then it follows from \eqref{2.2} that
 \begin{align*}
 \begin{split}
 H''(t)& =2\int_{\Omega(t)}\Big| \frac{\mathcal{M}}{\sqrt{\rho}} \Big|^2\dd \boldsymbol{x}
 +2n\int_{\Omega(t)}p\dd \boldsymbol{x}+2\alpha\Big(E-\int_{\Omega(t)}\big(\frac{1}{2}\Big | \frac{\mathcal{M}}{\sqrt{\rho}} \Big |^2+\rho e(\rho)\big)\dd \boldsymbol{x}\Big) \\
 & = 2\alpha E + (2-\alpha)\int_{\Omega(t)}\Big| \frac{\mathcal{M}}{\sqrt{\rho}} \Big|^2\dd \boldsymbol{x} 
 + 2\int_{\Omega(t)}(np - \alpha \rho e(\rho))\dd \boldsymbol{x} \\
 & \geq 2\alpha E + (2-\alpha)\int_{\Omega(t)}\Big| \frac{\mathcal{M}}{\sqrt{\rho}} \Big|^2\dd \boldsymbol{x} 
 + 2 \Big(n-\frac{\alpha}{\gamma_{\inf} - 1} \Big)\int_{\Omega(t)}p\dd \boldsymbol{x}.
 \end{split}
 \end{align*}
 Thus, we need to assume $\alpha\in(0,2]$ and $\gamma_{\inf} \geq \frac{n + \alpha}{n}$ 
 such that $H''(t)\geq 2\alpha E$, which implies
 $$
 H(t)\geq H(0)+H'(0)t+\alpha E t^2.
 $$
 On the other hand, we have
 $$
H(t)=\int_{\Omega(t)}\rho(t,\boldsymbol{x})|\boldsymbol{x}|^2\dd \boldsymbol{x}\leq R(t)^2M.
 $$
 Then 
 $$
 H(0)+H'(0)t+\alpha Et^2\leq H(t)\leq R(t)^2M.
 $$
 This implies 
 $$
 O(t^{-1})+\alpha E\leq \frac{R(t)^2M}{t^2}.
 $$
 Therefore, we conclude
 $$
 \liminf_{t\rightarrow\infty}\Big(\frac{R(t)}{t}\Big)\geq\Big(\frac{\alpha E}{M}\Big)^{\frac{1}{2}}.
 $$
\end{proof}

For $\alpha > 2$, the previous method no longer yields a positive lower bound 
for the second derivative of $H(t)$. 
To overcome this difficulty, we analyze the balance between the internal energy and the gravitational potential energy. In particular, we show that, for sufficiently small initial mass, 
the internal energy dominates the gravitational potential energy appearing in the second derivative of 
$H(t)$, thereby restoring a positive lower bound.

\begin{lemma}
Suppose that $\alpha\in(2,n-1),$ $\gamma_{\inf} > \frac{n + 2}{n}$, 
$\lim_{\rho\to\infty}p(\rho)\rho^{-\frac{n+\alpha}{n}} = \infty$, and the total energy $E>0$.
Let $(\rho,\mathcal{M})$ be a global classical solution of the attractive 
CEREs {\rm (}$\kappa = 1${\rm )} \eqref{1.1}. 
Then there exists an $\tilde{M} \in (0,\infty]$ such that, if $M < \tilde{M}$, 
\begin{equation}\label{*mass}
 \Lambda_{\min}(M) := 2E 
 + \inf_{R>0} \frac{M}{R} \Big (np(R) - 2 R e(R) 
 - \frac{(\alpha - 2)}{2\alpha} C_* M^\frac{n-\alpha}{n} R^\frac{n + \alpha}{n} \Big) > 0,
\end{equation}
and
$$
\liminf_{t\rightarrow\infty}\Big(\frac{R(t)}{t}\Big)
\geq\Big(\frac{ \Lambda_{\min}(M)}{M}\Big)^{\frac{1}{2}},
$$
where $R(t)$ is defined in \eqref{2.5a}.
\end{lemma}

\begin{remark}
    In the polytropic case, \eqref{*mass} is equivalent to $\gamma \geq \frac{n + \alpha}{n}$ with the critical mass given by $\tilde{M} = M_*$, defined by
\[
 M_* := \begin{cases}
  \bigg ( \Big( \frac{4n\alpha (\gamma - 1) E}{(\alpha - 2)(n(\gamma - 1) - \alpha)C_*} 
  \Big )^{n(\gamma - 1)} 
\Big( \frac{(n(\gamma - 1) - 2)(n(\gamma - 1) - \alpha)}{2\alpha E(\gamma - 1)} \Big)^\alpha  
 \bigg )^\frac{1}{(2n - \alpha)\gamma - 2n} & \mbox{for $\gamma > \frac{n + \alpha}{n}$}, \\
 \bigg ( \frac{2\alpha (n(\gamma -1) - 2)}{(\alpha - 2)(\gamma - 1)C_*} \bigg )^\frac{n(\gamma -1)}{(2n-\alpha)\gamma - 2n} & \mbox{for $\gamma = \frac{n + \alpha}{n}$},
\end{cases} 
\]
where $C_* > 0$ is the sharp constant from the variant of the HLS inequality.
\end{remark}

\begin{proof}
As in Lemma \ref{E>0,general}, for $H(t)$ defined in \eqref{3.1a}, we have
\begin{align}\label{alpha>2,general}
 \begin{split}
 H''(t) & =2\Big(\int_{\Omega(t)}\Big| \frac{\mathcal{M}}{\sqrt{\rho}} \Big |^2\dd \boldsymbol{x}+n\int_{\Omega(t)}p\dd \boldsymbol{x} + \frac{\alpha}{2} \int_{\Omega(t)}\rho \Phi_\alpha\ast\rho\dd \boldsymbol{x}\Big) \\
 & = 4 E + 2\int_{\Omega(t)}(np - 2 \rho e(\rho))\dd \boldsymbol{x} + (\alpha - 2)\int_{\Omega(t)}\rho \Phi_\alpha\ast\rho\dd \boldsymbol{x}. \\
  \end{split}
    \end{align}
Applying the variant of the HLS inequality in Lemma \ref{HLSineq} to \eqref{alpha>2,general}, we obtain
 \begin{equation}\label{next}
  H''(t) \geq 4 E + 2\int_{\Omega(t)}(np - 2 \rho e(\rho))\dd \boldsymbol{x} - \frac{(\alpha - 2)}{\alpha} C_* \| \rho \|_{L^1}^\frac{n-\alpha}{n} \| \rho \|_{L^\frac{n+\alpha}{n}}^\frac{n+\alpha}{n}.
        \end{equation}
We seek a lower bound for the right-hand side of \eqref{next}. 
To this end, we now show that it is minimized by a function of the form: 
\begin{equation}\label{scale}
    \sigma_R(\boldsymbol{x}) := R \sigma ((R \|\sigma\|_{L^1})^\frac{1}{n} M^{-\frac{1}{n}} \boldsymbol{x} ),
\end{equation}
for $\sigma(\boldsymbol{x}) = \mathds{1}_{B_1}(\boldsymbol{x})$. In this case, it can be shown that $\sigma_{R}(\boldsymbol{x}) = R\mathds{1}_{B_{(\frac{nM}{R \omega_n})^{1/n}}}(\boldsymbol{x})$.

Therefore, we need to show that the following function $f_M:X_M \to \mathbb{R}$ for fixed $M$ is minimized by $\sigma_{R}$ for some $R$, where
\[
f_M(\rho) := \int_{\Omega(t)}(np - 2 \rho e(\rho))\dd \boldsymbol{x} - \frac{(\alpha - 2)}{2\alpha} C_* M^\frac{n-\alpha}{n} \int_{\Omega(t)} \rho^\frac{n+\alpha}{n} \dd \boldsymbol{x}.
\]
Plugging in $\sigma_R$ yields
\begin{equation}\label{mini}
f_M(\sigma_{R}) = \frac{M}{R} \Big (np(R) - 2 R e(R) - \frac{(\alpha - 2)}{2\alpha} C_* M^\frac{n-\alpha}{n} R^\frac{n + \alpha}{n} \Big).
\end{equation}
Thus, we consider the infimum $ \Lambda_{\min}(M)$ of this plus $2E$:
\[
 \Lambda_{\min}(M) = 2E + \inf_{R > 0} f_M(\sigma_{R}).
\]
Using \eqref{mini} and $\gamma_{\inf} > \frac{n + 2}{n}$ for $\gamma_{\inf}$ 
defined in \eqref{gammainf}, we deduce 
\begin{equation}\label{lowerforf}
    f_M(\sigma_{R}) \geq \frac{M}{R} \Big (\big(n-\frac{2}{\gamma_{\inf} - 1}\big)p(R)
    - \frac{(\alpha - 2)}{2\alpha} C_* M^\frac{n-\alpha}{n} R^\frac{n + \alpha}{n} \Big).
\end{equation}
Since $\lim_{\rho\to\infty}p(\rho)\rho^{-\frac{n+\alpha}{n}} = \infty$, we have
\begin{equation*}
\lim_{R \to \infty} f_M(\sigma_{R}) = \infty.
\end{equation*}
This implies that $ \Lambda_{\min}(M) > - \infty$. 
One can see that $M \mapsto  \Lambda_{\min}(M)$ is continuous, 
as $p(R)$ and $e(R)$ are continuous. 
Combining \eqref{lowerforf} with $\lim_{\rho\to\infty}p(\rho)\rho^{-\frac{n+\alpha}{n}} = \infty$, 
we conclude that there exists $C \in \mathbb{R}$ such that, for $0<M<1$,
\[
    \frac{1}{M} f_M(\sigma_{R}) \geq \Big (\big(n-\frac{2}{\gamma_{\inf} - 1}\big)\frac{p(R)}{R} 
    - \frac{(\alpha - 2)}{2\alpha} C_* R^\frac{\alpha}{n} \Big) \geq C.
\]
Thus, $\liminf_{M \to 0} \Lambda_{\min}(M) \geq 2E > 0$.

\smallskip
\noindent
{\bf Case 1.} {\it $\Lambda_{\min}(M) > 0$ for all $M>0$}.
Then we let $\tilde{M} := \infty$.

\smallskip
\noindent
{\bf Case 2.} {\it There exists $M > 0$ such that $\Lambda_{\min}(M) \leq 0$}.
Then we can see that $M \mapsto  f_M(\sigma_{R})$ is strictly decreasing in the set:
$\{ M > 0 \, : \, f_M(\sigma_{R}) \leq 0\}$. 
Hence, $M \mapsto  \Lambda_{\min}(M)$ is decreasing in the 
set: $\{ M > 0 \, : \, \Lambda_{\min}(M) \leq 0\}$. 
Since $\liminf_{M \to 0} \Lambda_{\min}(M) > 0$, there exists $\tilde{M} \in (0,\infty)$ (the first mass $M>0$ such that $\Lambda_{\min}(M)=0$) such that $ \Lambda_{\min}(M) \leq 0 $ for all $M \geq \tilde{M}$ and $ \Lambda_{\min}(M) > 0$ for all $M < \tilde{M}$.

We now show that $ \Lambda_{\min}(M) - 2E$ is the infimum of $f_M$. 
For $\rho \in X_M$, we see that, by the definition of $\Lambda_{\min}(M)$, 
\begin{equation*} 
    f_M(\rho) = \frac{1}{M}\int_{\{ \rho > 0 \}} f_M(\sigma_{\rho(\boldsymbol{x})}) \rho(\boldsymbol{x}) \dd \boldsymbol{x}  \geq \frac{ \Lambda_{\min}(M) -2E}{M} \int_{\mathbb{R}^n} \rho(\boldsymbol{x}) \dd \boldsymbol{x} =  \Lambda_{\min}(M) - 2E,
\end{equation*}
which implies that $\Lambda_{\min}(M) = 2E + \inf_{\rho \in X_M} f_M(\rho)$. 
Therefore, for $M < \tilde{M}$, employing \eqref{next}, we have
\[ H''(t) \geq 2 \Lambda_{\min}(M) > 0,
\]
so that 
$$
H(t)\geq H(0)+H'(0)t+ \Lambda_{\min}(M) t^2.
$$
On the other hand, since
$$
H(t)=\int_{\Omega(t)}\rho(t,\boldsymbol{x})|\boldsymbol{x}|^2\dd \boldsymbol{x}\leq R(t)^2M,
$$
then
$$
H(0)+H'(0)t+\Lambda_{\min}(M)t^2\leq H(t)\leq R(t)^2M.
$$
This implies
$$
O(t^{-1})+\Lambda_{\min}(M)\leq \frac{R(t)^2M}{t^2}.
$$
Therefore, we have
$$
\liminf_{t\rightarrow\infty}\Big(\frac{R(t)}{t}\Big)
\geq\Big(\frac{\Lambda_{\min}(M)}{M}\Big)^{\frac{1}{2}}.
$$
\end{proof}
\setcounter{case}{0}

In the polytropic case with $\alpha \in (0,2]$, 
we develop a different approach from \cite{Deng_2002} to show
the growth of solutions with total energy $E=0$, stated below.

\begin{lemma}
Let $\gamma>\frac{n+\alpha}{n}$ and $\alpha\in(0,2]\cap (0,n-1)$, 
and let $(\rho,\mathcal{M})$ be a global classical solution of 
the attractive CEREs {\rm (}$\kappa = 1${\rm )} \eqref{1.1}. 
If the total energy $E=0$, then there exists a sequence $(t_m)_m \subset [0,\infty)$ such that
$R(t_m) \to \infty$ as $m \to \infty$.
\end{lemma}

\begin{proof}
As in Lemma \ref{E>0,general}, for $H(t)$ as defined in \eqref{3.1a},
we have
    \begin{align*}
 H''(t)=2\int_{\Omega(t)}\Big| \frac{\mathcal{M}}{\sqrt{\rho}} \Big|^2\dd \boldsymbol{x}+2n\int_{\Omega(t)}p\dd \boldsymbol{x}-2\alpha\int_{\Omega(t)}\frac{1}{2}\Big| \frac{\mathcal{M}}{\sqrt{\rho}} \Big|^2
 +\rho e(\rho)\dd \boldsymbol{x}.
    \end{align*}
Then we require $\alpha\in(0,2]$ and $\gamma\geq\frac{n+\alpha}{n}$ to obtain
    \begin{align*}
  H''(t)\geq 2a_0\frac{n(\gamma - 1) - \alpha}{\gamma - 1}\int_{\Omega(t)} \rho^\gamma \dd \boldsymbol{x}.
    \end{align*}
By the H\"older inequality, we have
    \[
    M = \int_{\Omega(t)} \rho \dd \boldsymbol{x} 
    \leq \Big ( \int_{\Omega(t)} \rho^\gamma \dd \boldsymbol{x} \Big )^\frac{1}{\gamma} |\Omega(t)|^\frac{\gamma - 1}{\gamma}.
    \]
    Therefore, we obtain the lower bound for the second derivative of $H$ given by
    \begin{align}\label{lowerH}
 H''(t)\geq 2a_0\frac{n(\gamma - 1) - \alpha}{\gamma - 1} M^\gamma |\Omega(t)|^{1 - \gamma} \geq 2a_0\frac{n(\gamma - 1) - \alpha}{\gamma - 1} M^\gamma 
 \Big(\frac{\omega_n}{n}R(t)^n \Big)^{1 - \gamma},
\end{align}
    where $\omega_n$ is the surface area of the unit hypersphere in $\mathbb{R}^n$. 
By Taylor's theorem, for each $t \in [0, \infty)$, there exists $\xi_t \in [0, t)$ such that
    \begin{align}\label{remain}
 \begin{split}
  H(t) = H(0) + tH'(0) + \frac{t^2}{2} H''(\xi_t).
        \end{split}
    \end{align}
    Therefore, combining \eqref{lowerH} with \eqref{remain}, we obtain
    \begin{align*}
  \begin{split}
  H(t) \geq H(0) + tH'(0) + t^2a_0\frac{n(\gamma - 1) - \alpha}{\gamma - 1} M^\gamma \Big(\frac{\omega_n}{n}R(\xi_t)^n \Big)^{1 - \gamma}.
  \end{split}
    \end{align*}
    From the same approach as in Lemma \ref{E>0,general}, we have
    \[
    \liminf_{t\rightarrow\infty}\bigg(\frac{R(t) R(\xi_t)^\frac{n(\gamma - 1)}{2}}{t}\bigg)\geq\bigg(a_0\frac{n(\gamma - 1) - \alpha}{(\gamma - 1)}\Big(\frac{nM}{\omega_n}\Big)^{\gamma - 1}\bigg)^{\frac{1}{2}}.
    \]
    Therefore, taking any sequence $(t_m)_m \subset [0,\infty)$ such that $t_m \to \infty$ as $m \to \infty$, we obtain that
    $R(t_m) R(\xi_{t_m})^\frac{n(\gamma - 1)}{2} \to \infty$ as $m \to \infty$. 
    Therefore, there exists a subsequence (still denoted as) $(t_m)_m$ such that 
    either $R(t_m) \to \infty$ or $R(\xi_{t_m}) \to \infty$ as $m \to \infty$. In the latter case, we relabel $\xi_{t_m}$ to $t_m$.
\end{proof}

In the previous lemmas, we have required the solutions to be strong. 
We now turn to the study of finite-energy solutions of the attractive CEREs \eqref{1.1} 
with finite second moment.

\begin{corollary}
Suppose that $(\rho,\mathcal{M})$ is a finite energy solution as in {\rm Definition \ref{definition}} 
with finite second-moment 
$\rho(t,\boldsymbol{x}) |\boldsymbol{x}|^2 \in L^1(0,T;L^1(\mathbb{R}^n))$ for all $T>0$. 
If there exists $\Lambda > 0$ such that
\[
\Lambda \leq \int_{\mathbb{R}^n} \Big| \frac{\mathcal{M}}{\sqrt{\rho}} \Big|^2\dd \boldsymbol{x}+n\int_{\mathbb{R}^n}p\dd \boldsymbol{x} + \frac{\alpha}{2}\int_{\mathbb{R}^n}\rho\Phi_\alpha\ast\rho\dd \boldsymbol{x} 
\qquad\mbox{for {\it a.e.} $t \geq 0$},
\]
then either there exists a set $A \subseteq (0,\infty)$ with non-zero measure 
such that $R(t) = \infty$ for all $t \in A$ or
    \[
    \lim_{t \to \infty} \underset{s \geq t}{\text{ess}\inf}
    \Big(\frac{R(s)}{s}\Big)\geq \Big(\frac{\Lambda}{M}\Big)^{\frac{1}{2}},
    \]
where $R(t)$ is the radius of domain $\Omega$ as defined in \eqref{2.5a}.
\end{corollary}

\begin{proof} We separate the discussion into two cases:
\begin{case}
{\it The map: $\boldsymbol{x} \mapsto \rho(t,\boldsymbol{x})$ does not have compact support on a non-zero measure 
set of $t \geq 0$.}  Then $R(t) = \infty$ on a non-zero measure set of $t \geq 0$. 
\end{case}
\begin{case}
{\it The map: $\boldsymbol{x} \mapsto \rho(t,\boldsymbol{x})$ has compact support for {\it a.e.} $t \geq 0$}. 
Let $\eta \in C^\infty_{\rm c}(\mathbb{R})$ and $\mu\in C^\infty_{\rm c}(\mathbb{R}^n)$ 
be standard spherically symmetric mollifiers. As usual, 
let $\eta_{\varepsilon}(t) = \frac{1}{{\varepsilon}}\eta(\frac{t}{{\varepsilon}})$ 
and $\mu_{\delta}(\boldsymbol{x}) = \frac{1}{{\delta}^n}\mu(\frac{\boldsymbol{x}}{{\delta}})$. 
Then we define the following notation for taking the mollification of 
a function $f:[0,\infty)\times\mathbb{R}^n \to \mathbb{R}^m$ by
\[
\begin{cases}
    f_{\varepsilon}(t,\boldsymbol{x}) = (f(\cdot,\boldsymbol{x})*\eta_{\varepsilon})(t)&\quad \text{ for } (t,\boldsymbol{x}) \in [{\varepsilon},\infty)\times\mathbb{R}^n, \\
    f_{\delta}(t,\boldsymbol{x}) = (f(t,\cdot)*\mu_{\delta})(\boldsymbol{x})&\quad\text{ for } (t,\boldsymbol{x}) \in [0,\infty)\times\mathbb{R}^n, \\
    f_{{\varepsilon},{\delta}} = (f_{\varepsilon})_{\delta} = (f_{\delta})_{\varepsilon}.
\end{cases}
\]
Notice that, for $f(t,\cdot) \in L^1(\mathbb{R}^n)$, 
\begin{equation}\label{mollme}
\int_{\mathbb{R}^n} f_{\delta}(t,\boldsymbol{x}) \dd \boldsymbol{x} 
= \int_{\mathbb{R}^n} f(t,\boldsymbol{x}) \dd \boldsymbol{x}.
\end{equation}
We define $H_{{\varepsilon},{\delta}}: [{\varepsilon},\infty) \to \mathbb{R}$ by
\[
H_{{\varepsilon},{\delta}}(t) := \int_{\mathbb{R}^n} \rho_{{\varepsilon},{\delta}}(t,\boldsymbol{x}) |\boldsymbol{x}|^2 \dd \boldsymbol{x}.
\]
Since $(\rho,\mathcal{M})$ is a weak solution, we see that, 
for $(t,\boldsymbol{x}) \in ({\varepsilon},\infty)\times\mathbb{R}^n$, 
the mollification satisfies
\begin{equation}\label{mollmass}
    \begin{split}
  &(\rho_{{\varepsilon},{\delta}})_t + \nabla \cdot (\mathcal{M}_{{\varepsilon},{\delta}}) \\
  &= \int^\infty_0 \int_{\mathbb{R}^n} \Big ( \rho(s,\boldsymbol{y}) (\eta_{\varepsilon})_t(t-s) \mu_{\delta}(\boldsymbol{x} - \boldsymbol{y}) + \mathcal{M}(s,\boldsymbol{y}) \cdot \eta_{\varepsilon}(t-s) \nabla_{\boldsymbol{x}} \mu_{\delta}(\boldsymbol{x} - \boldsymbol{y}) \Big ) \dd \boldsymbol{y} \dd s \\
 &= - \int^\infty_0 \int_{\mathbb{R}^n} \Big ( \rho(s,\boldsymbol{y}) \big[\eta_{\varepsilon}(t-\cdot) \mu_{\delta}(\boldsymbol{x} - \cdot)\big]_s(s,\boldsymbol{y}) + \mathcal{M}(s,\boldsymbol{y}) \cdot \nabla_{\boldsymbol{y}} \big[\eta_{\varepsilon}(t-\cdot) \mu_{\delta}(\boldsymbol{x} - \cdot)\big](s,\boldsymbol{y}) \Big ) \dd \boldsymbol{y} \dd s\\
 & = 0,
    \end{split}
\end{equation}
where the last line follows from the fact that  
$\big[\eta_{\varepsilon}(t-\cdot) \mu_{\delta}(\boldsymbol{x}- \cdot)\big] 
\in C^\infty_{\rm c}((0,\infty) \times \mathbb{R}^n)$ due to $t>\epsilon$
and can then be used as a test function in the definition of weak solutions of CEREs. 
Similarly, one can show that, for $(t,\boldsymbol{x}) \in ({\varepsilon},\infty)\times\mathbb{R}^n$,
\begin{equation}\label{mollmoment}
\partial_t \M_{{\varepsilon},{\delta}}
+ \nabla \cdot \Big(\frac{\M\otimes\M}{\rho}\Big)_{{\varepsilon},{\delta}}
+\nabla p_{{\varepsilon},{\delta}} + (\rho \nabla \Phi_\alpha * \rho)_{{\varepsilon},{\delta}} 
= \boldsymbol{0}.
\end{equation}
Since $\rho(t,\boldsymbol{x}) |\boldsymbol{x}|^2 \in L^1(0,T;L^1(\mathbb{R}^n))$ 
and $\frac{\mathcal{M}}{\sqrt{\rho}} \in L^{\infty}(0,T; L^2(\mathbb{R}^n))$, then
\[
    \int_{\mathbb{R}^n}\mathcal{M} \cdot \boldsymbol{x}\dd \boldsymbol{x} \in L^1(0,T).
\]
Now, when differentiating $H_{{\varepsilon},{\delta}}$ in $t$, 
we apply \eqref{mollmass}, Lebesgue's differentiation theorem, 
Fubini's theorem, and integration by parts to obtain
\begin{equation*}
  H_{{\varepsilon},{\delta}}'(t) = 2\int_{\mathbb{R}^n}\mathcal{M}_{{\varepsilon},{\delta}}(t) \cdot \boldsymbol{x}\dd \boldsymbol{x} = 2\int_{\mathbb{R}^n}\mathcal{M}_{\varepsilon}(t) \cdot \boldsymbol{x}_{\delta} \dd \boldsymbol{x} 
= 2\,\Big( \int_{\mathbb{R}^n}\mathcal{M} \cdot \boldsymbol{x} \dd \boldsymbol{x} \Big)_{\varepsilon}(t).
\end{equation*}
Now, differentiating once more, applying \eqref{mollmoment}, 
and using the fact that 
$\frac{\mathcal{M}}{\sqrt{\rho}} \in L^{\infty}(0,T; L^2(\mathbb{R}^n))$, 
$p(\rho) \in L^{\infty}(0,T; L^1(\mathbb{R}^n))$, 
$\rho \Phi_\alpha * \rho \in L^{\infty}(0,T; L^1(\mathbb{R}^n))$, 
and $\boldsymbol{x} \mapsto \rho(t,\boldsymbol{x})$ has compact support for {\it a.e.} $t \geq 0$,
we can integrate by parts and conclude that the boundary terms vanish. Hence, we obtain
\begin{equation}\label{weakdouble}
 H_{{\varepsilon},{\delta}}''(t)=2\bigg(\int_{\mathbb{R}^n} \bigg ( \Big| \frac{\mathcal{M}}{\sqrt{\rho}} \Big|^2 \bigg )_{{\varepsilon},{\delta}} \dd \boldsymbol{x}+n\int_{\mathbb{R}^n}p_{{\varepsilon},{\delta}}\dd \boldsymbol{x}-\int_{\mathbb{R}^n}(\rho\nabla\Phi_\alpha\ast\rho)_{{\varepsilon},{\delta}} \cdot \boldsymbol{x}\dd \boldsymbol{x}\bigg).
\end{equation}
Since $\boldsymbol{x} \mapsto 
(\rho\nabla\Phi_\alpha\ast\rho)_{{\delta}}(s,\boldsymbol{x}) \cdot \boldsymbol{x}$ 
has compact support for {\it a.e.} $s \geq 0$ and 
is continuous for each fixed $s\geq 0$,
we can apply Fubini's theorem to obtain
\begin{equation*}
 \int_{\mathbb{R}^n}(\rho\nabla\Phi_\alpha\ast\rho)_{{\delta}} \cdot \boldsymbol{x}\dd \boldsymbol{x} = \int_{\mathbb{R}^n}(\rho\nabla\Phi_\alpha\ast\rho) \cdot \boldsymbol{x}_{{\delta}} \dd \boldsymbol{x}.
\end{equation*}
Since $\boldsymbol{x}_{\delta} = \boldsymbol{x}$, then 
$(\rho\nabla\Phi_\alpha\ast\rho) \cdot \boldsymbol{x}$ is integrable 
for {\it a.e.} $s \geq 0$ so that
\begin{align*}\nonumber
I&:=\int_{\mathbb{R}^n}\rho\,\nabla\Phi_\alpha\ast\rho\cdot \boldsymbol{x}\,\dd \boldsymbol{x}\\
&=\int_{\mathbb{R}^n}\rho(\boldsymbol{x})\int_{\mathbb{R}^n}
\frac{\rho(\boldsymbol{y})(\boldsymbol{x}-\boldsymbol{y})\cdot \boldsymbol{x}}
{|\boldsymbol{x}-\boldsymbol{y}|^{\alpha+2}}
\,\dd \boldsymbol{y}\,\dd \boldsymbol{x}\\
&=\int_{\mathbb{R}^n}\rho(\boldsymbol{x})\int_{\mathbb{R}^n}
\frac{\rho(\boldsymbol{y})|\boldsymbol{x}-\boldsymbol{y}|^2}
{|\boldsymbol{x}-\boldsymbol{y}|^{\alpha+2}}
\,\dd \boldsymbol{y}\,\dd \boldsymbol{x}
+\int_{\mathbb{R}^n}\rho(\boldsymbol{x})\int_{\mathbb{R}^n}
\frac{\rho(\boldsymbol{y})(\boldsymbol{x}-\boldsymbol{y})\cdot \boldsymbol{y}}
{|\boldsymbol{x}-\boldsymbol{y}|^{\alpha+2}}
\,\dd \boldsymbol{y}\,\dd \boldsymbol{x}\\
&=-\alpha\int_{\mathbb{R}^n}\rho\,\Phi_\alpha\ast\rho\,\dd \boldsymbol{x}
+\int_{\mathbb{R}^n}\rho(\boldsymbol{y})\int_{\mathbb{R}^n}
\frac{\rho(\boldsymbol{x})(\boldsymbol{y}-\boldsymbol{x})\cdot \boldsymbol{x}}
{|\boldsymbol{x}-\boldsymbol{y}|^{\alpha+2}}
\,\dd \boldsymbol{x}\,\dd \boldsymbol{y}\\
&=-\alpha\int_{\mathbb{R}^n}\rho\,\Phi_\alpha\ast\rho\,\dd \boldsymbol{x}-I.
 \end{align*}
Then we have 
$$
\int_{\mathbb{R}^n}(\rho\nabla\Phi_\alpha\ast\rho)_{{\delta}} \cdot \boldsymbol{x}\dd \boldsymbol{x}
=I=-\frac{\alpha}{2}\int_{\mathbb{R}^n}\rho\Phi_\alpha\ast\rho\dd \boldsymbol{x}.
$$
Now, for $t > {\varepsilon}$, multiplying by $\eta_{\varepsilon}(t-s)$ and integrating over $s$, we obtain
 $$\int_{\mathbb{R}^n}(\rho\nabla\Phi_\alpha\ast\rho)_{{\varepsilon},{\delta}} \cdot \boldsymbol{x}\dd \boldsymbol{x}=-\frac{\alpha}{2}\int_{\mathbb{R}^n}(\rho\Phi_\alpha\ast\rho)_{\varepsilon}\dd \boldsymbol{x}.$$
 Therefore, applying \eqref{mollme} and \eqref{weakdouble} yields
 \[
 \begin{split}
 H_{{\varepsilon},{\delta}}''(t) & =2\bigg(\int_{\mathbb{R}^n} \bigg ( \Big| \frac{\mathcal{M}}{\sqrt{\rho}} \Big|^2 \bigg )_{{\varepsilon}} \dd \boldsymbol{x}+n\int_{\mathbb{R}^n}p_{{\varepsilon}}\dd \boldsymbol{x} + \frac{\alpha}{2}\int_{\mathbb{R}^n}(\rho\Phi_\alpha\ast\rho)_{\varepsilon}\dd \boldsymbol{x}\bigg) \\
 & = 2\bigg(\int_{\mathbb{R}^n} \Big| \frac{\mathcal{M}}{\sqrt{\rho}} \Big|^2 \dd \boldsymbol{x}+n\int_{\mathbb{R}^n}p\dd \boldsymbol{x} + \frac{\alpha}{2}\int_{\mathbb{R}^n}(\rho\Phi_\alpha\ast\rho)\dd \boldsymbol{x}\bigg)_{{\varepsilon}}
 \end{split}
\]
and so is independent of ${\delta}$. 
Therefore, if there exists $\Lambda > 0$ such that
\[
 \Lambda \leq \int_{\mathbb{R}^n} \Big| \frac{\mathcal{M}}{\sqrt{\rho}} \Big|^2\dd \boldsymbol{x}+n\int_{\mathbb{R}^n}p\dd \boldsymbol{x} + \frac{\alpha}{2}\int_{\mathbb{R}^n}\rho\Phi_\alpha\ast\rho\dd \boldsymbol{x}
 \qquad\mbox{for {\it a.e.} $t > 0$},
\]
then, since $\eta$ is non-negative, we have
\[ H_{{\varepsilon},{\delta}}''(t) \geq 2\Lambda.
\]
This implies
\begin{equation}\label{molled}
    H_{{\varepsilon},{\delta}}(t) \geq H_{{\varepsilon},{\delta}}(t_0) + (t-t_0)H_{{\varepsilon},{\delta}}'(t_0) + (t - t_0)^2 \Lambda
    \qquad\mbox{for {\it a.e.} $\varepsilon < t_0 < t$}.
\end{equation}
 Since $\int_{\mathbb{R}^n}\rho(\boldsymbol{x})|\boldsymbol{x}|^2\dd \boldsymbol{x}, \,\int_{\mathbb{R}^n}\mathcal{M} \cdot \boldsymbol{x}\dd \boldsymbol{x} \in L^1(0,T)$ for all $T>0$, 
 we see that, for {\it a.e.} $0 < t_0 < t$, letting $\varepsilon \to 0$ in \eqref{molled} yields
\[
H(t) \geq H(t_0) + 2 (t-t_0)\int_{\mathbb{R}^n} \mathcal{M}(t_0) \cdot \boldsymbol{x} \dd \boldsymbol{x} 
+ (t - t_0)^2 \Lambda.
\]

This is true for all $T > 0$, so it can be extended to all $t > 0$. 
Consequently, applying the arguments used in Lemma \ref{E>0,general}, we obtain 
\[
\lim_{t \to \infty} \underset{s \geq t}{\text{ess}\inf}\Big(\frac{R(s)}{s}\Big)
 \geq \Big(\frac{\Lambda}{M}\Big)^{\frac{1}{2}}.
\]
\end{case}
In either case, we obtain the instability of solutions.
\end{proof}
\setcounter{case}{0}

\medskip
\section{Instability for $\frac{2n}{2n - \alpha} < \gamma < \frac{n + \alpha}{n}$}\label{instability}

Given $\varrho \in L^1_+(\R^n) \cap L^\gamma(\mathbb{R}^n)$  
and 
$\frac{2n}{2n - \alpha} < \gamma < \frac{n + \alpha}{n}$
for the polytropic case $p(\varrho) = a_0 \varrho^\gamma$, we define
\begin{equation*}
Q(\varrho) := n a_0 \int_{\mathbb{R}^n} \varrho^\gamma \dd \boldsymbol{x} + \frac{\alpha}{2} \int_{\mathbb{R}^n} \varrho \Phi_\alpha * \varrho \dd \boldsymbol{x}.
\end{equation*}
This quantity will play a crucial role in this section.
More precisely, we write energy \eqref{energies} as
\begin{equation}\label{energy2}
 E(\rho,\mathcal{M}) = \int_{\mathbb{R}^n} \frac12 
  \Big |\frac{\mathcal{M}}{\sqrt{\rho}} \Big|^2\,\dd \boldsymbol{x}
  +S_\mu(\rho)- \mu \int_{\mathbb{R}^n} \rho \dd \boldsymbol{x},
\end{equation}
where the action functional $S_\mu(\varrho)$ is defined similarly to \cite{Cheng_2025} as
\[
S_\mu(\varrho) := \frac{a_0}{\gamma - 1} \int_{\mathbb{R}^n} \varrho^\gamma \dd \boldsymbol{x} 
+ \frac{1}{2} \int_{\mathbb{R}^n} \varrho \Phi_\alpha * \varrho \dd \boldsymbol{x} 
+ \mu \int_{\mathbb{R}^n} \varrho \dd \boldsymbol{x}.
\]
As we will see next, identity \eqref{energy2} is convenient for the rest of this section, 
and functional $Q(\varrho)$ is closely related to the derivative of $S_\mu(\varrho)$ 
with respect to the scaling parameter associated with mass-preserving dilations. 
Our goal is to show that, if a solution $(\rho(t),\mathcal{M}(t))$ of \eqref{1.1} 
with initial conditions $(\rho_0,\mathcal{M}_0)$ satisfies
\[
 \int_{\mathbb{R}^n} \rho_0 \dd \boldsymbol{x} < \Big ( \frac{n + \alpha - n\gamma}{(2n-\alpha)\gamma - 2n} \Big )^\frac{n + \alpha - n\gamma}{(2n-\alpha)\gamma - 2n} \Big ( \ell_1 \frac{(2n-\alpha)\gamma - 2n}{(n - \alpha)(\gamma - 1)} \Big )^\frac{(n - \alpha)(\gamma - 1)}{(2n-\alpha)\gamma - 2n} E(\rho_0,\mathcal{M}_0)^\frac{n\gamma - n - \alpha}{(2n - \alpha)\gamma - 2n},
\]
for suitably chosen $\ell_1>0$, then either $R(t) \to \infty$ as $t \to \infty$ or there exists 
some $\Lambda > 0$ such that $Q(\rho(t)) \geq \Lambda$ for all $t \geq 0$. 

In the latter case, we can show that $H''(t) \geq 2 \Lambda$, as in \S 3,
which yields the instability of steady states satisfying \eqref{station}. 
We recall that $S_\mu(\varrho)$ is constructed so that \eqref{station} arises as its 
Euler-Lagrange equation 
under some restriction on $\varrho$. 
As will be shown later, any unique steady state $\bar\rho_\mu$ satisfies $Q(\bar\rho_\mu) = 0$. 
This motivates the study of minimizers of $S_\mu$ over the admissible set:
\[
\mathcal{K} := \big\{ \varrho \in L^1 \cap L^\gamma(\mathbb{R}^n) \, : \, Q(\varrho) = 0, 
\, \varrho \geq 0 \text{ {\it a.e.}, } \varrho \not\equiv 0\big\},
\]
as perturbations of $\bar\rho_\mu$. We define 
\begin{equation*}
\ell_\mu := \inf_{\varrho \in \mathcal{K}} S_\mu(\varrho),
\end{equation*}
the minimization of $S_\mu$ over $\mathcal{K}$. 

For $\varrho \in \mathcal{K}$, since $\gamma < \frac{n + \alpha}{n}$, then
\begin{equation}\label{overk}
S_\mu(\varrho) = a_0\frac{\alpha - n(\gamma - 1)}{\alpha(\gamma - 1)}\int_{\mathbb{R}^n} \varrho^\gamma \dd \boldsymbol{x} + \mu \int_{\mathbb{R}^n} \varrho \dd \boldsymbol{x} \geq 0.
\end{equation}
Consequently, $\ell_\mu \geq 0$. We will later show that, in fact, $\ell_\mu > 0$. 

To state the main theorem later, we first introduce the following definition:
\begin{definition}
We define
    \[
 \mathcal{I}_\mu :=\big\{ (\rho,\mathcal{M}) \,\,: \, Q(\rho) > 0, 
 \, I_\mu(\rho,\mathcal{M})< \ell_\mu\big\},
    \]
    where
    $$
        I_\mu(\rho,\mathcal{M}):= E(\rho,\mathcal{M}) + \mu \int_{\mathbb{R}^n} \rho \dd \boldsymbol{x}.
    $$
\end{definition}

We use the standard scaling: 
$\varrho_\lambda(\boldsymbol{x}) := 
\lambda \varrho(\lambda^\frac{1}{n} \boldsymbol{x})$ for $\lambda > 0$,
to preserve $L^1$ mass:
\[
\int_{\mathbb{R}^n} \varrho_\lambda \dd \boldsymbol{x} = \int_{\mathbb{R}^n} \varrho \dd \boldsymbol{x}.
\]
We define $S_{\mu,\lambda}(\varrho) := S_\mu(\varrho_\lambda)$ and state some important facts that will be useful later. For the rest of the section, we denote the negative interaction energy  by
\[
    A_\alpha(\varrho):=-\int_{\mathbb R^n}\varrho\,\Phi_\alpha*\varrho\,\dd \boldsymbol{x} \ge 0.
\]

\begin{lemma}\label{facts}
 Let $\alpha \in (0,n)$ and $\frac{2n}{2n - \alpha} < \gamma < \frac{n + \alpha}{n}$, 
 and let $\varrho \in L^1_+ \cap L^\gamma (\mathbb{R}^n)$ with $\varrho \not\equiv 0$. 
 Then the following statements are true{\rm :}
    \begin{enumerate}
 \item[{\rm(a)}] If $\varrho_m \rightharpoonup \tilde{\varrho}$ in $L^\gamma(\mathbb{R}^n)$ as $m \to \infty$, then $(\varrho_m)_\lambda \rightharpoonup \tilde{\varrho}_\lambda$ in $L^\gamma(\mathbb{R}^n)$ as $m \to \infty$.
\item[{\rm(b)}] $(\varrho_\lambda)^\# = (\varrho^\#)_\lambda$ for all $\lambda > 0$, where $(\cdot)^\#$ represents the spherically decreasing rearrangement.
 \item[{\rm(c)}] There exists a unique $\lambda^*(\varrho) > 0$ such that $\varrho_{\lambda^*(\varrho)} \in \mathcal{K}$.
 \item[{\rm(d)}] $\lambda^*(\varrho) > 1 \iff Q(\varrho) > 0$.
 \item[{\rm(e)}] $\lambda^*(\varrho) = 1 \iff Q(\varrho) = 0$.
 \item[{\rm(f)}] $\lambda \mapsto S_{\mu,\lambda}(\varrho)$ is strictly increasing {\rm(}decreasing{\rm)} for $\lambda \leq \lambda^*(\varrho)$ {\rm(}$\lambda \geq \lambda^*(\varrho)${\rm)}. Moreover, $S_{\mu,\lambda}(\varrho) < S_{\mu,\lambda^*(\varrho)}(\varrho)$ for all $\lambda \neq \lambda^*(\varrho)$ and $\lambda > 0$.
    \end{enumerate}
\end{lemma}

\begin{proof}
Both (a) and (b) follow trivially. For (c), we have
\[
Q(\varrho_\lambda) = n a_0 \lambda^{\gamma - 1} \int_{\mathbb{R}^n} \varrho^\gamma \dd \boldsymbol{x} - \frac{\alpha \lambda^{\frac{\alpha}{n}}}{2} A_\alpha(\varrho).
\]
Then there exists a unique $\lambda^*(\varrho) > 0$ such that $\varrho_{\lambda^*(\varrho)} \in \mathcal{K}$. We have
\begin{equation}\label{lambdastar}
   \lambda^*(\varrho) := \bigg ( \frac{2na_0 \int_{\mathbb{R}^n} \varrho^\gamma \dd \boldsymbol{x}}{\alpha A_\alpha(\varrho)} \bigg )^{\frac{n}{\alpha - n(\gamma - 1)}}.
\end{equation}
Both (d) and (e) follow from \eqref{lambdastar}. For (f), we calculate by a change of variables 
to obtain
\[
S_{\mu,\lambda}(\varrho) = \frac{a_0 \lambda^{\gamma - 1}}{\gamma - 1} \int_{\mathbb{R}^n} \varrho^\gamma \dd \boldsymbol{x} - \frac{\lambda^{\frac{\alpha}{n}}}{2} A_\alpha(\varrho) + \mu \int_{\mathbb{R}^n} \varrho \dd \boldsymbol{x}.
\]
Then differentiating with respect to $\lambda$ yields
\begin{equation}\label{firstderiva}
    \frac{\d S_{\mu,\lambda}(\varrho)}{\d \lambda} = a_0 \lambda^{\gamma - 2} \int_{\mathbb{R}^n} \varrho^\gamma \dd \boldsymbol{x} - \frac{\alpha \lambda^{\frac{\alpha - n}{n}}}{2n} A_\alpha(\varrho) = \frac{Q(\varrho_\lambda)}{n\lambda}.
\end{equation}
From part (d) and (e), we see that $Q(\varrho_\lambda) > 0$ when $\lambda < \lambda^*(\varrho)$
and $Q(\varrho_\lambda) < 0$ when $\lambda > \lambda^*(\varrho)$. 
Therefore, the map: $\lambda \mapsto S_{\mu,\lambda}(\varrho)$ is strictly increasing for $0 < \lambda < \lambda^*(\varrho)$ and strictly decreasing for $\lambda > \lambda^*(\varrho)$, which implies that
part (f) follows.
\end{proof}

We now study the concavity of the map: $\lambda \mapsto S_{\mu,\lambda}(\varrho)$. 
As mentioned previously, the concavity of the free energy with respect to the mass-preserving scaling $\lambda$ is fundamental in proving the convexity of the second moment of the density,
allowing us to prove the growth of the support of the density.

\begin{lemma}\label{concavity}
 Let $\alpha \in (0,n)$ and 
 $\frac{2n}{2n - \alpha} < \gamma < \frac{n + \alpha}{n}$, and 
 let $\varrho \in L^1_+ \cap L^\gamma (\mathbb{R}^n)$ with  $\varrho \not\equiv 0$. 
 Then there exists a neighborhood of $\lambda^*(\varrho)$ such that $\lambda \mapsto S_{\mu,\lambda}(\varrho)$ is a concave mapping. In particular, the map{\rm :} $\lambda \mapsto S_{\mu,\lambda}(\varrho)$ is concave for $0 < \lambda \leq \lambda^*(\varrho)$.
\end{lemma}

\begin{proof}
Differentiating \eqref{firstderiva} again, we see
\begin{equation}\label{secondS}
\frac{\d^2 S_{\mu,\lambda}(\varrho)}{\d \lambda^2} 
= \lambda^{\gamma - 3} 
\Big( a_0 (\gamma - 2)  \int_{\mathbb{R}^n} \varrho^\gamma \dd \boldsymbol{x} 
+ \frac{\alpha(\alpha - n) \lambda^{\frac{\alpha - n(\gamma - 1)}{n}}}{2n^2} 
\int_{\mathbb R^n}\varrho\Phi_\alpha*\varrho\dd \boldsymbol{x} \Big).
\end{equation}
Substituting in $\lambda^*(\varrho)$ and using the fact that $\gamma < \frac{n + \alpha}{n}$, we obtain
\begin{equation}\label{cri}
    \frac{\d^2 S_{\mu,\lambda}(\varrho)}{\d \lambda^2} \bigg |_{\lambda^*(\varrho)} = \frac{a_0 (n(\gamma - 1) - \alpha) \lambda^*(\varrho)^{\gamma - 3}}{n} \int_{\mathbb{R}^n} \varrho^\gamma \dd \boldsymbol{x} < 0.
\end{equation}
Since $\lambda \mapsto \frac{\d^2 S_{\mu,\lambda}(\varrho)}{\d \lambda^2}$ is continuous from \eqref{secondS}, we obtain the concavity in a neighborhood of $\lambda^*(\varrho)$. 
For $\alpha \in (0,n)$, we see that the map:
    \[
    \lambda \mapsto a_0 (\gamma - 2)  \int_{\mathbb{R}^n} \varrho^\gamma \dd \boldsymbol{x} + \frac{\alpha (\alpha - n) \lambda^{\frac{\alpha - n(\gamma - 1)}{n}}}{2n^2} \int_{\mathbb R^n}\varrho\Phi_\alpha*\varrho\dd \boldsymbol{x}
    \]
    is increasing in $\lambda$. Hence, applying similar calculations to that of \eqref{cri} and using \eqref{secondS}, we obtain that, for $0< \lambda \leq \lambda^*(\varrho)$,
    \[
  \frac{\d^2 S_{\mu,\lambda}(\varrho)}{\d \lambda^2} \leq \frac{a_0 (n(\gamma - 1) - \alpha) \lambda^{\gamma - 3}}{n} \int_{\mathbb{R}^n} \varrho^\gamma \dd \boldsymbol{x} < 0.
    \]
Thus, $\lambda \mapsto S_{\mu,\lambda}(\varrho)$ is concave in $(0,\lambda^*(\varrho)]$.
\end{proof}

Now, we come back to showing that, given a global classical solution $(\rho(t),\mathcal{M}(t))$ 
of the attractive CEREs \eqref{1.1} with suitable enough initial conditions 
on $(\rho_0,\mathcal{M}_0)$, then 
either $R(t) \to \infty$ as $t \to \infty$ or $Q(\rho(t)) \geq \Lambda$ 
for all $t \geq 0$ and some $\Lambda > 0$. 
To do this, we require our initial conditions to be in some suitable set 
of form $\mathcal{I}_\mu$.

\begin{lemma}[Invariance of $\mathcal{I}_\mu$ under the Solution Semiflow]\label{invariant}
Let $\alpha\in(0,n)$ and $\frac{2n}{2n-\alpha}<\gamma<\frac{n+\alpha}{n}$,
and let $(\rho(t),\mathcal{M}(t))$ be a classical solution of the attractive CEREs {\rm(}$\kappa=1${\rm)} with initial data $(\rho_0,\mathcal{M}_0)\in \mathcal{I}_\mu$. Then
\[
(\rho(t),\mathcal{M}(t))\in \mathcal{I}_\mu \qquad \text{for all } t>0.
\]
\end{lemma}

\begin{proof}
By the conservation laws of mass and energy,
\[
I_\mu(\rho(t),\mathcal{M}(t)) = I_\mu(\rho_0,\mathcal{M}_0) < \ell_\mu
\qquad \text{for all } t >0.
\]
Suppose by contradiction that there exists $t_0>0$ such that $Q(\rho(t_0))\le 0$.
Since $Q(\rho_0)>0$ and $t\mapsto Q(\rho(t))$ is continuous,
there exists a first time $t^\ast>0$ such that $Q(\rho(t^\ast))=0$.
By definition of $\ell_\mu$, we have
\[
S_\mu(\rho(t^\ast))\ge \ell_\mu.
\]
On the other hand,
\[
S_\mu(\rho(t^\ast))
\le I_\mu(\rho(t^\ast),\mathcal{M}(t^\ast))
= I_\mu(\rho_0,\mathcal{M}_0)
< \ell_\mu,
\]
which is a contradiction.
\end{proof}

We now prove the existence of minimizers fo the variational problem:
$$
\ell_\mu = \inf_{\varrho \in \mathcal{K}} S_\mu(\varrho),
$$ 
in order to obtain a more explicit characterization of 
$\bigcup_{\mu > 0} \mathcal{I}_\mu$.
A key step is to show that, for any minimizing sequence of $S_\mu$ over $\mathcal{K}$, 
the associated nonlocal potential term converges. 
To this end, we first establish weak convergence of minimizing sequences 
by using the following lemma:

\begin{lemma}\label{weaklowersemi}
Let $q > 1$, and let $\varrho_k \in L^1_+ \cap L^q (\mathbb{R}^n)$ a bounded sequence. 
Then there exists $\tilde{\varrho} \in L^1_+ \cap L^q (\mathbb{R}^n)$ 
such that  $\varrho_k \rightharpoonup \tilde{\varrho}$ 
in $L^q(\R^n)$ {\rm(}up to a subsequence{\rm)} and
    \[
    \|\tilde{\varrho}\|_{L^1} \leq \liminf_{k \to \infty} \|\varrho_k\|_{L^1}.
    \]
\end{lemma}
The proof is the same as in \cite{Cheng_2025}. \\\

Unlike the setting in \cite{Cheng_2025}, when $\alpha \neq n-2$, 
the potential no longer corresponds to the well-studied Poisson equation. 
Consequently, establishing convergence of the nonlocal term requires 
a different approach.
To this end, we obtain compactness by analyzing the potential in suitably chosen Sobolev spaces, 
or, in the case $\alpha \geq n-1$, in fractional Sobolev spaces. 
We exploit compact embedding results for these spaces, together with decay estimates 
for radial functions, to achieve the desired convergence.

For the next lemma, we denote $L^p_{\text{rad}}(\mathbb{R}^n)$ the Banach space of all radial 
functions in $L^p(\mathbb{R}^n)$, endowed with the subspace norm.

\begin{lemma}\label{weaktostrong}
    Let $n \geq 2$, $\alpha \in (0,n)$ and $\frac{2n}{2n - \alpha} <\gamma < \frac{n}{n - \alpha}$, 
    and let $\varrho_m \in L^1_{\text{rad}} \cap L^\gamma_{\text{rad}} (\mathbb{R}^n)$ 
    is a bounded sequence and $\varrho_m \rightharpoonup \tilde{\varrho}$ as $m \to \infty$ in $L^\gamma_{\text{rad}}(\mathbb{R}^n)$. Then, {\rm(}up to a subsequence{\rm)},
    \[
 \int_{\mathbb{R}^n} \varrho_m \Phi_\alpha * \varrho_m \dd \boldsymbol{x} \longrightarrow \int_{\mathbb{R}^n} \tilde{\varrho} \Phi_\alpha * \tilde{\varrho} \dd \boldsymbol{x}
 \qquad \mbox{as $m \to \infty$}.
     \]
\end{lemma}

\begin{proof} We divide the proof into five steps:

\medskip
\noindent
{\bf 1.} There are two cases:

\medskip
\noindent
{\bf Case 1. $\alpha \in (0,n-1)$}.
For $\frac{n\gamma}{n(\gamma - 1) + \gamma} < p < \frac{n \gamma}{n - \gamma(n - \alpha - 1)}$, 
then $\frac{np}{n + p(n-\alpha)} < \frac{np}{n + p(n-\alpha-1)} < \gamma$.
For $R, R' > 0$, using the HLS inequality and H\"older's inequality, we have
    \begin{align*}
  \begin{split}
   \|\Phi_\alpha * (\mathds{1}_{B_R} \varrho ) \|_{L^p(B_{R'})} & \lesssim \||\cdot|^{-\alpha} * (\mathds{1}_{B_R} \varrho ) \|_{L^p(B_{R'})} \\
  & \lesssim \|\varrho\|_{L^\frac{np}{n + p(n-\alpha)}(B_R)} \\
  & \lesssim R^\frac{\gamma(n + p(n-\alpha)) - np}{p\gamma} \|\varrho\|_{L^\gamma},
 \end{split}
    \end{align*}
    and
    \begin{align*}
 \begin{split}
 \|\nabla \Phi_\alpha * (\mathds{1}_{B_R} \varrho ) \|_{L^p(B_{R'})} & \leq \||\cdot|^{-(\alpha + 1)} * (\mathds{1}_{B_R} \varrho ) \|_{L^p(B_{R'})} \\
 & \lesssim \|\varrho\|_{L^\frac{np}{n + p(n-\alpha-1)}(B_R)} \\
& \lesssim R^\frac{\gamma(n + p(n - \alpha - 1)) - np}{p\gamma} \|\varrho\|_{L^\gamma}.
  \end{split}
    \end{align*}
Since $\varrho_m$ is a bounded sequence in $L^\gamma(\mathbb{R}^n)$, 
then $\Phi_\alpha * (\mathds{1}_{B_R} \varrho_m ) \in W^{1,p}(B_{R'})$ is a bounded sequence.
The Rellich-Kondrachov theorem tells us that, for $n \geq 2$ and $\gamma < \frac{n}{n-\alpha}$, 
then $1 < p < \frac{n \gamma}{n - \gamma(n - \alpha - 1)} < n$ and
    \[
 W^{1,p}(B_{R'}) \Subset L^q(B_{R'})\qquad \mbox{for any $1 \leq q < \frac{n p}{n - p}$}.
    \]
Since $\frac{n\gamma}{n(\gamma - 1) + \gamma} < p$, we have
    \[
 \frac{n p}{n - p} > \frac{\gamma}{\gamma - 1} = \gamma',
    \]
    where $\gamma' > 1$ is the H\"older conjugate of $\gamma$, so that
    \begin{equation*}
        W^{1,p}(B_{R'}) \Subset L^{\gamma'}(B_{R'}).
    \end{equation*}
    Thus, by compactness, there exists $g \in L^{\gamma'}(B_{R'})$ such that, up to a subsequence,
    \[
  \Phi_\alpha * (\mathds{1}_{B_R} \varrho_m ) \longrightarrow g
 \qquad \mbox{in $L^{\gamma'}(B_{R'})$ as $m \to \infty$}. 
    \]
   
\medskip    
\noindent   
{\bf Case 2. $\alpha \in [n-1,n)$}.
In order to prove the local convergence of the potential in the more singular case, 
we need to consider fractional Sobolev spaces.
We take $0 < s < \mu < n - \alpha \leq 1$ and 
$\frac{n\gamma}{n(\gamma - 1) + s\gamma} < p < \frac{n \gamma}{n - \gamma(n - \alpha - s)}$. 
Then, for $R,R'>0$ and $\varrho \in L^\gamma(\mathbb{R}^n)$, by the HLS inequality, we have
 \begin{equation}\label{simpfrac}
 \|\Phi_\alpha*(\mathds{1}_{B_R}\varrho)\|_{L^p(B_{R'})}
 \lesssim \|\varrho\|_{L^\frac{np}{n + p(n-\alpha)}(B_R)}.
 \end{equation}
 Since $p < \frac{n \gamma}{n - \gamma(n - \alpha - s)}$, 
 then $\frac{np}{n + p(n-\alpha)} < \gamma$. 
 Therefore, combining this with \eqref{simpfrac} and using H\"older's inequality, we obtain
  \[ \|\Phi_\alpha*(\mathds{1}_{B_R}\varrho)\|_{L^p(B_{R'})} \lesssim R^\frac{\gamma(n + p(n-\alpha)) - np}{p\gamma} \|\varrho\|_{L^\gamma}.
 \]
We now consider the semi-norm for $W^{s,p}(B_{R'})$, as defined in Definition \ref{fraccy}, given by
\begin{align}\label{seminorm}
 \begin{split}
 [\Phi_\alpha*(\mathds{1}_{B_R}\varrho)]_{W^{s,p}(B_{R'})}^p & = \int_{B_{R'}} \int_{B_{R'}} \frac{|\Phi_\alpha*(\mathds{1}_{B_R}\varrho)(\boldsymbol{x}) - \Phi_\alpha*(\mathds{1}_{B_R}\varrho)(\boldsymbol{y})|^p}{|\boldsymbol{x} - \boldsymbol{y}|^{n + sp}} \dd \boldsymbol{y} \dd \boldsymbol{x} \\
 & = \int_{B_{R'}} \int_{B_{R'}} \Big | \int_{B_R} \frac{|\boldsymbol{x} - \boldsymbol{z}|^{-\alpha} - |\boldsymbol{y} - \boldsymbol{z}|^{-\alpha}}{\alpha |\boldsymbol{x} - \boldsymbol{y}|^{\frac{n}{p} + s}} \varrho(\boldsymbol{z}) \dd \boldsymbol{z} \Big|^p \dd \boldsymbol{y} \dd \boldsymbol{x}.
\end{split}
\end{align}
Considering the difference of the two Riesz potentials in \eqref{seminorm}, we have
\begin{align}\label{fracfirst}
 \begin{split}
 \Big| \frac{|\boldsymbol{x} - \boldsymbol{z}|^{-\alpha} - |\boldsymbol{y} - \boldsymbol{z}|^{-\alpha}}{\alpha} \Big|
 & = \Big| \frac{1}{\alpha} \int^1_0 \frac{\d}{\d t} \big ( t |\boldsymbol{x} - \boldsymbol{z}| + (1-t) |\boldsymbol{y} - \boldsymbol{z}| \big )^{- \alpha} \dd t \Big| \\
 & = \Big| \int^1_0 \frac{|\boldsymbol{x} - \boldsymbol{z}| - |\boldsymbol{y} - \boldsymbol{z}|}{\big ( t |\boldsymbol{x} - \boldsymbol{z}| + (1-t) |\boldsymbol{y} - \boldsymbol{z}| \big )^{\alpha + 1}} \dd t \Big| \\
& \leq  \int^1_0 \frac{|\boldsymbol{x} - \boldsymbol{y}|}{\big ( t |\boldsymbol{x} - \boldsymbol{z}| + (1-t) |\boldsymbol{y} - \boldsymbol{z}| \big )^{\alpha + 1}} \dd t.
\end{split}
\end{align}
        
For $\boldsymbol{x},\boldsymbol{y},\boldsymbol{z} \in \mathbb{R}^n$, then at least one of the following
inequalities holds:
$$
|\boldsymbol{x} - \boldsymbol{y}| \leq 2 |\boldsymbol{x} - \boldsymbol{z}|
\qquad\mbox{or}\qquad |\boldsymbol{x} - \boldsymbol{y}| \leq 2 |\boldsymbol{y} - \boldsymbol{z}|.
$$ 
Indeed, if the first inequality fails, then
$$
|\boldsymbol{x} - \boldsymbol{y}| > 2 |\boldsymbol{x} - \boldsymbol{z}|.
$$ 
By the triangle inequality, we have
\[
 |\boldsymbol{x} - \boldsymbol{y}| \leq |\boldsymbol{x} - \boldsymbol{z}| + |\boldsymbol{y} - \boldsymbol{z}| < \frac{1}{2} |\boldsymbol{x} - \boldsymbol{y}| + |\boldsymbol{y} - \boldsymbol{z}|,
 \]
which implies
$$
|\boldsymbol{x} - \boldsymbol{y}| \leq 2 |\boldsymbol{y} - \boldsymbol{z}|.
$$
Thus, one of the two inequalities must hold.

In particular, if $|\boldsymbol{x} - \boldsymbol{z}| \leq |\boldsymbol{y} - \boldsymbol{z}|$, then $$
|\boldsymbol{x} - \boldsymbol{y}| \leq 2 |\boldsymbol{y} - \boldsymbol{z}|.
$$
Combining this with \eqref{fracfirst}, we obtain
 \begin{align}\label{fracsecond}
 \begin{split}
 \Big| \frac{|\boldsymbol{x} - \boldsymbol{z}|^{-\alpha} - |\boldsymbol{y} - \boldsymbol{z}|^{-\alpha}}{\alpha} \Big| & \leq |\boldsymbol{x} - \boldsymbol{y}|^\mu |\boldsymbol{x} - \boldsymbol{z}|^{-\mu - \alpha} \int^1_0 \Big ( \frac{1-t}{2}\Big )^{\mu - 1} \dd t \\
 & = \frac{2^{1-\mu}}{\mu}|\boldsymbol{x} - \boldsymbol{y}|^\mu |\boldsymbol{x} - \boldsymbol{z}|^{-\mu - \alpha}.
  \end{split}
 \end{align}
A symmetric estimate holds for $|\boldsymbol{y} - \boldsymbol{z}| \leq |\boldsymbol{x} - \boldsymbol{z}|$. 

We therefore define the set:
 \[
 B^{(\boldsymbol{x},\boldsymbol{y})}_R := B_R 
 \cap \{ \boldsymbol{z} \in \mathbb{R}^n \,\,: \, |\boldsymbol{x} - \boldsymbol{z}| \leq |\boldsymbol{y} - \boldsymbol{z}|\}.  \]
Thus, combining \eqref{seminorm} with \eqref{fracsecond}, we have
    \begin{align}\label{seminorm1}
 \begin{split}
  & [\Phi_\alpha*(\mathds{1}_{B_R}\varrho)]_{W^{s,p}(B_{R'})}^p \\
 & \leq \int_{B_{R'}} \int_{B_{R'}} |\boldsymbol{x} - \boldsymbol{y}|^{(\mu -s)p - n} \Big( \Big( \int_{B^{(\boldsymbol{x},\boldsymbol{y})}_R} + \int_{B^{(\boldsymbol{y},\boldsymbol{x})}_R} \Big) \Big| \frac{|\boldsymbol{x} - \boldsymbol{z}|^{-\alpha} - |\boldsymbol{y} - \boldsymbol{z}|^{-\alpha}}{\alpha |\boldsymbol{x} - \boldsymbol{y}|^{\mu}} \Big| \varrho(\boldsymbol{z}) \dd \boldsymbol{z} \Big)^p \dd \boldsymbol{y} \dd \boldsymbol{x} \\
 & \lesssim \int_{B_{R'}} \int_{B_{R'}} |\boldsymbol{x} - \boldsymbol{y}|^{(\mu -s)p - n} 
 \Big( \int_{B_R} \big (|\boldsymbol{x} - \boldsymbol{z}|^{-\mu - \alpha} + |\boldsymbol{y} - \boldsymbol{z}|^{-\mu - \alpha} \big )\varrho(\boldsymbol{z}) \dd \boldsymbol{z} \Big)^p \dd \boldsymbol{y} \dd \boldsymbol{x} \\
 & = \int_{B_{R'}} \int_{B_{R'}} |\boldsymbol{x} - \boldsymbol{y}|^{(\mu -s)p - n} \big ( |\cdot|^{-\mu - \alpha}*(\mathds{1}_{B_R}\varrho) (\boldsymbol{x}) + |\cdot|^{-\mu - \alpha}*(\mathds{1}_{B_R}\varrho) (\boldsymbol{y}) \big )^p \dd \boldsymbol{y} \dd \boldsymbol{x}.
 \end{split}
 \end{align}
 Applying Minkowski's inequality and Tonelli's theorem to \eqref{seminorm1}, we obtain
\begin{align}\label{seminorm2}
 \begin{split}
 & [\Phi_\alpha*(\mathds{1}_{B_R}\varrho)]_{W^{s,p}(B_{R'})}^p \\
& \lesssim \bigg ( \Big( \int_{B_{R'}} \int_{B_{R'}} |\boldsymbol{x} - \boldsymbol{y}|^{(\mu -s)p - n} \big ( |\cdot|^{-\mu - \alpha}*(\mathds{1}_{B_R}\varrho) (\boldsymbol{x}) \big )^p \dd \boldsymbol{y} \dd \boldsymbol{x} \Big)^\frac{1}{p} \\
 & \qquad\,\, + \Big( \int_{B_{R'}} \int_{B_{R'}} |\boldsymbol{x} - \boldsymbol{y}|^{(\mu -s)p - n} \big ( |\cdot|^{-\mu - \alpha}*(\mathds{1}_{B_R}\varrho) (\boldsymbol{y}) \big )^p \dd \boldsymbol{y} \dd \boldsymbol{x} \Big)^\frac{1}{p} \bigg )^p \\
 & = 2^p \int_{B_{R'}} \int_{B_{R'}} |\boldsymbol{x} - \boldsymbol{y}|^{(\mu -s)p - n} \big ( |\cdot|^{-\mu - \alpha}*(\mathds{1}_{B_R}\varrho) (\boldsymbol{y}) \big )^p \dd \boldsymbol{y} \dd \boldsymbol{x} \\
 & \leq 2^p \int_{B_{R'}} |\cdot|^{(\mu -s)p - n} * \big (|\cdot|^{-\mu - \alpha}*(\mathds{1}_{B_R}\varrho) \big )^p \dd \boldsymbol{x}.
\end{split}
\end{align}
Now, taking $1 < r \leq \frac{n \gamma}{p(n - \gamma(n-\alpha-s))}$, then
$\frac{n\gamma}{n(\gamma - 1) + s\gamma} < pr \leq \frac{n \gamma}{n - \gamma(n - \alpha - s)}$. Applying H\"older's inequality to the last integral in \eqref{seminorm2} and applying the HLS inequality twice, we have
 \begin{align}\label{seminorm3}
 \begin{split}
[\Phi_\alpha*(\mathds{1}_{B_R}\varrho)]_{W^{s,p}(B_{R'})}^p & \lesssim (R')^\frac{n(r-1)}{r} \big \| \, |\cdot|^{(\mu -s)p - n} * \big (|\cdot|^{-\mu - \alpha}*(\mathds{1}_{B_R}\varrho) \big )^p \, \big \|_{L^r(\mathbb{R}^n)} \\
 & \lesssim (R')^\frac{n(r-1)}{r} \big \| \, |\cdot|^{-\mu - \alpha}*(\mathds{1}_{B_R}\varrho) \, \big \|_{L^\frac{npr}{n + pr(\mu -s)}(\mathbb{R}^n)}^p \\
 & \lesssim (R')^\frac{n(r-1)}{r} \| \varrho \|_{L^\frac{npr}{n + pr(n - \alpha - s)}(B_R)}^p.
 \end{split}
\end{align}
 Notice that both the applications of the HLS inequality are well-defined. That is,
 under the conditions: $0 < n - (\mu -s)p$ and $\mu + \alpha < n$, using 
 $\frac{n\gamma}{n(\gamma - 1) + s\gamma} < pr$ and $\frac{2n}{2n - \alpha} < \gamma < \frac{n}{n-\alpha}$, we obtain
 \[
 1 < \frac{n\gamma}{(2n - \alpha)\gamma - n} < \frac{npr}{n + pr(n - \alpha - s)} < \frac{npr}{n + pr(\mu -s)}.
\] 
Moreover, Since $pr \leq \frac{n \gamma}{n - \gamma(n - \alpha - s)}$, 
it follows that $\frac{npr}{n + pr(n - \alpha - s)} \leq \gamma$. 
Combining this with \eqref{seminorm3} and applying H\"older's inequality, we deduce
 \begin{equation*}
 [\Phi_\alpha*(\mathds{1}_{B_R}\varrho)]_{W^{s,p}(B_{R'})}^p \lesssim (R')^\frac{n(r-1)}{r} R^\frac{n + pr(n - \alpha - s)}{r} \| \varrho \|_{L^\gamma}^p.
\end{equation*} 
Therefore, since $\varrho_m$ is a bounded sequence in $L^\gamma(\mathbb{R}^n)$, 
then $\Phi_\alpha * (\mathds{1}_{B_R} \varrho_m ) \in W^{s,p}(B_{R'})$ is 
a bounded sequence. 
For fractional Sobolev spaces, since $sp < n$, by Lemma \ref{compactfraccy}, we have
 \[
  W^{s,p}(B_{R'}) \Subset L^q(B_{R'})\qquad \mbox{for any $1 \leq q < \frac{n p}{n - sp}$}. 
 \]
Since $\frac{n\gamma}{n(\gamma - 1) + s\gamma} < p$, we have
    \[
  \frac{n p}{n - sp} > \frac{\gamma}{\gamma - 1} = \gamma',
    \]
so that
    \begin{equation*}
  W^{s,p}(B_{R'}) \Subset L^{\gamma'}(B_{R'}).
    \end{equation*}
    Thus, by compactness, there exists $g \in L^{\gamma'}(B_{R'})$ such that, up to a subsequence,
    \[ \Phi_\alpha * (\mathds{1}_{B_R} \varrho_m ) \longrightarrow g,
    \]
    in $L^{\gamma'}(B_{R'})$ as $m \to \infty$.
    
\bigskip
\noindent
{\bf 2.}
Notice that $\varrho_m \rightharpoonup \tilde{\varrho}$ in $L^\gamma(\mathbb{R}^n)$
as $m \to \infty$, for $\gamma > \frac{2n}{2n - \alpha}$. 
Let $f \in L^{\gamma}(B_{R'})$, extended by zero outside $B_{R'}$.
By the HLS inequality and the fact that $B_R$ has finite measure, 
we have
$$
\Phi_\alpha * f \in L^\frac{n\gamma}{n - \gamma (n-\alpha)}(B_R) \subset L^{\gamma'}(B_R).
$$
Using Fubini's theorem, we obtain
    \[
 \int_{B_{R'}} f \Phi_\alpha * (\mathds{1}_{B_R}\varrho_m) \dd \boldsymbol{x} = \int_{B_{R}} \varrho_m  \Phi_\alpha * f \dd \boldsymbol{x} \longrightarrow \int_{B_{R}} \tilde{\varrho}  \Phi_\alpha * f \dd \boldsymbol{x} = \int_{B_{R'}} f \Phi_\alpha * (\mathds{1}_{B_R}\tilde{\varrho}) \dd \boldsymbol{x},
\]
as $m \to \infty$. 
It follows that
$$
\Phi_\alpha * (\mathds{1}_{B_R}\varrho_m) \rightharpoonup \Phi_\alpha * (\mathds{1}_{B_R}\tilde{\varrho})
\qquad \mbox{in $L^{\gamma'}(B_{R'})$  as $m \to \infty$.}
$$
By uniqueness of weak limits,  we coclude that
$$
g \equiv \Phi_\alpha * (\mathds{1}_{B_R}\tilde{\varrho}) \qquad\mbox{in $L^{\gamma'}(\R^n)$}.
$$
Therefore, up to a subsequence,
    \[
  \Phi_\alpha * (\mathds{1}_{B_R}\varrho_m) \longrightarrow \Phi_\alpha * (\mathds{1}_{B_R}\tilde{\varrho})\qquad\mbox{strongly in $L^{\gamma'}(B_{R'})$ as $m \to \infty$}.
    \]
Consequently, again up to a subsequence,
    \begin{equation}\label{localcon}
 \int_{B_{R'}} \varrho_m \Phi_\alpha * (\mathds{1}_{B_R}\varrho_m) \dd \boldsymbol{x} \longrightarrow \int_{B_{R'}} \tilde{\varrho} \Phi_\alpha * (\mathds{1}_{B_R}\tilde{\varrho}) \dd \boldsymbol{x}
 \qquad\mbox{as $m \to \infty$.}
    \end{equation}
  
\bigskip
\noindent
{\bf 3.}
    Using that $\varrho_m$ are bounded in $L^1(\mathbb{R}^n)$ and taking $R' > R > 0$, we estimate
    \begin{align}\label{R'R}
\begin{split}
  \Big|\int_{B_{R'}^{\rm c}} \varrho_m \Phi_\alpha * (\mathds{1}_{B_R}\varrho_m) \dd \boldsymbol{x}\Big| 
 & \lesssim \frac{1}{(R' - R)^\alpha} \int_{B_{R'}^{\rm c}} |\varrho_m(\boldsymbol{x})| \int_{B_{R}} |\varrho_m(\boldsymbol{y})| \dd \boldsymbol{y} \dd \boldsymbol{x} \\
  & \lesssim \frac{1}{(R' - R)^\alpha}.
 \end{split}
    \end{align}
Consequently, by choosing $R'>0$ sufficiently large compared to $R>0$ in \eqref{R'R} and combining 
this estimate with \eqref{localcon}, we obtain
    \begin{equation}\label{lesslocalcon}
 \int_{\mathbb{R}^n} \varrho_m \Phi_\alpha * (\mathds{1}_{B_R}\varrho_m) \dd \boldsymbol{x} \longrightarrow \int_{\mathbb{R}^n} \tilde{\varrho} \Phi_\alpha * (\mathds{1}_{B_R}\tilde{\varrho}) \dd \boldsymbol{x}\qquad \mbox{as $m \to \infty$}.
    \end{equation}
    
\bigskip
\noindent
{\bf 4.}
Since $\varrho_m$ is radial, we can use the radial representation formula with its bounds 
as in Lemma \ref{potentially} to obtain
    \[
 (\Phi_\alpha * \varrho_m)(\boldsymbol{x}) = \int^\infty_0 K(|\boldsymbol{x}|,\eta) \varrho_m(\eta) \eta^{n-1} \dd \eta,
    \]
    where
    \[
 K(r,\eta) := -\int_{\mathbb{S}^{n-1}} \frac{|r\boldsymbol{e_1}-\eta \boldsymbol{y}|^{-\alpha}}{-\alpha} \dd S (\boldsymbol{y})
    \]
    satisfies that, for $\alpha \neq n-1$,
    \[
 |K(r,\eta)| \lesssim
        \begin{cases}
\displaystyle
(r\eta)^{-\frac{\alpha}{2}} & \mbox{ for $\alpha \in(0,n-1)$}, \\[1mm]
 \displaystyle
 (r\eta)^{-\frac{n-1}{2}} |r-\eta|^{n-1-\alpha}
 & \mbox{ for $\alpha \in (n-1,n)$}.
\end{cases}
    \]
   
   \smallskip
   \noindent
   {\bf Case 1.} $\alpha \in (0,n-1)$. Taking $R,R'>0$ and using that $\varrho_m$ are bounded in $L^1(\mathbb{R}^n)$, we have
    \begin{align}\label{biggRs1}
 \begin{split}
  \Big| \int_{B_{R'}^{\rm c}} \varrho_m \Phi_\alpha * (\mathds{1}_{B_R^{\rm c}}\varrho_m) \dd \boldsymbol{x} \Big| 
 & = \Big| \int^\infty_{R'} \varrho_m(r) r^{n-1} \int^\infty_R K(r,\eta) \varrho_m(\eta) \eta^{n-1} \dd \eta \dd r \Big| \\
 & \lesssim \frac{1}{(R'R)^\frac{\alpha}{2}} \int^\infty_{R'} |\varrho_m(r)| r^{n-1} \int^\infty_R |\varrho_m(\eta)| \eta^{n-1} \dd \eta \dd r \\
 & \lesssim \frac{1}{(R'R)^\frac{\alpha}{2}}.
 \end{split}
    \end{align}
      
  \smallskip
  \noindent
  {\bf Case 2}. $\alpha \in (n-1,n)$.
  Taking $R,R'>0$, and using that $\varrho_m$ are bounded in $L^1 \cap L^\gamma (\mathbb{R}^n)$, the HLS inequality, and $\frac{2}{n + 1 - \alpha} < \frac{2n}{2n - \alpha} < \gamma$, we obtain
    \begin{align}\label{biggRs2}
      \begin{split}
 \Big| \int_{B_{R'}^{\rm c}} \varrho_m \Phi_\alpha * (\mathds{1}_{B_R^{\rm c}}\varrho_m) \dd \boldsymbol{x} \Big|
 & = \Big| \int^\infty_{R'} \varrho_m(r) r^{n-1} \int^\infty_R K(r,\eta) \varrho_m(\eta) \eta^{n-1} \dd \eta \dd r \Big| \\
 & \lesssim \int^\infty_{R'} |\varrho_m(r)| r^{\frac{n-1}{2}} \int^\infty_R |r-\eta|^{n-1-\alpha} |\varrho_m(\eta)| \eta^{\frac{n-1}{2}} \dd \eta \dd r \\
 & \lesssim \Big( \int^\infty_{R'} |\varrho_m(r)|^\frac{2}{n + 1 - \alpha} r^{\frac{n-1}{n + 1 - \alpha}} \int^\infty_R |\varrho_m(\eta)|^\frac{2}{n + 1 - \alpha} \eta^{\frac{n-1}{n + 1 - \alpha}} \dd \eta \dd r \Big)^\frac{n + 1 - \alpha}{2} \\
 & \leq \frac{1}{(R'R)^\frac{(n-1)(n-\alpha)}{n + 1 - \alpha}} \|\varrho_m\|_{L^\frac{2}{n + 1 - \alpha}(\mathbb{R}^n)}^2 \\
 & \leq \frac{1}{(R'R)^\frac{(n-1)(n-\alpha)}{n + 1 - \alpha}} \|\varrho_m\|_{L^1}^\frac{\gamma(n + 1 - \alpha) - 2}{\gamma - 1} \|\varrho_m\|_{L^\gamma}^\frac{\gamma(\alpha - (n-1))}{\gamma - 1} \\
 & \lesssim \frac{1}{(R'R)^\frac{(n-1)(n-\alpha)}{n + 1 - \alpha}}.
 \end{split}
    \end{align}

\smallskip
\noindent
{\bf Case 3}. $\alpha = n-1$. 
Taking $R,R'>0$, by considering the behavior of $\Phi_\alpha$ around and away from 
the origin separately, we obtain control of the potential energy, 
by the use of less and more singular Riesz potentials. 
That is, given $0 < \delta < 1$, we have
 \begin{equation}\label{split}
 \Big| \int_{B_{R'}^{\rm c}} \varrho_m \Phi_{n-1} * (\mathds{1}_{B_R^{\rm c}}\varrho_m) \dd \boldsymbol{x} \Big| \lesssim \Big| \int_{B_{R'}^{\rm c}} \varrho_m \Phi_{n-1- \delta} * (\mathds{1}_{B_R^{\rm c}}\varrho_m) \dd \boldsymbol{x} \Big| 
  + \Big| \int_{B_{R'}^{\rm c}} \varrho_m \Phi_{n - 1 + \delta} * (\mathds{1}_{B_R^{\rm c}}\varrho_m) 
 \dd \boldsymbol{x} \Big|.
 \end{equation}
Thus, taking $0 < \delta < \frac{n-1}{n} = \frac{\alpha}{n}$, 
we see that $\frac{2}{2- \delta} < \frac{2n}{n + 1} = \frac{2n}{2n - \alpha}$. 
Combining this with \eqref{biggRs1}--\eqref{split}, we have
 \begin{equation}\label{biggRs3}
  \Big| \int_{B_{R'}^{\rm c}} \varrho_m \Phi_{n-1} * (\mathds{1}_{B_R^{\rm c}}\varrho_m) \dd \boldsymbol{x} \Big| \lesssim \frac{1}{(R'R)^\frac{n - 1 - \delta}{2}} + \frac{1}{(R'R)^\frac{(n-1)(1-\delta)}{2 - \delta}}.
  \end{equation}
   
  \bigskip     
 \noindent
 {\bf 5.}
For $R > R' > 0$, as in \eqref{R'R}, 
using that $\varrho_m$ are bounded in $L^1(\mathbb{R}^n)$ and Fubini's theorem, 
we obtain by symmetry that
\begin{equation}\label{R'Rother}
\Big| \int_{B_{R'}} \varrho_m \Phi_\alpha * (\mathds{1}_{B_R^{\rm c}}\varrho_m) \dd \boldsymbol{x} \Big| \lesssim \frac{1}{(R - R')^\alpha}.
    \end{equation}
Combining \eqref{lesslocalcon}--\eqref{biggRs2} with \eqref{biggRs3}--\eqref{R'Rother} and 
taking $R > R' > 0$ large with $R>0$ sufficiently larger than $R'>0$, we obtain the desired result.
 \end{proof}
    \setcounter{step}{0}
    \setcounter{case}{0}

\begin{lemma}\label{minexist}
Let $n \geq 2$, $\alpha \in (0,n)$, 
and $\frac{2n}{2n - \alpha} < \gamma < \frac{n + \alpha}{n}$. 
Then  $\ell_\mu > 0$, and there exists a minimizer of $S_\mu$ over $\mathcal{K}$.
\end{lemma}

\begin{proof} We divide the proof into two steps:

\medskip
\noindent
{\bf 1}. Suppose that $\varrho_m \in \mathcal{K}$ is a minimizing sequence of $S_\mu$, {\it i.e.,} 
    \[
    Q(\varrho_m) = 0,
    \]
    and
    \[
    \lim_{m \to \infty} S_{\mu}(\varrho_m) = \ell_\mu.
    \]
 We aim to reduce to a minimizing sequence consisting of spherically symmetric function. 
 To this end, we first take the spherically decreasing rearrangement to define $w_m = \varrho_m^\#$. 
 However, it is not immediate that $w_m \in \mathcal{K}$.
 We therefore introduce the rescale functions:
 $\tilde{w}_m = (w_m)_{\lambda^*(w_m)}$. 
 By Lemma \ref{facts} (b) and (c), 
we see that $\tilde{w}_m = ((\varrho_m)_{\lambda^*(w_m)})^\#$, so that
$\tilde{w}_m \in \mathcal{K}$ are spherically symmetric. 
Moreover, using Lemma \ref{facts} (f) and the same argument as in Lemma \ref{symm}, 
we see that $\tilde{w}_m$ is still a minimizing sequence, since
    \[
    \ell_\mu \leq S_\mu(\tilde{w}_m) \leq S_\mu((\varrho_m)_{\lambda^*(w_m)}) \leq S_\mu(\varrho_m).
    \]
Thus, without loss of generality, we may assume that $\{\varrho_m\}_m$ is a sequence of spherically symmetric, radially decreasing, and non-negative functions. 

From \eqref{overk}, we obtain the uniform bound:
    \begin{equation}\label{seqbound}
 \| \varrho_m \|_{L^\gamma} + \| \varrho_m \|_{L^1} \lesssim 1.
    \end{equation}
  Since $\varrho_m \in \mathcal{K}$, we also have
    \begin{equation}\label{sequenceHLS}
 \| \varrho_m \|_{L^\gamma}^\gamma = - \frac{\alpha}{2 n a_0} \int_{\mathbb{R}^n} \varrho_m \Phi_\alpha * \varrho_m \dd \boldsymbol{x} \lesssim \| \varrho_m \|_{L^1}^\frac{\gamma(2n - \alpha) - 2n}{n(\gamma - 1)} \| \varrho_m \|_{L^\gamma}^\frac{\gamma \alpha}{n(\gamma - 1)}.
    \end{equation}
Combining \eqref{seqbound}, \eqref{sequenceHLS}, and the condition:
    $\frac{2n}{2n - \alpha} < \gamma < \frac{n + \alpha}{n}$, 
we deduce 
    \[
 \| \varrho_m \|_{L^\gamma}^\frac{\gamma(n(\gamma - 1) - \alpha)}{n(\gamma - 1)} \lesssim \| \varrho_m \|_{L^1}^\frac{\gamma(2n - \alpha) - 2n}{n(\gamma - 1)} \lesssim 1,
    \]
so that
    \begin{equation}\label{lowerseq}
  \| \varrho_m \|_{L^\gamma} \gtrsim 1.
    \end{equation}
It follows from \eqref{overk} that $\ell_\mu > 0$. By Lemma \ref{weaklowersemi}, we conclude that \eqref{seqbound} 
    implies the existence of $\tilde{\varrho} \in L^\gamma_{\text{rad}}(\mathbb{R}^n)$ 
    such that, up to a subsequence, 
    $\varrho_m \rightharpoonup \tilde{\varrho}$ as $m \to \infty$ in $L^\gamma(\mathbb{R}^n)$. Since $\frac{n+\alpha}{n} < \frac{n}{n-\alpha}$, Lemma \ref{weaktostrong} implies that,
    up to a subsequence,
    \begin{equation}\label{potentialstrong}
 \int_{\mathbb{R}^n} \varrho_m \Phi_\alpha * \varrho_m \dd \boldsymbol{x} \longrightarrow \int_{\mathbb{R}^n} \tilde{\varrho} \Phi_\alpha * \tilde{\varrho} \dd \boldsymbol{x}
      \qquad\mbox{as $m \to \infty$.}  
    \end{equation}
    Combining \eqref{sequenceHLS} with \eqref{lowerseq}, we obtain 
    \[
    - \int_{\mathbb{R}^n} \varrho_m \Phi_\alpha * \varrho_m \dd \boldsymbol{x} \gtrsim 1.
    \]
Passing to the limit in view of \eqref{potentialstrong}, we deduce
    \[
    - \int_{\mathbb{R}^n} \tilde{\varrho} \Phi_\alpha * \tilde{\varrho} \dd \boldsymbol{x} \gtrsim 1,
    \]
   which implies that $\tilde{\varrho} \not \equiv 0$.

 \medskip   
 \noindent
 {\bf 2.}
 On the other hand, $\tilde{\varrho}$ may not satisfy the constraint: $Q(\tilde{\varrho}) = 0$. 
 Therefore, we take $\bar{\rho} := \tilde{\varrho}_{\lambda^*(\tilde{\varrho})} \in \mathcal{K}$. 
 By Lemma \ref{facts} (a), it follows that,  up to a subsequence, 
 $\varrho_m^* := (\varrho_m)_{\lambda^*(\tilde{\varrho})} \rightharpoonup \bar{\rho}$ 
 as $m \to \infty$ in $L^\gamma(\mathbb{R}^n)$. 
 From Lemma \ref{weaklowersemi}, we see that $\tilde{\varrho} \geq 0$ {\it a.e.} and
    \[
    \| \bar{\rho} \|_{L^\gamma} + \| \bar{\rho} \|_{L^1} \lesssim 1.
    \]
    Furthermore, by Lemma \ref{weaklowersemi} and the weak lower semi-continuity 
   of the map: $\rho\mapsto \int_{\mathbb{R}^n} \varrho^\gamma \dd \boldsymbol{x}$, we have
    \[
    \int_{\mathbb{R}^n} \bar{\rho} \dd \boldsymbol{x} \leq \liminf_{m \to \infty} \int_{\mathbb{R}^n} \varrho_m^* \dd \boldsymbol{x}, \qquad \,\,
    \int_{\mathbb{R}^n} (\bar{\rho})^\gamma \dd \boldsymbol{x} \leq \liminf_{m \to \infty} \int_{\mathbb{R}^n} (\varrho_m^*)^\gamma \dd \boldsymbol{x}.
    \]
Consequently, using Lemma \ref{facts} (f), we obtain 
    $$
    S_\mu(\bar{\rho}) \leq \liminf_{m \to \infty} S_\mu(\varrho_m^*) \leq \liminf_{m \to \infty} S_\mu(\varrho_m) = \ell_\mu.
    $$
Therefore, $\bar{\rho}$ is a spherically symmetric minimizer.
    \end{proof}
    \setcounter{step}{0}

\begin{theorem}
Suppose there exists an energy minimizer $\bar{\rho}$ of the energy functional $S_\mu$ 
over $\mathcal{K}$, and define
\[
 \Gamma := \big\{ \boldsymbol{x} \in \mathbb{R}^n \, : \, \bar{\rho}(\boldsymbol{x}) > 0\big\}.
\]
Then $\bar{\rho}$ is a steady state and satisfies
\begin{align*}
\begin{cases}
\frac{\dd}{\dd \rho}(\rho e(\rho))|_{\rho=\bar{\rho}}+\Phi_\alpha\ast\bar{\rho}= - \mu
 \quad  &\mbox{for $\boldsymbol{x}\in \Gamma$},\\[1mm]
\Phi_\alpha\ast\bar{\rho}\geq - \mu &\mbox{for $\boldsymbol{x}\in \R^n\setminus\Gamma$}.
\end{cases}
\end{align*}
Moreover, when $\alpha \in (0,n) \cap [n-2,n)$, $\bar{\rho}$ coincides, up to translation, 
with $\bar{\rho}_\mu$.
\end{theorem}

\begin{proof} We divide the proof into two steps:
 
\bigskip\noindent
{1.} We derive the Euler-Lagrange equations for the minimization problem. 
For $\varepsilon > 0$, define
    \[\Gamma_\varepsilon := \big\{ \boldsymbol{x} \in \mathbb{R}^n \, : \, \varepsilon < \bar{\rho}(\boldsymbol{x}) < \frac{1}{\varepsilon} \big\}.
    \]
Let $w \in L^\infty(\mathbb{R}^n)$ have compact support and satisfy $w\ge 0$ 
on $\mathbb{R}^n \setminus \Gamma_\varepsilon$. 
For $\tau > 0$ sufficiently small, set  
    \[
    \bar{\rho}_\tau := \bar{\rho} + \tau w
    \]
so that $\bar{\rho}_\tau \geq 0$ {\it a.e.}, 
    $\bar{\rho}_\tau \not \equiv 0$, and 
    $\bar{\rho}_\tau \in L^1 \cap L^\gamma (\mathbb{R}^n)$.
    By Lemma \ref{facts} (c), there exists a unique
    \[
 \lambda^*(\bar{\rho}_\tau) = \bigg ( -\frac{2na_0 \int_{\mathbb{R}^n} \bar{\rho}_\tau^\gamma \dd \boldsymbol{x}}{\alpha \int_{\mathbb{R}^n} \bar{\rho}_\tau \Phi_\alpha * \bar{\rho}_\tau \dd \boldsymbol{x}} \bigg )^\frac{n}{\alpha - n(\gamma - 1)},
    \]
    such that $(\bar{\rho}_\tau)_{\lambda^*(\bar{\rho}_\tau)} \in \mathcal{K}$. 
Since $\bar{\rho}$ is a minimizer of $S_\mu$ over $\mathcal{K}$, 
using the simplified form of $S_\mu(\rho)$ for $\rho \in \mathcal{K}$ in \eqref{overk},
we obtain
    \begin{align}\label{comparemin}
 \begin{split}
  0 \leq S_\mu((\bar{\rho}_\tau)_{\lambda^*(\bar{\rho}_\tau)}) - S_\mu(\bar{\rho})
 & = a_0(\lambda^*(\bar{\rho}_\tau))^{\gamma - 1} \frac{\alpha - n(\gamma - 1)}{\alpha(\gamma - 1)}\int_{\mathbb{R}^n} \bar{\rho}_\tau^\gamma \dd \boldsymbol{x} \\
  & \quad\, + \tau \mu \int_{\mathbb{R}^n} w \dd \boldsymbol{x}
 - a_0 \frac{\alpha - n(\gamma - 1)}{\alpha(\gamma - 1)}\int_{\mathbb{R}^n} \bar{\rho}^\gamma \dd \boldsymbol{x}.
  \end{split}
    \end{align}
Using
Lemma \ref{facts} (e) and the Taylor expansion, 
we have
    \begin{align*}
  \begin{split}
      (\lambda^*(\bar{\rho}_\tau))^{\gamma - 1}
 & = (\lambda^*(\bar{\rho}))^{\gamma - 1} \\
 &\quad\,- \tau \frac{2n^2(\gamma - 1)}{\alpha(\alpha - n(\gamma - 1))} \Bigg ( \frac{a_0 \gamma \int_{\mathbb{R}^n} \bar{\rho}^{\gamma - 1} w \dd \boldsymbol{x}}{\int_{\mathbb{R}^n} \bar{\rho} \Phi_\alpha * \bar{\rho} \dd \boldsymbol{x}} 
 + \frac{2a_0 \int_{\mathbb{R}^n} \bar{\rho}^\gamma \dd \boldsymbol{x}}{\big (\int_{\mathbb{R}^n} \bar{\rho} \Phi_\alpha * \bar{\rho} \dd \boldsymbol{x} \big )^2} \int_{\mathbb{R}^n} w \Phi_\alpha * \bar{\rho} \dd \boldsymbol{x}  \Bigg ) \\
 &\quad\, + o(\tau) \\
 & = 1 + \tau \frac{\gamma - 1}{\alpha - n(\gamma - 1)}\,
 \frac{n a_0 \gamma \int_{\mathbb{R}^n} \bar{\rho}^{\gamma - 1} w \dd \boldsymbol{x} + \alpha \int_{\mathbb{R}^n} w \Phi_\alpha * \bar{\rho} \dd \boldsymbol{x}}{a_0 \int_{\mathbb{R}^n} \bar{\rho}^\gamma \dd \boldsymbol{x}} + o(\tau),
\end{split}
    \end{align*}
and
    \begin{equation*}
  \int_{\mathbb{R}^n} \bar{\rho}_\tau^\gamma \dd \boldsymbol{x} = \int_{\mathbb{R}^n} \bar{\rho}^\gamma \dd \boldsymbol{x} + \tau \gamma \int_{\mathbb{R}^n} \bar{\rho}^{\gamma - 1} w \dd \boldsymbol{x} + o(\tau).
    \end{equation*}
Substituting them into \eqref{comparemin} yields
    \[
 0 \leq \tau \int_{\mathbb{R}^n} \Big ( \frac{a_0\gamma}{\gamma - 1} \bar{\rho}^{\gamma - 1} + \Phi_\alpha * \bar{\rho} + \mu \Big) w \dd \boldsymbol{x} + o(\tau).
    \]
Then the coefficient of $\tau$ must be non-negative:
    \[
 \int_{\mathbb{R}^n} \Big ( \frac{a_0\gamma}{\gamma - 1} \bar{\rho}^{\gamma - 1} + \Phi_\alpha * \bar{\rho} + \mu \Big) w \dd \boldsymbol{x} \geq 0
 \qquad\mbox{ for any such $w$}.
    \]
This implies
\[
    \begin{split}
        & \frac{a_0\gamma}{\gamma - 1} \bar{\rho}^{\gamma - 1} + \Phi_\alpha * \bar{\rho} + \mu = 0, \qquad  \text{ {\it a.e.} on } \Gamma_\varepsilon,\\
        & \frac{a_0\gamma}{\gamma - 1} \bar{\rho}^{\gamma - 1} + \Phi_\alpha * \bar{\rho} + \mu \geq 0, \quad \,\,\,\,\,\, \text{ {\it a.e.} on } \mathbb{R}^n \setminus \Gamma_\varepsilon.
    \end{split}
\]
Letting $\varepsilon \to 0$, we obtain 
\begin{equation}\label{EulerLagrange}
    \begin{split}
        & \frac{a_0\gamma}{\gamma - 1} \bar{\rho}^{\gamma - 1} + \Phi_\alpha * \bar{\rho} + \mu = 0, \qquad  \text{ {\it a.e.} on } \Gamma,\\
        & \Phi_\alpha * \bar{\rho} + \mu \geq 0, \qquad \qquad \qquad \quad \,\,\,\, \text{ {\it a.e.} on } \mathbb{R}^n \setminus \Gamma.
    \end{split}
\end{equation}
Using $\bar{\rho} \in L^1 \cap L^\gamma (\mathbb{R}^n)$ and by HLS, 
we deduce  
$$
\Phi_\alpha * \bar{\rho} \in L^p(\mathbb{R}^n)
\qquad\mbox{for $\frac{n}{\alpha} < p 
\leq \frac{n\gamma}{n - \gamma (n - \alpha)}$}.
$$. 
Since $\frac{2n}{2n - \alpha} < \gamma < \frac{n + \alpha}{n}$
implies that 
$\frac{n}{\alpha} < \frac{\gamma}{\gamma - 1} \leq \frac{n\gamma}{n - \gamma (n - \alpha)}$,
we obtain
$$
\bar{\rho}^{\gamma - 1}, \Phi_\alpha * \bar{\rho} \in L^\frac{\gamma}{\gamma - 1}(\mathbb{R}^n).
$$ 
Combining this with \eqref{EulerLagrange} yields 
$$
\mu \mathds{1}_{\Gamma} \in L^\frac{\gamma}{\gamma - 1}(\mathbb{R}^n),
$$
which implies that $|\Gamma| < \infty$. 
Since $\bar{\rho}$ is spherically decreasing, its support is compact. 

Finally, rewriting \eqref{EulerLagrange}, we obtain
    \[
    \bar{\rho} = \Big ( \frac{\gamma - 1}{a_0 \gamma} \Big )^\frac{1}{\gamma - 1} (- \Phi_\alpha * \bar{\rho} - \mu)_+^\frac{1}{\gamma - 1}.
    \]

   \medskip
   \noindent
   {\bf 2.}
    For $\alpha \in [n-2,n)$, as noted at the beginning of the section, 
    $\bar{\rho}_\mu$ is the unique spherically symmetric solution 
    of the Euler-Lagrange equation. 
    We now show that $\bar{\rho}$ must be a translation of $\bar{\rho}_\mu$.
    
    By Lemma \ref{facts} (c), we see that $(\bar{\rho}^\#)_{\lambda^*(\bar{\rho}^\#)} \in \mathcal{K}$. 
    Using Lemma \ref{facts} (b) and (f), we obtain
    \[
    \ell_\mu \leq S_\mu ((\bar{\rho}^\#)_{\lambda^*(\bar{\rho}^\#)} ) 
    = S_\mu (((\bar{\rho})_{\lambda^*(\bar{\rho}^\#)})^\#)
    \leq S_\mu( (\bar{\rho})_{\lambda^*(\bar{\rho}^\#)}) 
    \leq S_\mu ( \bar{\rho} ) = \ell_\mu.
    \]
    Thus, all inequalities must be equalities above.
    In particular,
    $$
    \lambda^*(\bar{\rho}^\#) = \lambda^*(\bar{\rho}) = 1,
    \qquad 
    S_\mu(\bar{\rho}^\#) = S_\mu(\bar{\rho}).
    $$
    It follows that $\bar{\rho}^\#$ is also minimizer, and therefore
    $$
   \bar{\rho}^\# = \bar{\rho}_\mu.
   $$ 
    Moreover, we have
    \[
    \int_{\mathbb{R}^n} \bar{\rho}^\# \Phi_\alpha * \bar{\rho}^\# \dd \boldsymbol{x} = \int_{\mathbb{R}^n} \bar{\rho} \Phi_\alpha * \bar{\rho} \dd \boldsymbol{x}.
    \]
    By Lemma \ref{symm}, we conclude 
    that $\bar{\rho}$ coincides with $\bar{\rho}^\# = \bar{\rho}_\mu$ up to translation.
    Hence, $\bar{\rho}$ is a translation of  $\bar{\rho}_\mu$. 
\end{proof}
\setcounter{step}{0}

We now give a more explicit description of $\bigcup_{\mu > 0} \mathcal{I}_\mu$.

\begin{lemma}\label{union}
For $n \geq 2$, $\alpha \in (0,n)$, and $\frac{2n}{2n - \alpha} < \gamma < \frac{n + \alpha}{n}$, 
 \[
\begin{split}
\bigcup_{\mu > 0} \mathcal{I}_\mu  
= \mathcal{I} := \Big\{ (\rho,\mathcal{M}) \, : \,\, Q(\rho) > 0, \,
\int_{\mathbb{R}^n} \rho \dd \boldsymbol{x} < C(n,\alpha,\gamma,\ell_1)                     
   E(\rho,\mathcal{M})^\frac{n\gamma - n - \alpha}{(2n - \alpha)\gamma - 2n} \Big\},
\end{split}
\]
where 
$$
C(n,\alpha,\gamma,\ell_1):=
\Big ( \frac{n + \alpha - n\gamma}{(2n-\alpha)\gamma - 2n} \Big )^\frac{n + \alpha - n\gamma}{(2n-\alpha)\gamma - 2n}
\Big ( \ell_1 \frac{(2n-\alpha)\gamma - 2n}{(n - \alpha)(\gamma - 1)} \Big )^\frac{(n - \alpha)(\gamma - 1)}{(2n-\alpha)\gamma - 2n}. 
$$    
\end{lemma}

\begin{proof}
    By the definition of $\mathcal{I}_\mu$, 
    if $(\rho,\mathcal{M}) \in \mathcal{I}_\mu$, 
    then $E(\rho,\mathcal{M}) < \ell_\mu - \mu \int_{\mathbb{R}^n} \rho \dd \boldsymbol{x}$. 
    Define
    $$
    f(\mu) := \ell_\mu - \mu \int_{\mathbb{R}^n} \rho \dd \boldsymbol{x}.
    $$
Then, if $(\rho,\mathcal{M}) \in \bigcup_{\mu > 0} \mathcal{I}_\mu$, we have
    \begin{equation}\label{muset}
        E(\rho,\mathcal{M}) < \max_{\mu > 0} f(\mu).
    \end{equation}
    By Lemma \ref{minexist}, there exists $\bar{\rho}_1 \in \mathcal{K}$ that is a minimizer 
    of $S_1$ over $\mathcal{K}$.
    By Proposition \ref{minform}, $\bar{\rho}_1$ is a solution of \eqref{station} for $\mu = 1$. 
    Using the scaling relation:
    $$
    \bar{\rho}_\mu(\boldsymbol{x}) 
    = \mu^\frac{1}{\gamma - 1}\bar{\rho}_1(\mu^\frac{2-\gamma}{(\gamma - 1)(n-\alpha)} \boldsymbol{x}),
    $$ 
    we obtain a solution of \eqref{station} for each $\mu > 0$.
    By construction, $\bar{\rho}_\mu \in \mathcal{K}$. 
    
    We now claim that $\bar{\rho}_\mu$ is a minimizer of $S_\mu$ over $\mathcal{K}$. 
    If not, then there exists $\varrho^*_\mu \in \mathcal{K}$ such that 
    $S_\mu(\varrho^*_\mu) < S_\mu(\bar{\rho}_\mu)$. 
    Define the rescaled function:
    $$
    \varrho^*_1(\boldsymbol{x}) 
    := \mu^\frac{-1}{\gamma - 1}\varrho^*_\mu 
    (\mu^\frac{\gamma - 2}{(\gamma - 1)(n-\alpha)} \boldsymbol{x}).
    $$ 
    Then $\varrho^*_1 \in \mathcal{K}$ and, by scaling
    $$
    S_1(\varrho^*_1) = \mu^\frac{2n - \gamma(2n - \alpha)}{(\gamma - 1)(n - \alpha)} S_\mu(\varrho^*_\mu) < \mu^\frac{2n - \gamma(2n - \alpha)}{(\gamma - 1)(n - \alpha)} S_\mu(\bar{\rho}_\mu) = S_1(\bar{\rho}_1),
    $$
which contradicts the minimality of $\bar{\rho}_1$. 
Hence $\bar{\rho}_\mu$ is a minimizer of $S_\mu$ over $\mathcal{K}$ so that
    \[
\ell_\mu = \mu^\frac{(2n - \alpha)\gamma - 2n}{(\gamma - 1)(n - \alpha)} \ell_1.
    \]
Then
    \[
 f(\mu) = \mu^\frac{(2n - \alpha)\gamma - 2n}{(\gamma - 1)(n - \alpha)} \ell_1 - \mu \int_{\mathbb{R}^n} \rho \dd \boldsymbol{x}.
    \]

We now determine the maximum point $\mu_0$ of $f(\mu)$. 
Solving $f'(\mu_0)=0$ yields 
    \begin{equation}\label{mu0}
 \mu_0 = \bigg ( \frac{(n-\alpha)(\gamma - 1)\int_{\mathbb{R}^n} \rho \dd \boldsymbol{x}}{((2n - \alpha)\gamma - 2n) \ell_1} \bigg )^\frac{(\gamma - 1)(n - \alpha)}{n\gamma - n - \alpha}.
    \end{equation}
Since $\frac{2n}{2n - \alpha} < \gamma < \frac{n + \alpha}{n}$, we see that $\mu_0 > 0$ and 
    $f''(\mu_0) < 0$, so $f$ attains its maximum at $\mu_0$ indeed.
Consequently, we have
\[
  \max_{\mu > 0} f(\mu) = f(\mu_0),
\]        
which gives an explicit expression in terms of $\int_{\mathbb{R}^n} \rho \dd \boldsymbol{x}$.

Rearranging and applying \eqref{muset}, together with $n(\gamma - 1) - \alpha < 0$, we deduce
    \[
 \int_{\mathbb{R}^n} \rho \dd \boldsymbol{x} < C(n,\alpha,\gamma,\ell_1)
 E(\rho,\mathcal{M})^\frac{n\gamma - n - \alpha}{(2n - \alpha)\gamma - 2n}.
    \]
Therefore, $(\rho,\mathcal{M}) \in \mathcal{I}$, which implies that
    $\bigcup_{\mu > 0} \mathcal{I}_\mu \subseteq \mathcal{I}$. 

    Conversely, if $(\rho,\mathcal{M}) \in \mathcal{I}$, then, by the above construction with $\mu_0$ defined in \eqref{mu0}, we conclude that $(\rho,\mathcal{M}) \in \mathcal{I}_{\mu_0}$, and hence
    $\mathcal{I} \subseteq \bigcup_{\mu > 0} \mathcal{I}_\mu$. 
    This completes the proof.
\end{proof}

\begin{remark}
From the above calculation, we deduce that, 
    for any $\mu > 0$ and any minimizer $\bar{\rho} \in \mathcal{K}$ of $S_\mu$ over $\mathcal{K}$, 
    \[
 E(\bar{\rho},\boldsymbol{0}) = \ell_\mu - \mu \int_{\mathbb{R}^n} \bar{\rho} \dd \boldsymbol{x} \leq f(\mu_0).
    \]
Since $Q(\bar{\rho}) = 0$, it follows that
\[
\int_{\mathbb{R}^n} \bar{\rho} \dd \boldsymbol{x} \leq C(n,\alpha,\gamma,\ell_1)
 E(\bar{\rho},\boldsymbol{0})^\frac{n\gamma - n - \alpha}{(2n - \alpha)\gamma - 2n}.
 \]
 Now let $0< \lambda < 1$. By {\rm Lemma \ref{facts} (d), (e)}, and {\rm(f)}, 
 we see that $Q(\bar\rho_\lambda) > 0$ and 
 $$
 E(\bar\rho_\lambda,\boldsymbol{0}) = S_{\mu,\lambda}(\bar\rho) 
 - \mu \int_{\mathbb{R}^n} \bar\rho_\lambda \dd \boldsymbol{x} < S_{\mu}(\bar\rho) 
 - \mu \int_{\mathbb{R}^n} \bar\rho \dd \boldsymbol{x} = E(\bar\rho,\boldsymbol{0}).
 $$
 Notice that $\frac{n\gamma - n - \alpha}{(2n - \alpha)\gamma - 2n} < 0$
 for $\frac{2n}{2n-\alpha} < \gamma < \frac{n+\alpha}{n}$.
 Therefore, for $0<\lambda<1$, then $(\bar\rho_\lambda,\boldsymbol{0}) \in \mathcal{I}$. 
 Consequently, there exists arbitrarily small perturbations $(\rho_0,\mathcal{M}_0)$ 
 of $(\bar{\rho},\boldsymbol{0})$ such that $Q(\rho_0) > 0$ and \eqref{energymasscond} holds. 
\end{remark}

\smallskip
We now show that, if $(\rho_0,\mathcal{M}_0) \in \mathcal{I}_\mu$, 
then either $R(t) \to \infty$ as $t \to \infty$ or $Q(\rho(t))$ 
must remain bounded below by a positive constant.

\begin{remark}
In the 
spherically symmetric case,  it was proved in {\rm\cite{Carrillo2025}} that, if
$M < M^\alpha_{\rm c}(\gamma)$
and $\gamma \geq \frac{3n}{3n - 2(1+\alpha)}$ {\rm (}for $\alpha \neq 2${\rm )}, 
then global finite-energy solutions to \eqref{1.1} exist 
with initial conditions $(\rho_0,\mathcal{M}_0)$.
\end{remark}

The following Lemma allows us to characterize the asymptotic behavior of either
the radius of the support, $R(t)$, or the quantity $Q(\rho(t))$. 
In contrast to \cite{Cheng_2025}, 
it is not guaranteed that $Q(\rho(t))$ remains uniformly bounded below by a positive constant 
for all time. Nevertheless, in the case where this lower bound holds, we can derive the 
corresponding asymptotic behavior of $R(t)$.

\begin{lemma}\label{boundbel}
    Suppose that $(\rho(t),\mathcal{M}(t))$ is a classical solution of the attractive CEREs {\rm (}$\kappa = 1${\rm )} \eqref{1.1} with $\alpha \in (0,n)$,
    $\frac{2n}{2n - \alpha} < \gamma < \frac{n + \alpha}{n}$, and
    initial data $(\rho_0,\mathcal{M}_0) \in \mathcal{I}_\mu$. 
    Then one of the following alternatives holds{\rm :}
    \begin{enumerate}
\item[\rm (i)] $Q(\rho(t)) > 0$ for all $t \in [0,\infty)$ and there exists a sequence $t_m \to \infty$ as $m \to \infty$ such that $Q(\rho(t_m)) \to 0$ and $R(t_m) \to \infty$ as $m \to \infty${\rm ;}

\smallskip
\item[\rm (ii)] There exists $\Lambda > 0$ 
    such that $Q(\rho(t)) \geq \Lambda$ for all $t \geq 0$.
\end{enumerate}
\end{lemma}

\begin{proof}
Suppose that there does not exist $\Lambda > 0$ such that $Q(\rho(t)) \geq \Lambda$ 
for all $t>0$. Then there exists
$t_0 \in (0,\infty]$ and a sequence $t_m \to t_0$ as $m \to \infty$ such that 
$$
Q(\rho(t))> 0 \,\,\,\,\mbox{ for all $t \in [0,t_0)$}, \qquad\,\,\, 
\lim_{m \to \infty} Q(\rho(t_m)) = 0.
$$

We claim that $t_0 = \infty$. Indeed, if $t_0 < \infty$, then $Q(\rho(t_0)) = 0$, 
which contradicts the invariance of $\mathcal{I}_\mu$ established 
in Lemma \ref{invariant}. 
Hence, $t_0=\infty$ so that
$$
Q(\rho(t))> 0 \qquad\mbox{ for all $t\ge 0$}.
$$
It then follows from the scaling properties in Lemma \ref{facts}
that 
$$
\lambda^*(\rho(t)) > 1 \qquad\mbox{for $t \in [0,\infty)$}.
$$
By Lemma \ref{concavity}, the map: $\lambda \mapsto S_{\mu,\lambda}(\rho(t))$ is concave 
on $[1,\lambda^*(\rho(t))]$. Using the concavity together with 
the relation of the derivative of $S_\mu(\rho)$ 
at $\lambda=1$
and $Q(\rho)$ (see \eqref{firstderiva}), we obtain that, for all $t > 0$,
    \begin{align*}
 \ell_\mu \leq S_{\mu, \lambda^*(\rho(t))}(\rho(t)) 
 &\leq S_\mu(\rho(t)) 
 + \big(\lambda^*(\rho(t)) - 1\big)\frac{\d S_{\mu,\lambda}(\rho(t))}{\d \lambda} \bigg |_{\lambda = 1}\\ 
 &= S_\mu(\rho(t)) + \big(\lambda^*(\rho(t)) - 1\big) \frac{Q(\rho(t))}{n}.
    \end{align*}
    Rearranging, and using the conservation of energy, we deduce
    \begin{equation}\label{lamb}
    Q(\rho(t)) \geq \frac{n(\ell_\mu - S_\mu(\rho(t)))}{\lambda^*(\rho(t)) - 1} \geq \frac{n(\ell_\mu - I_\mu(\rho_0,\mathcal{M}_0))}{\lambda^*(\rho(t)) - 1},
    \end{equation}
where $\ell_\mu - I_\mu(\rho_0,\mathcal{M}_0) > 0$.

\begin{case}
Suppose that
    \begin{equation*}
 \liminf_{m \to \infty} \frac{Q(\rho(t_m))}{A_\alpha(\rho(t_m))} > 0.
    \end{equation*}
Since $\lim_{m \to \infty} Q(\rho(t_m)) = 0$, it follows that 
$$
\lim_{m \to \infty}\int_{\mathbb R^n}\rho(t_m)\Phi_\alpha*\rho(t_m)\dd \boldsymbol{x} = 0,
\qquad \lim_{m \to \infty} \int_{\mathbb{R}^n} \rho(t_m)^\gamma \dd \boldsymbol{x} = 0.
$$
Assume for contradiction that $R(t_m) \not \to \infty$ as $m \to \infty$. 
Then there exist both a subsequence of $(t_m)_m$ (which we relabel as $(t_m)_m$) and a constant 
$C>0$ such that $R(t_m) \leq C$ for all $m$.
Hence, we have
    \[
    M = \int_{\Omega(t_m)} \rho(t_m) \dd \boldsymbol{x} = \int_{B_{\rm c}} \rho(t_m) \dd \boldsymbol{x}. 
    \]
By H\"{o}lder's inequality, we obtain    
\[ 
M\le |B_{\rm c}|^{\frac{\gamma-1}{\gamma}}    
    \|\rho(t_m)\|_{L^\gamma} 
   =\Big(\frac{\omega_n C^{n}}{n}\Big)^\frac{\gamma - 1} {\gamma} \|\rho(t_m)\|_{L^\gamma} 
    \rightarrow 0\qquad \mbox{as $m \to \infty$}. 
    \]
Since $\|\rho(t_m)\|_{L^\gamma}\to 0$ as $m\to \infty$, 
it follows that $M\to 0$, which contradicts the assumption that $M>0$.
Therefore, we conclude that
$R(t) \to \infty$ as $t \to \infty$.
\end{case}

\begin{case}
    Otherwise, suppose that
    \begin{equation*}
 \liminf_{m \to \infty} \frac{Q(\rho(t_m))}{A_\alpha(\rho(t_m))} = 0.
    \end{equation*}
    Then there exists a subsequence of $(t_m)_m$ (which we relabel as $(t_m)_m$) such that
    \begin{equation}\label{sequ}
  \lim_{m \to \infty} \frac{Q(\rho(t_m))}{A_\alpha(\rho(t_m))} = 0.
    \end{equation}
From \eqref{lambdastar}, we have
    \begin{equation}\label{newlamb}
 \lambda^*(\rho) = \Big( 1 + \frac{2 Q(\rho)}{\alpha A_\alpha(\rho)} \Big)^\frac{n}{\alpha - n(\gamma - 1)}.
    \end{equation}
Combining \eqref{sequ} with \eqref{newlamb}, 
we deduce 
$$
\lim_{m \to \infty} \lambda^*(\rho(t_m)) = 1.
$$ 
Substituting this into \eqref{lamb}, we obtain 
$\lim_{m \to \infty} Q(\rho(t_m)) = \infty$.
This is a contradiction.
Therefore, this case cannot occur, and we must be in Case 1.
\end{case}
\setcounter{case}{0}
\end{proof}

\begin{remark}
 Unlike in {\rm\cite{Cheng_2025}}, the implication
 $$
 \lim_{t \to t_0} Q(\rho(t)) = 0
 \quad\Longrightarrow \quad \lim_{t \to t_0} \lambda^*(\rho(t)) = 1 
 $$
 does not hold here.
 Therefore, it is necessary to distinguish between the cases 
 in which $\lim_{t \to t_0} \lambda^*(\rho(t)) = 1$ and in which this limit does not hold, and to
 analyze the consequences in each scenario.
\end{remark}

\begin{lemma}\label{asymlem}
    Suppose that  $(\rho(t),\mathcal{M}(t))$ is a classical solution of the attractive CEREs ${\rm (\kappa = 1)}$ \eqref{1.1}
    for $\alpha \in (0,n)$ and $\frac{2n}{2n - \alpha} < \gamma < \frac{n + \alpha}{n}$ 
    such that there exists a constant $\Lambda > 0$ 
    satisfying $Q(\rho(t)) \geq \Lambda$ for all $t \in [0,\infty)$. Then 
    \[
    \liminf_{t \to \infty} \Big ( \frac{R(t)}{t} \Big ) \geq \Big ( \frac{\Lambda}{M} \Big )^\frac{1}{2}.
    \]
\end{lemma}

\begin{proof}
    As in the proof of Lemma \ref{E>0,general}, we have 
    \[
  H''(t) = 2\Big(\int_{\Omega(t)}\Big| \frac{\mathcal{M}}{\sqrt{\rho}} \Big|^2(t) \dd \boldsymbol{x}+na_0\int_{\Omega(t)}\rho^\gamma(t)\dd \boldsymbol{x} + \frac{\alpha}{2} \int_{\Omega(t)}\rho(t) \Phi_\alpha\ast\rho(t)\dd \boldsymbol{x}\Big) \geq 2Q(\rho(t)) \geq 2\Lambda.
    \]
    Thus, as in Lemma \ref{E>0,general}, we obtain
    \[
    H(0)+H'(0)t+\Lambda t^2\leq H(t)\leq R(t)^2M.
    \]
    The result follows.
\end{proof}

\begin{proof}[Proof of Theorem \ref{smallgamma}]
    Lemma \ref{union} shows that the condition $(\rho_0,\mathcal{M}_0) \in \bigcup_{\mu > 0} \mathcal{I}_\mu$ is equivalent to $Q(\rho_0) > 0$ together with \eqref{energymasscond}. 
    
    By Lemma \ref{boundbel}, since $(\rho_0,\mathcal{M}_0) \in \mathcal{I}_{\mu}$ 
    for some $\mu > 0$, then either there exists a sequence $t_m \to \infty$ as $m \to \infty$ such that $R(t_m) \to \infty$ as $m \to \infty$, 
    or there exists $\Lambda > 0$ such that $Q(\rho(t)) \geq \Lambda$ for all $t \in [0,\infty)$. 
   
    In the latter case, Lemma \ref{asymlem} yields
    \[
 \liminf_{t \to \infty} \Big ( \frac{R(t)}{t} \Big ) \geq \Big ( \frac{\Lambda}{M} \Big )^\frac{1}{2}.
    \]
\end{proof}

\setcounter{step}{0}
\setcounter{case}{0}

\section{Stability of Steady States for General and Polytropic Pressures with $\gamma > \frac{n + \alpha}{n}$} \label{stability}

In this section, we turn to the stability of steady states. 
In Section 4, we have proved the instability of steady states in the mass-supercritical range:
$\frac{2n}{2n-\alpha} < \gamma < \frac{n + \alpha}{n}$. 
We now consider the complementary range: $\gamma > \frac{n + \alpha}{n}$ in the polytropic case. 
We show that, in this range, steady states exhibit stability properties. 
Furthermore, the analysis extends naturally to general pressure laws 
with suitable asymptotic behavior near both the origin and infinity.

\subsection{Global Minimizers of the Free Energy}

We now employ the classical concentration compactness principle of Lions 
in \cite[Theorem II.2]{Lions_1984} to prove the existence of minimizers of the problem
\begin{equation*}
g_M = \inf_{\rho \in X_M} \mathcal{G}(\rho).
\end{equation*}
We first prove some facts about the minimizers of general free energies. 
To do this, we consider a general nonlocal interaction kernel of the form: 
\begin{equation}\label{potential}
\Psi \in L^p_{\rm w}(\mathbb{R}^n) \,\,{\rm (}\text{{\it i.e.}, weak $L^p$ space{\rm )}}\,\,\,
\text{for $p \in (1,\infty)$
\,\,\,\,\,  with $\,\Psi \geq 0\,$  {\it a.e.}}
\end{equation}
such that there exists $\beta \in (0,n)$ satisfying
\begin{equation}\label{Psi}
\Psi(t \xi) \geq t^{-\beta} \Psi(\xi)
\qquad\mbox{ for all $t \geq 1$ and $\xi \in \mathbb{R}^n$ {\it a.e.}}
\end{equation}
We take a convex function $\Pi$ satisfying
    \begin{equation}\label{convex}
\begin{cases}
\Pi:[0,\infty) \to [0,\infty) \text{ strictly convex,}
\\
 \lim_{t \to 0^+} \Pi(t)t^{-1} = 0, \\
 \liminf_{t \to \infty} \Pi(t)t^{-q} > 0,
  \end{cases}
    \end{equation}
    where $q = \frac{p + 1}{p}$. We define the associated free energy
\begin{equation*}
    \mathcal{F}(\varrho):=\int_{\mathbb{R}^n}\Big(\Pi(\varrho(\boldsymbol{x})) - \frac{1}{2}\varrho(\boldsymbol{x})(\Psi \ast \varrho)(\boldsymbol{x})\Big)\,\dd\boldsymbol{x},
\end{equation*}
with the corresponding admissible set
\begin{equation*}
  \mathcal{A}_M := \Big\{ \varrho \in L^1_+(\mathbb{R}^n) \cap L^q(\mathbb{R}^n) \,:
  \int_{\mathbb{R}^n} \varrho \, \dd \boldsymbol{x} = M \Big\}.
    \end{equation*}
Define the minimum of $\mathcal{F}$ over $\mathcal{A}_M$ by
\begin{equation*}
    F_M := \inf_{\varrho \in \mathcal{A}_M} \mathcal{F}(\varrho).
\end{equation*}
Under analogous conditions on mass $M$ to that of \eqref{lower}--\eqref{higherpress}, 
we can show the following useful facts about the sign of the free energy.

\begin{lemma}\label{negative}
Let $\Psi$ and $\Pi$ satisfy \eqref{potential} and \eqref{convex}, respectively.
\begin{enumerate}
    \item[{\rm(}a{\rm)}] When
    \begin{equation}\label{masscond}
    M < \bigg ( \frac{2 \liminf_{t \to \infty} \Pi(t) t^{-q}}{C_*^\Psi} \bigg )^\frac{1}{2-q},
    \end{equation}
    where $C_*^\Psi>0$ is the sharp constant from weak Young's convolution inequality {\rm Lemma \ref{precised}}, then $F_M > - \infty$.

    \item[{\rm(}b{\rm)}] Let $\Psi$ further satisfy \eqref{Psi} 
    for $\beta \leq \frac{n}{p}$ and $\limsup_{t \to 0} \Pi(t) t^{-q} < \infty$. 
    When $\beta = \frac{n}{p}$ and
    $$
    M > \bigg ( \frac{2 \limsup_{t \to 0} \Pi(t) t^{-q}}{C_*^\Psi} \bigg )^\frac{1}{2-q},
    $$
    then $F_M < 0$.
\end{enumerate}
\end{lemma}

\begin{proof} We divide the proof into two steps, correspondingly.
    
\smallskip
{\bf 1.} We take $\rho^*$ large enough such that
\[
\Pi(\varrho) \varrho^{-q} > \frac{C_*^\Psi}{2} M^{2-q}
\qquad\mbox{ for all $\varrho \geq \rho^*$}.
\]
Then, for any $\varrho \in \mathcal{A}_M$, applying the weak Young convolution inequality, 
we have 
\begin{align*}
\mathcal{F}(\varrho) & \geq \int_{\{ \varrho \geq \rho^* \}} \Pi(\varrho(\boldsymbol{x}))\,\dd\boldsymbol{x} - \frac{C_*^\Psi}{2} \| \varrho \|_{L^1}^{2-q} \| \varrho \|_{L^q}^q \nonumber\\
& \geq \int_{\{ \varrho \geq \rho^* \}} \varrho^q \Big ( \Pi(\varrho(\boldsymbol{x})) \varrho^{-q} - \frac{C_*^\Psi}{2} M^{2-q} \Big ) \,\dd\boldsymbol{x} - \frac{C_*^\Psi}{2} M^{2-q} \int_{\{ \varrho < \rho^* \}} \!\!\!\!\varrho^q \,\dd\boldsymbol{x} \nonumber\\
& \geq - \frac{C_*^\Psi}{2} M^{2-q} (\rho^*)^{q-1} \int_{\{ \varrho < \rho^* \}} \varrho \,\dd\boldsymbol{x} \\
& \geq - \frac{C_*^\Psi}{2} M^{3-q} (\rho^*)^{q-1} \nonumber,
\end{align*}
where we have used $\Pi\geq 0$ in the first inequality. Thus, $F_M > - \infty.$

\medskip
{\bf 2.} There exists $\theta \in (0,1)$ such that
 \[
  \delta_{\beta, \frac{n}{p}} \limsup_{t \to 0} \Pi(t) t^{-q} < \frac{\theta C_*^\Psi}{2} M^{2-q},
 \]
  where 
 $$
 \delta_{\beta, \frac{n}{p}} = \begin{cases}
 1 & \mbox{for $\beta = \frac{n}{p}$}, \\
 0 & \mbox{for $\beta < \frac{n}{p}$}.
\end{cases}
  $$
By Lemma \ref{precised} and the sharpness of $C_*^\Psi > 0$, 
there exists $\varrho \in L^1_+\cap L^q(\mathbb{R}^n)$ such that
\begin{equation}\label{criti}
\int_{\mathbb{R}^n} \varrho \Psi * \varrho \dd \boldsymbol{x} > \theta C_*^\Psi \| \varrho \|_{L^q}^q \| \varrho \|_{L^1}^{2-q}.
\end{equation}
By density of $C^\infty_{\rm c}(\mathbb{R}^n)$ in $L^1\cap L^q(\mathbb{R}^n)$, 
we may assume that $\varrho \in C^\infty_{\rm c}(\mathbb{R}^n)$ and is hence a bounded function. 
Now, defining $\varrho_R(\boldsymbol{x})$ as in \eqref{scale}, then
$\varrho_R \in \mathcal{A}_M$. 
Since $\Psi$ satisfies \eqref{Psi} for $\beta \leq \frac{n}{p}$, 
we combine this with \eqref{criti} to see that, for $R < \frac{M}{\|\varrho\|_{L^1}}$,
\begin{align*}
\int_{\mathbb{R}^n} \varrho_R \Psi * \varrho_R \dd \boldsymbol{x} 
& =  \Big (\frac{M}{\|\varrho\|_{L^1}} \Big)^2 \int_{\mathbb{R}^n} 
\int_{\mathbb{R}^n} \varrho(\boldsymbol{x}) \varrho(\boldsymbol{y}) 
\Psi((R\|\varrho\|_{L^1})^{-\frac{1}{n}} M^\frac{1}{n} (\boldsymbol{x} - \boldsymbol{y})) 
\dd \boldsymbol{x} \dd \boldsymbol{y} \\
& \geq R^\frac{\beta}{n} \Big(\frac{M}{\|\varrho\|_{L^1}} \Big)^{2 - \frac{\beta}{n}} \int_{\mathbb{R}^n} \varrho \Psi * \varrho \dd \boldsymbol{x} \\
& \geq \theta C_*^\Psi M^{2-\frac{\beta}{n}} R^\frac{\beta}{n} \| \varrho \|_{L^q}^q \| \varrho \|_{L^1}^{\frac{\beta}{n} - q}.
\end{align*}
Thus, taking $R< \min \big\{1,\frac{M}{\|\varrho\|_{L^1}} \big \}$, we obtain
\begin{align*}
\mathcal{F}(\varrho_R) R^{-\frac{\beta}{n}} \Big( \frac{M}{\| \varrho \|_{L^1}} \Big)^{\frac{\beta-nq}{n}} 
\leq &\, R^{\frac{1}{p} - \frac{\beta}{n}} 
\Big( \frac{M}{\| \varrho \|_{L^1}} \Big)^{\frac{\beta}{n}- \frac{1}{p}} 
\int_{\mathbb{R}^n} \frac{\Pi(R \varrho)}{(R \varrho)^q} \varrho^q \dd \boldsymbol{x} 
- \frac{\theta C_*^\Psi}{2} \| \varrho \|_{L^q}^q M^{2-q} \\
= &\,\int_{\{\varrho > 0\}} \bigg ( R^{\frac{1}{p} 
- \frac{\beta}{n}} 
\Big( \frac{M}{\| \varrho \|_{L^1}} \Big)^{\frac{\beta}{n}-\frac{1}{p}}\frac{\Pi(R\varrho)}{(R\varrho)^q} 
- \frac{\theta C_*^\Psi}{2} M^{2-q} \bigg ) \varrho^q \dd \boldsymbol{x}.
\end{align*}
Since $\varrho$ is a bounded function and 
$$
\limsup_{R \to 0} R^{\frac{1}{p} - \frac{\beta}{n}}\frac{\Pi(R \varrho)}{(R \varrho)^q} 
\leq  \frac{C_*^\Psi}{2} M^{2-q} \qquad\mbox{for $\varrho > 0$},
$$
then $R^{\frac{1}{p} - \frac{\beta}{n}}\frac{\Pi(R \varrho)}{(R \varrho)^q}$ is bounded for $R<1$. 
Thus, there exists $C>0$ such that $R^{\frac{1}{p} - \frac{\beta}{n}}\frac{\Pi(R \varrho)}{(R \varrho)^q} 
\leq C$ for $R<1$. Then we can apply the reverse Fatou lemma to obtain 
\[
\limsup_{R \to 0} \mathcal{F}(\varrho_R) R^{- \frac{\beta}{n}} 
\Big( \frac{M}{\| \varrho \|_{L^1}} \Big)^{\frac{\beta-nq}{n}} 
\leq \| \varrho \|^q_{L^q} \Big ( \delta_{\beta, \frac{n}{p}} \limsup_{t \to 0} \Pi(t) t^{-q} 
- \frac{\theta C_*^\Psi}{2}  M^{2-q} \Big ) < 0.
\]
This implies that $F_M < 0$.
\end{proof}

For the sake of completeness, we remind the reader the main result in \cite{Lions_1984}.

\begin{theorem}[\text{\cite[Theorem II.2]{Lions_1984}}]
\label{minimizer}
Suppose that $\Psi$ and $\Pi$ satisfy \eqref{potential} and \eqref{convex}, respectively, 
and that \eqref{masscond} is satisfied. 
Then, for any minimizing sequence $\{\varrho_m\}_m \subset \mathcal{A}_M$, 
there exist $\boldsymbol{y}_m \in \mathbb{R}^n$ and 
a minimizer $\bar{\rho} \in \mathcal{A}_M$ to $\mathcal{F}$ {\rm(}{\it i.e.,} 
$\mathcal{F}(\bar{\rho}) = F_M${\rm)} 
with $T^{\boldsymbol{y}_m} \varrho_m$ compact in $L^1(\mathbb{R}^n) \cap L^q(\mathbb{R}^n)$
so that $T^{\boldsymbol{y}_m} \varrho_m \to \bar{\rho}$ in $L^1(\mathbb{R}^n) \cap L^q(\mathbb{R}^n)$
if and only if
\begin{equation}\label{inequality}
    F_M < F_m + F_{M - m} \qquad \text{ for all } m \in (0,M).
\end{equation}
\end{theorem}

In order to prove the existence of minimizers for our problem, it suffices to prove \eqref{inequality} for applying Theorem \ref{minimizer}. In order to verify condition \eqref{inequality}, we also employ a corollary from \cite{Lions_1984}.

\begin{corollary}[\text{\cite[ Corollary II.1]{Lions_1984}}]
\label{cor}
Suppose \eqref{potential} and \eqref{convex} are satisfied.
If there exists some $\beta \in (0,n)$ such that
    \[
  \Psi(t \xi) \geq t^{-\beta} \Psi(\xi)
 \qquad\mbox{for all $t \geq 1$ and $\xi \in \mathbb{R}^n$ a.e.},
    \]
then \eqref{inequality} is satisfied if and only if $F_M < 0$.
\end{corollary}

Under the assumptions of Theorem \ref{minimizer} and \eqref{Psi}, Corollary \ref{cor} implies that \eqref{inequality} is satisfied. Thus, Theorem \ref{minimizer} implies that there exists 
a minimizer $\bar{\rho} \in \mathcal{A}_M$, that is, $\mathcal{F}(\bar{\rho}) = F_M$. 
We now show that, under the radial conditions on the interaction kernel $\Psi$, 
the minimizers are indeed radially symmetric, up to translation. 
First of all, we prove a generalization of Cavalieri's formula. 
To do this, we use a corollary from \cite{Serrin_1969}.

\begin{corollary}
[\text{\cite[Corollary 9]{Serrin_1969}}]
\label{change}
Let $I \subseteq \mathbb{R}$ be a subinterval, and 
let $g:I \to \mathbb{R}$ be an injective and locally absolutely continuous function of
local bounded variation.  
If $f: g(I) \to \mathbb{R}$ is measurable and locally bounded on $g(I)$, then 
    \begin{equation*}
   \int_{g(I)} f(t) \dd t = \int_I f(g(t)) g'(t) \dd t.
    \end{equation*}
\end{corollary}

\begin{lemma}
    Let $(X,\mu)$ be a topological measure space, and let $\Pi$ satisfy \eqref{convex}. 
    Then, for all measurable functions $\varrho: (X,\mu) \to \mathbb{R}$ such that
    \begin{equation*}
  \int_X \Pi(|\varrho|) \dd \mu < \infty,
    \end{equation*}
the following Cavalieri formula holds{\rm:}
    \begin{equation}\label{cavalieri}
\int_X \Pi(|\varrho|) \dd \mu = \int^\infty_0 \Pi'(\lambda) \mu(\{ |\varrho| > \lambda \}) \dd \lambda.
    \end{equation}
\end{lemma}

\begin{proof}
Since $\Pi$ is convex on $[0,\infty)$, it is differentiable {\it a.e.}, and therefore \eqref{cavalieri} makes sense. By $\eqref{convex}_2$, we see that $\lim_{t \to 0^+} \Pi(t)t^{-1} = 0$, 
which implies that $\Pi(0) = 0$, the right derivative $\Pi_+'(0) = 0$.
Since $\Pi$ is strictly convex, then $\Pi$ is strictly increasing on $[0,\infty)$ 
and is a bijection from $[0,\infty)$ to $[0,\infty)$. 
Thus, by applying the standard Cavalieri's formula to the identity function, we have
\begin{align*}
\int_X \Pi(|\varrho|) \dd \mu & = \int^\infty_0 \mu(\{ \Pi(|\varrho|) > \tau \}) \dd \tau = \int^\infty_0 \mu(\{ |\varrho| > \Pi^{-1}(\tau) \}) \dd \tau  = \int^\infty_0 \Pi'(\lambda) \mu(\{ |\varrho| > \lambda \}) \dd \lambda.
\end{align*}
Here, the last equality follows from Corollary \ref{change} with $I = (0,\infty)$, 
$g = \Pi: (0,\infty) \to (0,\infty)$, and $f = \mu(\{ |\varrho| > \cdot \}): (0,\infty) \to (0,\infty)$
that satisfy the required properties, as $\Pi$ being strictly convex, which imply
the local absolute continuity and local bounded variation.
\end{proof}

\begin{lemma}\label{symm}
Suppose that the conditions of {\rm Lemma \ref{negative}} are satisfied, with $\Psi$ radially symmetric and radially decreasing. Then there exists a radial minimizer $\bar{\varrho} \in \mathcal{A}_M$ of $\mathcal{F}$ over $\mathcal{A}_M$. Moreover, any minimizer of $\mathcal{F}$ over $\mathcal{A}_M$ is the translation of a radially symmetric minimizer.
\end{lemma}

\begin{proof}
We can now apply Cavalieri's formula to obtain that, 
for a minimizer $\bar{\rho} \in \mathcal{A}_M$, 
taking the spherically decreasing rearrangement $\rho^\#$ of $\bar{\rho}$, then
$\rho^\# \in \mathcal{A}_M$ and
\begin{align*}
 \int_{\mathbb{R}^n} \Pi (\rho^\#) \dd \boldsymbol{x} & =  \int^\infty_0 \Pi'(\tau) |\{ \rho^\# > \lambda \}| \dd \lambda  = \int^\infty_0 \Pi'(\tau) |\{ \bar{\rho} > \lambda \}| \dd \lambda= \int_{\mathbb{R}^n} \Pi (\bar{\rho}) \dd \boldsymbol{x}.
\end{align*}
If $\Psi$ is spherically symmetric and decreasing in $r$, 
then $\Psi^* = \Psi$. 
Moreover, by \cite{sobolev1938theorem} and \cite[Theorem 3.9]{Lieb_2001}, we have
\begin{equation*}
- \int_{\mathbb{R}^n} \rho^\#(\boldsymbol{x}) (\Psi * \rho^\#)(\boldsymbol{x}) \dd \boldsymbol{x} \leq - \int_{\mathbb{R}^n} \bar{\rho}(\boldsymbol{x}) (\Psi * \bar{\rho})(\boldsymbol{x}) \dd \boldsymbol{x},
\end{equation*}
with equality only if $\bar{\rho} = \rho^\#(\cdot -\boldsymbol{y})$ 
for some $\boldsymbol{y} \in \mathbb{R}^n$. 
This implies that $\mathcal{F}(\rho^\#) \leq \mathcal{F}(\bar{\rho})$. 
Since $\bar{\rho}$ is a minimizer of $\mathcal{F}$ over $\mathcal{A}_M$, 
we conclude that the equality holds and that $\bar{\rho}$ is the translation of a radial minimizer.
\end{proof}

We now apply the general framework to our particular minimization problem. 
Taking $\Psi = - \Phi_\alpha$, which satisfies both \eqref{potential} 
for $p = \frac{n}{\alpha}$ and $q = \frac{p + 1}{p} = \frac{n + \alpha}{n}$ 
and \eqref{Psi} with an equality for $\beta = \alpha$ 
and all $\xi \in \mathbb{R}^n\setminus \{\boldsymbol{0}\}$. 
For the internal energy function $\Pi(\varrho) := \varrho e(\varrho)$, 
we obtain that
$\Pi''(\varrho) = \frac{p'(\varrho)}{\varrho} > 0$ for $\varrho > 0$ 
from assumption \eqref{Press} on $p$. 
This implies that $\Pi$ is strictly convex on $[0,\infty)$, so that we can clearly check 
from the formula for $e(\varrho)$, $\eqref{Press}$, and \eqref{lower}--\eqref{higherpress} 
that it is non-negative and satisfies \eqref{convex}.
When
\[
 M > \bigg ( \frac{2n \limsup_{\rho \to 0} p(\rho) \rho^{-\frac{n +\alpha}{n}}}{ C_*} \bigg )^\frac{n}{n-\alpha},
\]
we see that $g_M < 0$ by Lemma \ref{negative}. 
It follows from Theorem \ref{minimizer} and Corollary \ref{cor} 
that there exists a minimizer $\bar{\rho} \in X_M$ of $\mathcal{G}$ over $X_M$. 
Since $\Psi = - \Phi_\alpha$ is spherically symmetric and radially decreasing, 
we apply Lemma \ref{symm} to prove the existence of a spherically symmetric minimizer 
of functional $\mathcal{G}$ over $X_M$. Thus, we obtain the following corollary.

\begin{corollary}
Let $\alpha \in (0,n)$, and $M>0$ with $p \in C^2([0,\infty))$ satisfying \eqref{Press} and
\eqref{lower}--\eqref{higherpress}. 
Then there exists a spherically symmetric minimizer $\bar{\rho} \in X_M$ to $\mathcal{G}$.
\end{corollary}

For the uniqueness of global minimizers and steady states, see \cite[Theorem 5.5]{Carrillo2025} and the references therein. For the rest of this subsection, we assume the uniqueness 
of the global minimizer (up to translation) as required in Theorem \ref{ns}.

\subsection{Stability of the Global Minimizers}
We now turn to the question of stability of the global minimizer obtained in \S 5.1
which is unique (up to translation). 
For $\alpha \in (0,n)$, we define the relative free energy towards the steady state as
\begin{align}\label{difference}
    \begin{split}
 \mathcal{G}(\rho(t)) - \mathcal{G}(\bar{\rho}) & = d(\rho(t),\bar{\rho}) + \frac{1}{2} \int_{\mathbb{R}^n} (\rho(t) - \bar{\rho})\, \Phi_\alpha* (\rho(t) - \bar{\rho}) \dd \boldsymbol{x}.
    \end{split}
\end{align}
Combining \eqref{Energydiff} with \eqref{difference}, in the attractive case we obtain
\begin{align}\label{relative}
 \begin{split}
 E(\rho,\mathcal{M}\, | \, \bar{\rho}, \boldsymbol{0})(t)
& = d(\rho(t),\bar{\rho}) + \frac{1}{2} \int_{\mathbb{R}^n} (\rho(t) - \bar{\rho})\, \Phi_\alpha* (\rho(t) - \bar{\rho}) \dd \boldsymbol{x}
+ \frac{1}{2}\int_{\mathbb{R}^n}
 \Big|\frac{\mathcal{M}}{\sqrt{\rho}}\Big|^2(t,\boldsymbol{x})\,\dd\boldsymbol{x}.
    \end{split}
\end{align}
Notice that the second term in the last equality can be obtained by the H\"older inequality 
and the Hardy-Littlewood-Sobolev (HLS) inequality from Lemma \ref{simplehls} as
\[
\Big|\int_{\mathbb{R}^n} (\rho(t) - \bar{\rho}) \,\Phi_\alpha* (\rho(t) - \bar{\rho}) \dd \boldsymbol{x}\Big| \leq C_{n,\alpha} \| \rho(t) - \bar{\rho} \|_{L^\frac{2n}{2n - \alpha}}^2.
\]

Having established the existence of minimizers, we now prove the main theorem.

\begin{proof}[Proof of Theorem \ref{ns}]
We follow a similar approach to that in Rein \cite{Rein_2003}; however,
a key ingredient in our analysis is the uniqueness of global minimizers modulo translations.

Suppose, for contradiction, that the statement of Theorem \ref{ns} is false. 
Then there exist $\varepsilon > 0$, a sequence of global weak solutions $(\rho^j,\mathcal{M}^j)$ of the attractive CEREs associated to the initial data $(\rho^j_0,\mathcal{M}^j_0)$,
and a sequence of times $t_j > 0$ such that
\begin{align}\label{contradict}
d(\rho^j_0,\bar{\rho})+ \| \rho^j_0 - \bar{\rho} \|_{L^\frac{2n}{2n - \alpha}}^2
+\frac{1}{2}\int_{\mathbb{R}^n}\Big|\frac{\mathcal{M}^j_0}{\sqrt{\rho^j_0}}\Big|^2\,\dd \boldsymbol{x}<\frac{1}{j},
\end{align}
and
\begin{align}\label{auxx}
d(\rho^j(t_j),T^{\boldsymbol{y}}\bar{\rho})
+ \| \rho^j(t_j) - T^{\boldsymbol{y}} \bar{\rho} \|_{L^\frac{2n}{2n - \alpha}}^2
+\frac{1}{2}\int_{\mathbb{R}^n}\Big|\frac{\mathcal{M}^j}{\sqrt{\rho^j}}\Big|^2(t_j) \,\dd \boldsymbol{x}
\geq \v \qquad\,\,\mbox{for all $\boldsymbol{y} \in \mathbb{R}^n$.}
\end{align}
Notice that the sequence of times $t_j$ can be chosen outside the measure-zero set, so that
the terms in \eqref{auxx} are well defined on the sequence, in accordance with Definition \ref{definition}.
Then, by \eqref{contradict} and \eqref{difference}, we have
\begin{equation*}
    \lim_{j \to \infty} E(\rho^j_0,\mathcal{M}^j_0) = E(\bar{\rho},\boldsymbol{0}) = \mathcal{G}(\bar{\rho}).
\end{equation*}
Since the energy $E$ for the weak solutions is non-increasing from the initial energy 
by definition, we obtain 
\begin{equation*}
    \limsup_{j \to \infty} \mathcal{G}(\rho^j(t_j)) \leq \limsup_{j \to \infty} E(\rho^j(t_j),\mathcal{M}^j(t_j)) \leq \lim_{j \to \infty} E(\rho^j_0,\mathcal{M}^j_0) = \mathcal{G}(\bar{\rho}).
\end{equation*}
Thus, $\rho^j(t_j) \in X_M$ is a minimizing sequence of $\mathcal{G}$ in $X_M$ and, by Theorem \ref{minimizer}, 
there exists $\boldsymbol{y}_j \in \mathbb{R}^n$ such that $T^{\boldsymbol{y}_j} \rho^j(t_j) \to \bar{\rho}$ 
in $L^1(\mathbb{R}^n) \cap L^\frac{n + \alpha}{n}(\mathbb{R}^n)$, 
due to the uniqueness of $\bar{\rho}$ up to translation. 
Then
\begin{equation*}
 \| \rho^j(t_j) - T^{-\boldsymbol{y}_j} \bar{\rho} \|_{L^\frac{2n}{2n - \alpha}} 
 =  \| T^{\boldsymbol{y}_j} \rho^j(t_j) - \bar{\rho} \|_{L^\frac{2n}{2n - \alpha}} 
\rightarrow 0 \qquad\mbox{as $j \to \infty$},
\end{equation*}
since $\frac{2n}{2n - \alpha} \in (1, \frac{n + \alpha}{n})$. 
It follows that
\begin{equation*}
 \lim_{j \to \infty} E(\rho^j(t_j),\mathcal{M}^j(t_j)) = \mathcal{G}(\bar{\rho}).
\end{equation*}
Therefore, by \eqref{relative}, we have 
\begin{equation*}
 d(\rho^j(t_j),T^{-\boldsymbol{y}_j}\bar{\rho}) 
 + \frac{1}{2}\int_{\mathbb{R}^n}   
 \Big|\frac{\mathcal{M}^j}{\sqrt{\rho^j}}\Big|^2(t_j)\,\dd \boldsymbol{x} \rightarrow 0 
 \qquad \mbox{as $j \to \infty$}. 
\end{equation*}
This is a contradiction to \eqref{auxx}.
\end{proof}

\subsection{Relative Entropy Estimates}
We now focus on better quantifying the stability of steady states in the polytropic case, {\it i.e.}, $p(\rho) = a_0 \rho^\gamma$ for $\gamma > \frac{n + \alpha}{n}$. The argument is based on the relative entropy method. Throughout this subsection, we denote
\[ h(\rho)=\frac{a_0}{\gamma-1}\rho^\gamma,\qquad\theta=\frac{\gamma-1}{2},
\]
and focus on the case:
\[\max\big\{\frac{n+\alpha}{n},\,\frac{3n}{3n-2(\alpha+1)}\big\}<\gamma\le2.
\]
Recall that
\[\bar{\mathcal W}_\alpha:=\Phi_\alpha*\bar\rho, \qquad
 G_\mu:=\mu+\bar{\mathcal W}_\alpha,
\]
and 
$$ \Gamma:=\{\boldsymbol{x}\in\mathbb R^n\,:\,\bar\rho(\boldsymbol{x})>0\},$$
Set $K:=\supp  \,\bar \rho.$ 
Since $\bar\rho$ is compactly supported, then $K$ is compact and $\Gamma\subset K.$ 
By \eqref{station},
\begin{equation}\label{property}
G_\mu=-h'(\bar\rho)\quad\hbox{{\it a.e. }on } \Gamma, \qquad\,\, G_\mu\ge0\quad\hbox{{\it a.e. }on }\Gamma^{\rm c}.
\end{equation}
The relative energy is
$$
\eta_\mu(\rho,\mathcal M\mid\bar\rho) :=\frac{|\mathcal M|^2}{2\rho}+h_\mu(\rho\mid\bar\rho),
$$
with 
the relative internal energy:
\begin{equation}\nonumber
h_\mu(\rho\mid\bar\rho):=h(\rho)-h(\bar\rho)+G_\mu(\rho-\bar\rho),
\end{equation}
Then the relative total energy over $\mathbb{R}^n$ is
$$
E_\mu(t):=\int_{\mathbb R^n}\eta_\mu(\rho,\mathcal M\mid\bar\rho)(t)\,\dd\boldsymbol{x},
$$

\begin{lemma}\label{coer}
Let $1<\gamma\le2$. Then there exists a constant $c>0,$ depending only 
on $a_0,\gamma,B_{\bar{\rho}}:=\|\bar\rho\|_{L^\infty}$, and the steady state, 
such that, for all $a\ge0$ and for {\it a.e.} $\boldsymbol{x}\in\mathbb R^n$,
\begin{align}  
&h_\mu(a\mid\bar\rho(\boldsymbol{x}))
 \ge c\,\min\{|a-\bar\rho(\boldsymbol{x})|^2,\,|a-\bar\rho(\boldsymbol{x})|^\gamma\},
\label{onsupportestimate}\\
&h_\mu(a\mid0)
 =\frac{a_0}{\gamma-1}a^\gamma+G_\mu(\boldsymbol{x})a
 \ge \frac{a_0}{\gamma-1}a^\gamma
\qquad\mbox{for $\boldsymbol{x}\in \Gamma^{\rm c}$}.\label{offestimate}
\end{align}
Equivalently, \eqref{onsupportestimate} implies
\begin{align}
 &h_\mu(a\mid\bar\rho(\boldsymbol{x}))
 \ge c |a-\bar\rho(\boldsymbol{x})|^2
 \qquad\text{if $\,\boldsymbol{x}\in\Gamma$ and 
 $|a-\bar\rho(\boldsymbol{x})|\le1$},\label{onsupportsplit}\\
&h_\mu(a\mid\bar\rho(\boldsymbol{x}))
 \ge c |a-\bar\rho(\boldsymbol{x})|^\gamma
 \qquad\text{if $\,\boldsymbol{x}\in\Gamma$ and 
 $|a-\bar\rho(\boldsymbol{x})|>1$}. \label{onsupportsplitlarge}
\end{align}
Consequently, there exists $C>0$ such that
\begin{equation}\label{Hmuestimate}
  \int_\Gamma\min\{|\rho-\bar\rho|^2,|\rho-\bar\rho|^\gamma\}\,\dd\boldsymbol{x}
  +\int_{\Gamma^{\rm c}}\rho^\gamma\,\dd\boldsymbol{x}
 +\int_{\Gamma^{\rm c}}G_\mu\rho\,\dd\boldsymbol{x}
 \le C H_\mu(t),
\end{equation}
where $H_\mu(t):=\int_{\mathbb R^n}h_\mu(\rho\mid\bar\rho)(t)\,\dd\boldsymbol{x}$.
\end{lemma}

\begin{proof}
We first prove the estimate on $\Gamma$. For {\it a.e.} $\boldsymbol{x}\in \Gamma$, \eqref{property} gives
\[
G_\mu(\boldsymbol{x})=-h'(\bar\rho(\boldsymbol{x})).
\]
Then
\[
\begin{aligned}
  h_\mu(a\mid\bar\rho(\boldsymbol{x}))
  &=  h(a)-h(\bar\rho(\boldsymbol{x}))
  +G_\mu(\boldsymbol{x})(a-\bar\rho(\boldsymbol{x}))        \\
  &=
 h(a)-h(\bar\rho(\boldsymbol{x})) -h'(\bar\rho(\boldsymbol{x}))(a-\bar\rho(\boldsymbol{x})).
\end{aligned}
\]
It remains to estimate the  relative internal energy
\[
h(a)-h(b)-h'(b)(a-b)
 \qquad \mbox{for $0<b\le B_{\bar{\rho}}$}.
\]
We claim that there exists a constant $c>0$, depending only on
$a_0,\gamma$ and $B_{\bar{\rho}}$, such that
\begin{equation}\label{relativelower}
 h(a)-h(b)-h'(b)(a-b)
  \ge  c\min\{|a-b|^2,|a-b|^\gamma\}
\end{equation}
for all $a\ge0$ and all $0<b\le B_{\bar{\rho}}$.

\smallskip
If $a=b$, then both sides of \eqref{relativelower}
are zero.  We therefore assume $a\ne b$.

\smallskip
First suppose that
  $|a-b|\le \frac b2.$
Then the segment between $a$ and $b$ is contained in
$\left[\frac b2,\frac{3b}{2}\right].$
By Taylor's formula with integral remainder,
\[
  h(a)-h(b)-h'(b)(a-b)
   =
  (a-b)^2
  \int_0^1(1-s)h''(b+s(a-b))\,\dd s.
\]
Note that $h''(t)=a_0\gamma t^{\gamma-2}$ 
and $\gamma-2\leq0$. 
Since
\[
 b+s(a-b)\le \frac{3b}{2}\le \frac{3B_{\bar{\rho}}}{2},
\]
we have
\[
   h''(b+s(a-b))
  =
 a_0\gamma (b+s(a-b))^{\gamma-2}
  \ge
  a_0\gamma \Big(\frac{3B_{\bar{\rho}}}{2}\Big)^{\gamma-2}.
\]
Then
\[
 h(a)-h(b)-h'(b)(a-b)
  \ge  c|a-b|^2 .
\]

It remains to consider the case
\[  |a-b|>\frac b2 .
\]
In this case, $a$ is not close to $b$.  We use the homogeneity of
$r\mapsto r^\gamma$.  Since $b>0$,
\[
\begin{aligned}
h(a)-h(b)-h'(b)(a-b)
 &= \frac{a_0}{\gamma-1}
 \big(a^\gamma-b^\gamma-\gamma b^{\gamma-1}(a-b)
  \big)                                      \\
 &=  \frac{a_0}{\gamma-1}b^\gamma \Big(\big(\frac ab\big)^\gamma -1-\gamma\big(\frac ab-1\big) \Big).
\end{aligned}
\]
Define
\[
\Phi(z):=z^\gamma-1-\gamma(z-1)  \qquad \mbox{for $z\ge 0$}.
\]
The function $\Phi$ is continuous on $[0,\infty)$ and strictly positive for
$z\ne1$, and satisfies
\[\Phi(z)\sim z^\gamma
      \qquad\text{as }z\to\infty .
\]
Therefore, the ratio
$\frac{\Phi(z)}{|z-1|^\gamma}$
has a strictly positive lower bound on the closed set
\[
\big\{z\ge0:\ |z-1|\ge\frac12\big\}.
\]
Moreover, the condition $|a-b|>\frac{b}{2}$ is equivalent to
$\left|\frac ab-1\right|>\frac12.$
Then we have
\[ 
\Phi(\frac ab) \ge c\,\big|\frac ab-1\big|^\gamma .
\]
Substituting this into the previous identity gives
\[
\begin{aligned}
  h(a)-h(b)-h'(b)(a-b)  &\ge  c b^\gamma \big|\frac ab-1\big|^\gamma
  =c\,|a-b|^\gamma.
\end{aligned}
\]
Hence, we have 
\[h(a)-h(b)-h'(b)(a-b)  \ge
   c\,\min\{|a-b|^2,\,|a-b|^\gamma\}.
\]
This proves \eqref{relativelower}.

We now return to $h_\mu(a\mid\bar\rho(\boldsymbol{x}))$ on $\Gamma$.
If $\bar\rho(\boldsymbol{x})>0$, then
\eqref{relativelower} with
$b=\bar\rho(\boldsymbol{x})$ gives
\[
h_\mu(a\mid\bar\rho(\boldsymbol{x})) 
\ge c\,\min\{|a-\bar\rho(\boldsymbol{x})|^2,\,|a-\bar\rho(\boldsymbol{x})|^\gamma\}.
\]
If $\bar\rho(\boldsymbol{x})=0$, then 
\eqref{property} gives $G_\mu(\boldsymbol{x})\ge0$,
so that
\[ h_\mu(a\mid0) =  h(a)+G_\mu(\boldsymbol{x})a =
  \frac{a_0}{\gamma-1}a^\gamma+G_\mu(\boldsymbol{x})a\ge   \frac{a_0}{\gamma-1}a^\gamma.
\]
Since $a^\gamma\ge \min\{a^2,\,a^\gamma\},$
we also have
\[   h_\mu(a\mid0)  \ge c\,\min\{a^2,a^\gamma\}.
\]
Combining the above 
with $\bar\rho(\boldsymbol{x})>0$ for this case gives
\[h_\mu(a\mid\bar\rho(\boldsymbol{x})) 
\ge c\,\min\{|a-\bar\rho(\boldsymbol{x})|^2,
\,|a-\bar\rho(\boldsymbol{x})|^\gamma\} \qquad\text{for {\it a.e.} }\boldsymbol{x}\in \mathbb{R}^n.
\]
In particular, this proves \eqref{onsupportestimate}.

On $\Gamma^{\rm c}$, we know that $\bar\rho=0$.  Again, by 
\eqref{property},
$G_\mu(\boldsymbol{x})\ge0\,
\text{for {\it a.e.} }\boldsymbol{x}\in \Gamma^{\rm c}.$
Then
\[  h_\mu(a\mid0)=h(a)+G_\mu(\boldsymbol{x})a =\frac{a_0}{\gamma-1}a^\gamma+G_\mu(\boldsymbol{x})a  \ge  \frac{a_0}{\gamma-1}a^\gamma.
\]
This proves \eqref{offestimate}.

Finally, take $a=\rho(t,\boldsymbol{x})$ and integrate the pointwise
estimates.  From the estimate on $\Gamma$,
\[ \int_\Gamma  \min\{|\rho-\bar\rho|^2,\,|\rho-\bar\rho|^\gamma\}   \,\dd\boldsymbol{x}  \le  C \int_\Gamma h_\mu(\rho\mid\bar\rho)\,\dd\boldsymbol{x}.
\]
From the identity on $\Gamma^{\rm c}$,
\[  h_\mu(\rho\mid0)  =  \frac{a_0}{\gamma-1}\rho^\gamma+G_\mu\rho ,
\]
we obtain
\[  \int_{\Gamma^{\rm c}}\rho^\gamma\,\dd\boldsymbol{x}  +  \int_{\Gamma^{\rm c}}G_\mu\rho\,\dd\boldsymbol{x} \le  C \int_{\Gamma^{\rm c}} h_\mu(\rho\mid0)\,\dd\boldsymbol{x}.
\]
Adding the last two inequalities together gives
\[ \int_\Gamma
\min\{|\rho-\bar\rho|^2,\,|\rho-\bar\rho|^\gamma\} \dd\boldsymbol{x}
 +\int_{\Gamma^{\rm c}}\rho^\gamma\,\dd\boldsymbol{x}
 +
  \int_{\Gamma^{\rm c}}G_\mu\rho\,\dd\boldsymbol{x} \le C H_\mu(t).
\]
This proves \eqref{Hmuestimate}.
\end{proof}

\begin{lemma}
Let $(\rho,\mathcal M)$ and $(\bar\rho,\boldsymbol 0)$ satisfy the assumptions
of {\rm Theorem \ref{thm:stability}}. Set
\[
p_0:=\frac{3n}{3n-2(\alpha+1)}.
\]
Assume that $p_0<\gamma\le2$. Choose
$p_{\rm loc}\in(1,\gamma)$ such that
\begin{equation*}
0<\frac1{p_{\rm loc}}-\frac{n-\alpha-1}{n}\le\frac12.
\end{equation*}
Then, for {\it a.e.} $t\in(0,T)$, 
\begin{align}\label{festimates} 
&\|f(t,\cdot)\|_{L^{p_{\rm loc}}(\mathbb R^n)}^2 + \|f(t,\cdot)\|_{L^{p_0}(\mathbb R^n)}^3 \nonumber\\
&\,\,\le C\big(  H_\mu(t)  +H_\mu(t)^{2/p_{\rm loc}} +H_\mu(t)^{2/\gamma}+H_\mu(t)^{3/2}  +H_\mu(t)^{3/p_0}
  +H_\mu(t)^{3/\gamma} \big),
\end{align}
where $f(t,\boldsymbol{x}):=\rho(t,\boldsymbol{x})-\bar\rho(\boldsymbol{x})$.
\end{lemma}

\begin{proof}
First, we consider the set $\Gamma$.  We split $\Gamma=(\Gamma\cap\{|f|\le1\})\cup (\Gamma\cap\{|f|>1\}).$
On $\Gamma\cap\{|f|\le1\}$, 
\eqref{onsupportsplit} gives $|f|^2 \le C h_\mu(\rho\mid\bar\rho).$
Therefore, we have
\[\int_{\Gamma\cap\{|f|\le1\}} |f|^2\,\dd\boldsymbol{x}  \le   C H_\mu(t).
\]
Since $p_{\rm loc}<2$ and $\Gamma$ has finite measure, 
the H\"older inequality gives
\[
\begin{aligned}
 \|f\|_{L^{p_{\rm loc}}(\Gamma\cap\{|f|\le1\})}^2  &\le
 |\Gamma|^{2/p_{\rm loc}-1} \|f\|_{L^2(\Gamma\cap\{|f|\le1\})}^2          \le C H_\mu(t).
\end{aligned}
\]
Similarly, since $p_0<\gamma\le2$, then 
\[
\begin{aligned}  \|f\|_{L^{p_0}(\Gamma\cap\{|f|\le1\})}^3 &\le |\Gamma|^{3/p_0-3/2} \|f\|_{L^2(\Gamma\cap\{|f|\le1\})}^3\le C H_\mu(t)^{3/2}.
\end{aligned}
\]

Next, we consider $\Gamma\cap\{|f|>1\}$.  On this set, \eqref{onsupportsplitlarge} gives $|f|^\gamma \le C h_\mu(\rho\mid\bar\rho)$,
so that
\[\int_{\Gamma\cap\{|f|>1\}} |f|^\gamma\,\dd\boldsymbol{x}  \le  C H_\mu(t).
\]
Since $p_{\rm loc}<\gamma$ and $|f|>1$ on this set, then $|f|^{p_{\rm loc}}\le |f|^\gamma$. This implies 
\[
\begin{aligned} \|f\|_{L^{p_{\rm loc}}(\Gamma\cap\{|f|>1\})}^2
&= \bigg(\int_{\Gamma\cap\{|f|>1\}} |f|^{p_{\rm loc}} \,\dd\boldsymbol{x}
   \bigg)^{2/p_{\rm loc}}
   \le \bigg( \int_{\Gamma\cap\{|f|>1\}} |f|^\gamma
 \,\dd\boldsymbol{x} \bigg)^{2/p_{\rm loc}}
 \le C H_\mu(t)^{2/p_{\rm loc}}.
\end{aligned}
\]
Since $p_0<\gamma$, a similar argument gives
\[
\begin{aligned}  \|f\|_{L^{p_0}(\Gamma\cap\{|f|>1\})}^3  &\le  C H_\mu(t)^{3/p_0}.
\end{aligned}
\]

Combining the estimates on $\Gamma\cap\{|f|\le1\}$ and
$\Gamma\cap\{|f|>1\}$, we obtain
\begin{equation}\label{S-ploc}  \|f\|_{L^{p_{\rm loc}}(\Gamma)}^2  \le  C\big(  H_\mu(t)+H_\mu(t)^{2/p_{\rm loc}} \big),
\end{equation}
and
\begin{equation}\label{S-po} \|f\|_{L^{p_0}(\Gamma)}^3 \le C\big( H_\mu(t)^{3/2} + H_\mu(t)^{3/p_0} \big).
\end{equation}

We now turn to $\Gamma^{\rm c}$.  On $\Gamma^{\rm c}$, $f=\rho$.  By
\eqref{offestimate}, we deduce 
\[ 
\int_{\Gamma^{\rm c}}\rho^\gamma\,\dd\boldsymbol{x} \le C H_\mu(t).
\]
Moreover, we obtain
\[\int_{\Gamma^{\rm c}}G_\mu\rho\,\dd\boldsymbol{x} \le C H_\mu(t).
\]
Since $\bar\rho$ is compactly supported and
$\bar{\mathcal W}_\alpha(\boldsymbol{x})\to0$ as
$|\boldsymbol{x}|\to\infty$, we may choose $R>0$ such that
\[K\subset B_R\qquad\text{and}\qquad G_\mu(\boldsymbol{x})\ge \frac{\mu}{2} \quad\text{for } \boldsymbol{x}\in B_R^{\rm c} .
\]
In particular, $B_R^{\rm c}\subset\Gamma^{\rm c}$. Therefore, we infer that
\[ \int_{B_R^{\rm c}}\rho\,\dd\boldsymbol{x} \le \frac2\mu \int_{B_R^{\rm c}}G_\mu\rho\,\dd\boldsymbol{x} \le C H_\mu(t).
\]

We first estimate the bounded vacuum region $B_R\setminus \Gamma$.  
Since $B_R$ has finite measure, for every $p\in(1,\gamma)$,
\[
\begin{aligned} \int_{B_R\setminus \Gamma}\rho^p\,\dd\boldsymbol{x} &\le  |B_R|^{1-p/\gamma} 
\bigg(\int_{B_R\setminus \Gamma}\rho^\gamma\,\dd\boldsymbol{x}\bigg)^{p/\gamma}\le  C H_\mu(t)^{p/\gamma}.
\end{aligned}
\]
Taking $p=p_{\rm loc}$ gives
\[
\begin{aligned}  \|f\|_{L^{p_{\rm loc}}(B_R\setminus \Gamma)}^2  =  \|\rho\|_{L^{p_{\rm loc}}(B_R\setminus \Gamma)}^2 &=
 \bigg(\int_{B_R\setminus \Gamma}\rho^{p_{\rm loc}} \,\dd\boldsymbol{x} \bigg)^{2/p_{\rm loc}}\le  C H_\mu(t)^{2/\gamma}.
\end{aligned}
\]
Taking $p=p_0$ gives
\[
\begin{aligned}
   \|f\|_{L^{p_0}(B_R\setminus \Gamma)}^3  = \|\rho\|_{L^{p_0}(B_R\setminus \Gamma)}^3  
   &=  \bigg(
 \int_{B_R\setminus \Gamma}\rho^{p_0} \,\dd\boldsymbol{x} \bigg)^{3/p_0} \le C H_\mu(t)^{3/\gamma}.
\end{aligned}
\]

It remains to estimate $B_R^{\rm c}$.  Let $p\in(1,\gamma)$. An interpolation
between $L^1(B_R^{\rm c})$ and $L^\gamma(B_R^{\rm c})$ implies
\[ \|\rho\|_{L^p(B_R^{\rm c})} \le
\|\rho\|_{L^1(B_R^{\rm c})}^\lambda
 \|\rho\|_{L^\gamma(B_R^{\rm c})}^{1-\lambda} 
 \qquad \mbox{for $\frac1p=\lambda+\frac{1-\lambda}{\gamma}$}.
\]
Solving for $\lambda$ gives
$\lambda=\frac{\gamma-p}{p(\gamma-1)}\in(0,1).$
Raising the interpolation inequality to the power $p$, and using
\[ \|\rho\|_{L^1(B_R^{\rm c})}
  \le C H_\mu(t), \qquad \|\rho\|_{L^\gamma(B_R^{\rm c})} \le C H_\mu(t)^{1/\gamma},
\]
we obtain 
\[
\begin{aligned}
 \|\rho\|_{L^p(B_R^{\rm c})}^p  \le
  C H_\mu(t)^{p\lambda} H_\mu(t)^{p(1-\lambda)/\gamma} =  C H_\mu(t)^{ p\lambda+p(1-\lambda)/\gamma} = C H_\mu(t),
\end{aligned}
\]
because
\[ p\lambda+\frac{p(1-\lambda)}{\gamma}
 =p\big(\lambda+\frac{1-\lambda}{\gamma}\big) 
 = 1.
\]
Therefore, taking $p=p_{\rm loc}$, we have
\[
\begin{aligned}
 \|f\|_{L^{p_{\rm loc}}(B_R^{\rm c})}^2
 = \|\rho\|_{L^{p_{\rm loc}}(B_R^{\rm c})}^2
&=\left(\|\rho\|_{L^{p_{\rm loc}}(B_R^{\rm c})}^{p_{\rm loc}} \right)^{2/p_{\rm loc}}\le C H_\mu(t)^{2/p_{\rm loc}}.
\end{aligned}
\]
Similarly, taking $p=p_0$, we obtain
\[
\begin{aligned}
\|f\|_{L^{p_0}(B_R^{\rm c})}^3 = \|\rho\|_{L^{p_0}(B_R^{\rm c})}^3
&= \left( \|\rho\|_{L^{p_0}(B_R^{\rm c})}^{p_0} \right)^{3/p_0}\le  C H_\mu(t)^{3/p_0}.
\end{aligned}
\]

Finally we combine the three regions $\mathbb R^n =\Gamma\cup(B_R\setminus \Gamma)\cup B_R^{\rm c}.$
For $p_{\rm loc}$, using
\[  \|f\|_{L^{p_{\rm loc}}(\mathbb R^n)}  \le  \|f\|_{L^{p_{\rm loc}}(\Gamma)} +   \|f\|_{L^{p_{\rm loc}}(B_R\setminus \Gamma)} + \|f\|_{L^{p_{\rm loc}}(B_R^{\rm c})},
\]
we obtain
\[
\begin{aligned}
  \|f\|_{L^{p_{\rm loc}}(\mathbb R^n)}^2
   &\le
  C\|f\|_{L^{p_{\rm loc}}(\Gamma)}^2
  +
  C\|f\|_{L^{p_{\rm loc}}(B_R\setminus \Gamma)}^2
  +
   C\|f\|_{L^{p_{\rm loc}}(B_R^{\rm c})}^2        \\
 &\le C\big( H_\mu(t)+H_\mu(t)^{2/p_{\rm loc}}
+H_\mu(t)^{2/\gamma}\big).
\end{aligned}
\]
This proves \eqref{S-ploc}.

For $p_0$, using
\[
 \|f\|_{L^{p_0}(\mathbb R^n)}
\le \|f\|_{L^{p_0}(\Gamma)} +  \|f\|_{L^{p_0}(B_R\setminus \Gamma)}  +   \|f\|_{L^{p_0}(B_R^{\rm c})},
\]
we obtain
\[
\begin{aligned}
  \|f\|_{L^{p_0}(\mathbb R^n)}^3
 &\le  C\|f\|_{L^{p_0}(\Gamma)}^3  +
   C\|f\|_{L^{p_0}(B_R\setminus \Gamma)}^3 + C\|f\|_{L^{p_0}(B_R^{\rm c})}^3               \\
 &\le C\big( H_\mu(t)^{3/2}
 + H_\mu(t)^{3/p_0} +H_\mu(t)^{3/\gamma}  \big).
\end{aligned}
\]
This proves \eqref{S-po}.  Adding
\eqref{S-ploc} and \eqref{S-po} yields
\eqref{festimates}.
\end{proof}

\begin{lemma}\label{relativenonlocal}
Under the assumptions of {\rm Theorem \ref{thm:stability}}, set
\[
f(t,\boldsymbol{x}):=\rho(t,\boldsymbol{x})-\bar\rho(\boldsymbol{x}).
\]
Then, for {\it a.e.} $t\in(0,T)$,
\begin{equation}\label{relativeestimate}
\int_{\mathbb R^n}\rho\,|\nabla\Phi_\alpha*f|^2\,\dd\boldsymbol{x}
 \le C\Big( \mathcal E_\mu(t) +\mathcal E_\mu(t)^{2/p_{\rm loc}} +\mathcal E_\mu(t)^{2/\gamma} +\mathcal E_\mu(t)^{3/2}+\mathcal E_\mu(t)^{3/p_0}  +\mathcal E_\mu(t)^{3/\gamma}  \Big).
\end{equation}
\end{lemma}

\begin{proof}

Fix time $t$ and, for simplicity, write
$f=f(t,\cdot)$ and $\rho=\rho(t,\cdot)$. Since $\rho=\bar\rho+f$ and $\rho\ge0$, we see that $\rho\le \bar\rho+|f|.$
This implies
\begin{equation}\label{festimate}
\begin{aligned} \int_{\mathbb R^n}\rho|\nabla\Phi_\alpha*f|^2\,\dd\boldsymbol{x}  &\le \int_\Gamma\bar\rho|\nabla\Phi_\alpha*f|^2\,\dd\boldsymbol{x} + \int_{\mathbb R^n}|f||\nabla\Phi_\alpha*f|^2\,\dd\boldsymbol{x}.
\end{aligned}
\end{equation}

We first estimate the first term on the RHS of \eqref{festimate}.  By the choice of
$p_{\rm loc}$, we have
\[ 0<\frac1{p_{\rm loc}}-\frac{n-\alpha-1}{n} \le  \frac12 .
\]
Thus, the HLS inequality gives
\[ \|\nabla\Phi_\alpha*f\|_{L^q(\mathbb R^n)}  \le C\|f\|_{L^{p_{\rm loc}}(\mathbb R^n)}  \qquad \mbox{for $\frac1q
 = \frac1{p_{\rm loc}}-\frac{n-\alpha-1}{n}$}.
\]
The above choice implies $q\ge2$.  Since $\Gamma$ has finite measure and
$\bar\rho\in L^\infty(\mathbb R^n)$, the H\"older inequality gives
\[
\begin{aligned}
 \int_\Gamma\bar\rho|\nabla\Phi_\alpha*f|^2\,\dd\boldsymbol{x}
  &\le \|\bar\rho\|_{L^\infty}\int_\Gamma|\nabla\Phi_\alpha*f|^2\,\dd\boldsymbol{x}          \\
 &\le C |\Gamma|^{1-2/q}
   \|\nabla\Phi_\alpha*f\|_{L^q(\mathbb R^n)}^2          \\
  &\le  C\|f\|_{L^{p_{\rm loc}}(\mathbb R^n)}^2 .
\end{aligned}
\]
Using \eqref{S-ploc}, we obtain
\begin{equation}\label{fbarestimate}
  \int_\Gamma\bar\rho|\nabla\Phi_\alpha*f|^2\,\dd\boldsymbol{x}  \le  C\Big( H_\mu(t)
  +  H_\mu(t)^{2/p_{\rm loc}}  +  H_\mu(t)^{2/\gamma}  \Big).
\end{equation}

We now estimate the second term in \eqref{festimate}.  Recall that
$p_0=\frac{3n}{3n-2(\alpha+1)}.$
A direct computation gives
$
  \frac1{p_0}-\frac{n-\alpha-1}{n}=\frac{\alpha+1}{3n}.$
On the other hand,
\[
 \frac{p_0-1}{2p_0} =
 \frac12\Big(1-\frac1{p_0}\Big)
 = \frac12\Big( 1-\frac{3n-2(\alpha+1)}{3n}  \Big)  
 =  \frac{\alpha+1}{3n}.
\]
Therefore, we have
\[  \frac1{p_0}-\frac{n-\alpha-1}{n}  =
  \frac{p_0-1}{2p_0}.
\]
The HLS inequality then gives
\begin{equation*} 
\|\nabla\Phi_\alpha*f\|_{L^{2p_0/(p_0-1)}(\mathbb R^n)}
  \le   C\|f\|_{L^{p_0}(\mathbb R^n)}.
\end{equation*}
Using the H\"older inequality with exponents $p_0$ and $p_0/(p_0-1)$,
we have
\[
\begin{aligned}
  \int_{\mathbb R^n}|f||\nabla\Phi_\alpha*f|^2\,\dd\boldsymbol{x}
  &\le
 \|f\|_{L^{p_0}(\mathbb R^n)} \big\||\nabla\Phi_\alpha*f|^2
 \big\|_{L^{p_0/(p_0-1)}(\mathbb R^n)}                      \\
  &= \|f\|_{L^{p_0}(\mathbb R^n)}
 \|\nabla\Phi_\alpha*f\|_{L^{2p_0/(p_0-1)}(\mathbb R^n)}^2   \\
 &\le  C\|f\|_{L^{p_0}(\mathbb R^n)}^3 .
\end{aligned}
\]
Using \eqref{festimates} yields
\begin{equation}\label{seesstimate}
\int_{\mathbb R^n}|f||\nabla\Phi_\alpha*f|^2\,\dd\boldsymbol{x}
  \le C\Big( H_\mu(t)^{3/2} +  H_\mu(t)^{3/p_0}  +   H_\mu(t)^{3/\gamma} \Big).
\end{equation}

Combining \eqref{festimate},
\eqref{fbarestimate}, and
\eqref{seesstimate}, we obtain
\[
\begin{aligned}
 \int_{\mathbb R^n}\rho|\nabla\Phi_\alpha*f|^2\,\dd\boldsymbol{x}  &\le
 C\Big(H_\mu(t)+ H_\mu(t)^{2/p_{\rm loc}}
  +   H_\mu(t)^{2/\gamma}  +  H_\mu(t)^{3/2}  +  H_\mu(t)^{3/p_0}
  +   H_\mu(t)^{3/\gamma}  \Big).
\end{aligned}
\]
Since $H_\mu(t)\le \mathcal E_\mu(t),$
the preceding estimate implies
\[
\begin{aligned}
\int_{\mathbb R^n}\rho|\nabla\Phi_\alpha*f|^2\,\dd\boldsymbol{x}
&\le C\Big( \mathcal E_\mu(t) + \mathcal E_\mu(t)^{2/p_{\rm loc}}
  + \mathcal E_\mu(t)^{2/\gamma} +
  \mathcal E_\mu(t)^{3/2}  +
 \mathcal E_\mu(t)^{3/p_0} +
 \mathcal E_\mu(t)^{3/\gamma} \Big).
\end{aligned}
\]
This proves \eqref{relativeestimate}.
\end{proof}

Recall the mechanical entropy and entropy flux:
\[
  \eta(\rho, \mathcal{M})
  = \frac{|\mathcal{M}|^{2}}{2\rho} + h(\rho),
  \qquad
  q(\rho,\mathcal{M})
  =\frac12 \mathcal{M}\frac{|\mathcal{M}|^{2}}{\rho^{2}}+\mathcal{M}h'(\rho).
\]
Define the  relative entropy flux by
\begin{equation}\nonumber
  q_\mu(\rho,\mathcal{M}\mid\bar\rho)
  :=q(\rho,\mathcal{M})+G_\mu\mathcal{M}.
\end{equation}

\begin{lemma}\label{effrelativeei}
Under the assumptions of {\rm Theorem \ref{thm:stability}},
\begin{equation}\label{relativeentropy}
  \partial_t\eta_\mu(\rho,\mathcal{M}\mid\bar\rho)
  +\nabla\cdot q_\mu(\rho,\mathcal{M}\mid\bar\rho)
  \le -\mathcal M\cdot \nabla\Phi_\alpha*(\rho-\bar\rho)
\end{equation}
in the sense of distributions on $(0,T)\times\mathbb R^n$.  Equivalently, for every nonnegative Lipschitz function $\psi$ compactly supported in $[0,T)\times\mathbb R^n$,
\begin{align}\label{testinequality}
&\int_0^T\!\!\int_{\mathbb R^n}
\Big(\partial_\tau\psi\,\eta_\mu(\rho,\mathcal M\mid\bar\rho)
+\nabla\psi\cdot q_\mu(\rho,\mathcal M\mid\bar\rho)\Big)\,\dd\boldsymbol{x}\dd\tau
+\int_{\mathbb R^n}\psi(0,\boldsymbol{x})\,\eta_\mu(\rho,\mathcal M\mid\bar\rho)(0,\boldsymbol{x})\,\dd\boldsymbol{x}
\nonumber\\
&\qquad\ge
\int_0^T\!\!\int_{\mathbb R^n}
\psi\,\mathcal M\cdot\nabla\Phi_\alpha*(\rho-\bar\rho)\,\dd\boldsymbol{x}\dd\tau.
\end{align}
\end{lemma}

\begin{proof}
The entropy solution satisfies
\[
  \partial_t\eta(\rho,\mathcal{M})+
  \nabla\cdot q(\rho,\mathcal{M})
  \le -\mathcal M\cdot\nabla\mathcal W_\alpha
\]
in the sense of distributions, 
with $\mathcal W_\alpha=\Phi_\alpha*\rho$.  Since $G_\mu$ is independent of time and the continuity equation is
$\partial_t\rho+\nabla\cdot\mathcal M=0$, we compute in distributions:
\[\partial_t\big(\eta(\rho,\mathcal M)+G_\mu(\rho-\bar\rho)-h(\bar\rho)\big)
=\partial_t\eta(\rho,\mathcal M)+G_\mu\partial_t\rho=\partial_t\eta(\rho,\mathcal M)-G_\mu\nabla\cdot\mathcal M.
\]
Similarly, we deduce 
\[
\begin{aligned}
\nabla\cdot\big(q(\rho,\mathcal M)+G_\mu\mathcal M\big) &=\nabla\cdot q(\rho,\mathcal M)+G_\mu\nabla\cdot\mathcal M+\mathcal M\cdot\nabla G_\mu.
\end{aligned}
\]
Adding the last two identities cancels $G_\mu\nabla\cdot\mathcal M$ and using 
$\nabla G_\mu=\nabla\bar{\mathcal W}_\alpha$, we obtain
\[
\begin{aligned}
\partial_t\eta_\mu+\nabla\cdot q_\mu
&\le -\mathcal M\cdot\nabla\mathcal W_\alpha +\mathcal M\cdot\nabla\bar{\mathcal W}_\alpha  \\
&=-\mathcal M\cdot\nabla\Phi_\alpha*(\rho-\bar\rho),
\end{aligned}
\]
which is \eqref{relativeentropy}.  
Inequality \eqref{testinequality} follows by multiplying the equation via a nonnegative test function $\psi$, integrating over space-time, and applying integration by parts.
\end{proof}

\bigskip
\begin{proof}[Proof of {\rm Theorem \ref{thm:stability}}]
We divide the proof into six steps.

\smallskip
1. Fix a Lebesgue point $t\in(0,T)$.   Let $R_0>0$ be such that
$K\subset B_{R_0}$.  Choose $L>R_0+1$ and take
$0<\varepsilon<\min\{1,T-t\}$.  For such $\varepsilon$, define
\[
\vartheta_\varepsilon(\tau)=
\begin{cases}
1& \mbox{for $0\le\tau<t$},\\
1+(t-\tau)/\varepsilon& \mbox{for $t\le\tau<t+\varepsilon$},\\
0&\mbox{for $t+\varepsilon\le\tau\le T$},
\end{cases}
\]
and let $\zeta_L\in W^{1,\infty}_{\rm c}(\mathbb R^n)$ be chosen so that
\begin{align*}
&0\le \zeta_L\le1,
\qquad
\zeta_L=1\ \text{on }B_L,
\qquad
\zeta_L=0\ \text{on }B_{L+1}^{\rm c},\\[1mm]
&|\nabla\zeta_L|\le C\mathds{1}_{\{L<|\boldsymbol{x}|<L+1\}},
 \qquad
 \zeta_L(\boldsymbol{x}) \to 1^- 
 \quad\text{as }L\to\infty
 \quad\text{for every fixed }\boldsymbol{x}.
\end{align*}
Put $\psi_\varepsilon(\tau,\boldsymbol{x})=\vartheta_\varepsilon(\tau)\zeta_L(\boldsymbol{x})$ in \eqref{testinequality}.

\smallskip
2. For {\it a.e.} $\tau$, set $F(\tau):=\nabla\Phi_\alpha*(\rho(\tau)-\bar\rho)$.  Since $0\le\psi_\varepsilon\le1$,
\[
\begin{aligned}
\left|\int_{\mathbb R^n}\psi_\varepsilon\mathcal M\cdot F\,\dd\boldsymbol{x}\right|
&\le
\int_{\mathbb R^n}\frac{|\mathcal M|}{\sqrt\rho}\sqrt\rho |F|\,\dd\boldsymbol{x}\\
&\le
\frac12\int_{\mathbb R^n}\frac{|\mathcal M|^2}{\rho}\,\dd\boldsymbol{x}
+\frac12\int_{\mathbb R^n}\rho |F|^2\,\dd\boldsymbol{x}\\
&\le C\big(\mathcal E_\mu(\tau)+\mathcal E_\mu(\tau)^{\chi_*}\big),
\end{aligned}
\]
where Lemma \ref{relativenonlocal} has been used in the last inequality, 
and $\chi_*>1$.  Therefore, we have
\begin{equation}\label{sourceestimate}
\left|\int_0^T\!\!\int_{\mathbb R^n}
\psi_\varepsilon\mathcal M\cdot F\,\dd\boldsymbol{x}\dd\tau\right|
\le C\int_0^{t+\varepsilon}
\big(\mathcal E_\mu(\tau)+\mathcal E_\mu(\tau)^{\chi_*}\big)\,\dd\tau.
\end{equation}

\smallskip
3. Set $A_L:=\{L<|\boldsymbol{x}|<L+1\}$.  Notice that $\bar\rho=0$ on $A_L$. 
Moreover,
$
 |\nabla\zeta_L|\le C\mathds{1}_{A_L}.
$
On the same shell,
\[
 q_\mu(\rho,\mathcal M\mid\bar\rho)
=\frac12\mathcal M\frac{|\mathcal M|^2}{\rho^2}+h'(\rho)\mathcal M+G_\mu\mathcal M.
\]
Write $\mathcal M=\rho u$.  Then
\[
 \left|\frac12\mathcal M\frac{|\mathcal M|^2}{\rho^2}\right|
  =\frac12\rho |u|^3
  =\frac12\frac{|\mathcal M|^3}{\rho^2}.
\]
Since $h'(\rho)=a_0\gamma\rho^{\gamma-1}/(\gamma-1)$, $
 |h'(\rho)\mathcal M|
 \le C\rho^\gamma |u|.$
Using the Young inequality, we have
\[
\rho^\gamma|u|
 \le C\rho |u|^3+C\rho^{(3\gamma-1)/2}
 =C\frac{|\mathcal M|^3}{\rho^2}+C\rho^{\gamma+\theta}.
\]

Finally, since $K\subset B_{R_0}$ and $L>R_0+1$, 
we see that $|\boldsymbol{x}-\boldsymbol{y}|\ge 1$
for every
$\boldsymbol{x}\in A_L$ and every $\boldsymbol{y}\in K$.
Then the singularity of
$\Phi_\alpha$ is avoided uniformly in $L$.  Since $\bar\rho$ is compactly
supported, then
$\sup_{|\boldsymbol{x}|\ge R_0+1}
|\bar{\mathcal W}_\alpha(\boldsymbol{x})|<\infty.$
Consequently,
$\sup_{|\boldsymbol{x}|\ge R_0+1}|G_\mu(\boldsymbol{x})|<\infty.$
The corresponding bound is independent of $L$.  
Therefore, we obtain
\[
|G_\mu\mathcal M|
\le C\rho |u|
\le C\rho+C\rho |u|^3
=C\rho+C\frac{|\mathcal M|^3}{\rho^2}.
\]
Defining
\[
 \mathcal{F}(\tau,\boldsymbol{x}) :=\rho+\rho^{\gamma+\theta}+\frac{|\mathcal M|^3}{\rho^2},
\]
we have
\begin{equation}\label{fluxbound}
\left|\int_0^T\!\!\int_{\mathbb R^n}
\nabla\psi_\varepsilon\cdot q_\mu(\rho,\mathcal M\mid\bar\rho)\,\dd\boldsymbol{x}\dd\tau\right|
\le
C\int_0^T\!\!\int_{A_L} \mathcal{F}\,\dd\boldsymbol{x}\dd\tau.
\end{equation}
The right-hand side is finite because $\rho\in L^\infty(0,T;L^1(\R^n))$ for finite-energy solutions and, by assumptions,
$\rho^{\gamma+\theta}$ and $|\mathcal M|^3/\rho^2$ belong to $L^1(0,T;L^1(\R^n))$.

\smallskip
4. Since
\[
\partial_\tau\psi_\varepsilon =-\varepsilon^{-1}\zeta_L
 \,\,\,\,\hbox{on }(t,t+\varepsilon),
     \qquad
\partial_\tau\psi_\varepsilon=0
\,\,\,\,\hbox{elsewhere},
\]
\eqref{testinequality}, \eqref{sourceestimate}, and \eqref{fluxbound} imply
\begin{align}\label{epsiloninequality}
&\frac1\varepsilon\int_t^{t+\varepsilon}\!\!\int_{\mathbb R^n}
\zeta_L\eta_\mu(\rho,\mathcal M\mid\bar\rho)(\tau)\,\dd\boldsymbol{x}\dd\tau
\nonumber\\
&\quad\le
\int_{\mathbb R^n}\zeta_L(\boldsymbol{x})
\eta_\mu(\rho,\mathcal M\mid\bar\rho)(0,\boldsymbol{x})\,\dd\boldsymbol{x}
+C\int_0^{t+\varepsilon}\big(\mathcal E_\mu(\tau)+\mathcal E_\mu(\tau)^{\chi_*}\big)\,\dd\tau
\nonumber\\
&\qquad
+C\int_0^T\!\!\int_{A_L} \mathcal{F}\,\dd\boldsymbol{x}\dd\tau.
\end{align}
For fixed $L$, the Lebesgue differentiation theorem yields that, for {\it a.e.} such $t$,
\[
\lim_{\varepsilon\to0^+}
\frac1\varepsilon\int_t^{t+\varepsilon}\!\!\int_{\mathbb R^n}
\zeta_L\eta_\mu(\tau)\,\dd\boldsymbol{x}\dd\tau
=
\int_{\mathbb R^n}\zeta_L\eta_\mu(t)\,\dd\boldsymbol{x}.
\]
Letting $\varepsilon\to0^+$ in \eqref{epsiloninequality}, we obtain
\begin{align}\label{localL}
\int_{\mathbb R^n}\zeta_L\eta_\mu(t)\,\dd\boldsymbol{x}
&\le
\int_{\mathbb R^n}\zeta_L\eta_\mu(0)\,\dd\boldsymbol{x}
+C\int_0^t\big(\mathcal E_\mu(\tau)+\mathcal E_\mu(\tau)^{\chi_*}\big)\,\dd\tau
+C\mathcal R(L),
\end{align}
where
\[
 \mathcal R(L):=
 \int_0^T\!\!\int_{A_L} \mathcal{F}\,\dd\boldsymbol{x}\dd\tau.
\]

\smallskip
5. 
Since $ \mathcal{F}\in L^1((0,T)\times\mathbb R^n)$, we have
\[
 \mathcal R(L)\to0
 \qquad\text{as }L\to\infty .
\]
Moreover, $\zeta_L\nearrow1$ pointwise.  Hence, by the monotone convergence
theorem,
\[
\int_{\mathbb R^n}\zeta_L\eta_\mu(t)\,\dd\boldsymbol{x}\to\mathcal E_\mu(t),
  \qquad
\int_{\mathbb R^n}\zeta_L\eta_\mu(0)\,\dd\boldsymbol{x}\to\mathcal E_\mu(0).
\]
Letting $L\to\infty$ in \eqref{localL}, we obtain that, for {\it a.e.} $t\in(0,T)$,
\begin{equation}\label{ODEestimate} \mathcal E_\mu(t)\le
 \mathcal E_\mu(0)
+C\int_0^t\big(\mathcal E_\mu(\tau)+\mathcal E_\mu(\tau)^{\chi_*}\big)\,\dd\tau,
\end{equation}
where $C>0$ depends only on $n,\alpha,\gamma,a_0$, and the steady state.

\smallskip
6. We first prove that, for every fixed $T>0$, there exists a constant
$C_T=C_T(\mathcal E_\mu(0),\bar{\rho},T)>0$ such that
\begin{equation}\label{apriorrb}
 \operatorname*{ess\,sup}_{0\le t\le T}\mathcal E_\mu(t)\le C_T .
\end{equation}
Set
\[
 A(t):=\int_{\mathbb R^n}
 \big(\frac{|\mathcal M|^2}{2\rho}+h(\rho)\big)(t,\boldsymbol{x})\,
 \dd\boldsymbol{x}.
\]
Since $\Phi_\alpha\le0$ and $\bar\rho\ge0$, we see that
$\bar{\mathcal W}_\alpha\le0$.  Moreover, by \eqref{property},
\[
 G_\mu=-h'(\bar\rho)\le0
 \quad\hbox{\it a.e. on }\Gamma,
 \qquad
 0\le G_\mu=\mu+\bar{\mathcal W}_\alpha\le\mu
 \quad\hbox{\it a.e. on }\Gamma^{\rm c} .
\]
Hence, using $\bar\rho=0$ on $\Gamma^{\rm c}$, for {\it a.e.} $t\in(0,T)$,
\[
\begin{aligned}
 \int_{\mathbb R^n}G_\mu(\rho-\bar\rho)(t)\,\dd\boldsymbol{x}
 &=
 \int_{\Gamma}G_\mu\rho(t)\,\dd\boldsymbol{x}
 -\int_{\Gamma}G_\mu\bar\rho\,\dd\boldsymbol{x}
 +\int_{\Gamma^{\rm c}}G_\mu\rho(t)\,\dd\boldsymbol{x}  \\
 &\le
 -\int_{\Gamma}G_\mu\bar\rho\,\dd\boldsymbol{x}
 +\mu\int_{\Gamma^{\rm c}}\rho(t)\,\dd\boldsymbol{x}\le
 C_{\bar\rho}+\mu M .
\end{aligned}
\]
Therefore, we deduce
\[
\begin{aligned}
 \mathcal E_\mu(t)
 = A(t)-\int_{\mathbb R^n}h(\bar\rho)\,\dd\boldsymbol{x}
 +\int_{\mathbb R^n}G_\mu(\rho-\bar\rho)(t)\,\dd\boldsymbol{x}\le A(t)+C_{\bar\rho}+\mu M .
\end{aligned}
\]
By the energy estimate for finite-energy weak solutions,
\[
 \operatorname*{ess\,sup}_{0\le t\le T} A(t)
 \le C(A(0),M).
\]
Consequently, one obtains
\[
 \operatorname*{ess\,sup}_{0\le t\le T}\mathcal E_\mu(t)
 \le C(A(0),M,\bar\rho).
\]

It remains to express the right-hand side only in terms of
$\mathcal E_\mu(0)$ and $\bar\rho$.  Since
$G_\mu\ge0$ on $\Gamma^{\rm c}$ and $G_\mu=-h'(\bar\rho)$ on $\Gamma$,
\[
\begin{aligned}
 A(0)=\mathcal E_\mu(0)
   +\int_{\mathbb R^n}h(\bar\rho)\,\dd\boldsymbol{x}
   -\int_{\mathbb R^n}G_\mu(\rho_0-\bar\rho)\,\dd\boldsymbol{x}\le
 \mathcal E_\mu(0)+C_{\bar\rho}
 +\int_{\Gamma}h'(\bar\rho)\rho_0\,\dd\boldsymbol{x}.
\end{aligned}
\]
Since $\bar\rho\in L^\infty_{\rm c}(\mathbb R^n)$, $h'(\bar\rho)$ is bounded
and compactly supported.  By the Young inequality, for every $\delta>0$,
\[
 h'(\bar\rho)\rho_0
 \le \delta h(\rho_0)+C_{\delta,\bar\rho}
 \qquad\hbox{on }\Gamma .
\]
Therefore, we have
\[
 A(0)\le \mathcal E_\mu(0)+C_{\delta,\bar\rho}+\delta A(0).
\]
Choosing $\delta>0$ sufficiently small and absorbing the last term into the
left-hand side give
\[
 A(0)\le C_{\bar\rho}\bigl(1+\mathcal E_\mu(0)\bigr).
\]

We also control the mass $M$ by $\mathcal E_\mu(0)$.  Since
$\bar\rho$ is compactly supported and
$\bar{\mathcal W}_\alpha(\boldsymbol{x})\to0$ as
$|\boldsymbol{x}|\to\infty$, there exists $R>0$, depending only on
$\bar\rho$ and $\mu$, such that
\[
 K\subset B_R
 \qquad\hbox{and}\qquad
 G_\mu(\boldsymbol{x})
 =\mu+\bar{\mathcal W}_\alpha(\boldsymbol{x})
 \ge \frac{\mu}{2}
 \,\,\,\,\hbox{for }|\boldsymbol{x}|\ge R .
\]
In particular, $B_R^{\rm c}\subset\Gamma^{\rm c}$.  
On $\Gamma$, by the convexity of
$h$,
\[
 h_\mu(\rho_0\mid\bar\rho)
 =h(\rho_0)-h(\bar\rho)-h'(\bar\rho)(\rho_0-\bar\rho)\ge0,
\]
while, on $\Gamma^{\rm c}$,
\[
 h_\mu(\rho_0\mid0)=h(\rho_0)+G_\mu\rho_0\ge G_\mu\rho_0 .
\]
Since the kinetic part is nonnegative, we obtain
\[
 \mathcal E_\mu(0)
 \ge
 \int_{B_R^{\rm c}}G_\mu\rho_0\,\dd\boldsymbol{x}
 \ge
 \frac{\mu}{2}\int_{B_R^{\rm c}}\rho_0\,\dd\boldsymbol{x}.
\]
Thus, we infer that
\[
 \int_{B_R^{\rm c}}\rho_0\,\dd\boldsymbol{x}
 \le C_\mu\mathcal E_\mu(0).
\]
Moreover, since $B_R$ has finite measure,
\[
 \int_{B_R}\rho_0\,\dd\boldsymbol{x}
 \le C_R\Big(1+\int_{B_R}h(\rho_0)\,\dd\boldsymbol{x}\Big)
 \le C_{\bar\rho}\bigl(1+A(0)\bigr)
 \le C_{\bar\rho}\bigl(1+\mathcal E_\mu(0)\bigr).
\]
Therefore, we deduce
\[
 M=\int_{\mathbb R^n}\rho_0\,\dd\boldsymbol{x}
 \le C_{\bar\rho}\bigl(1+\mathcal E_\mu(0)\bigr).
\]
Combining the estimates for $M$ and $A(0)$,  we have
\[
 \operatorname*{ess\,sup}_{0\le t\le T}\mathcal E_\mu(t)
 \le C(\mathcal E_\mu(0),\bar\rho).
\]
This proves \eqref{apriorrb}.

Since $\mathcal E_\mu(t)\le C$ for {\it a.e.} $t\in[0,T]$ and $\chi_*>1$, we have
\[
 \mathcal E_\mu(t)^{\chi_*}
 \le C^{\chi_*-1}\mathcal E_\mu(t)
 \qquad\hbox{for {\it a.e.} }t\in[0,T].
\]
Substituting this into \eqref{ODEestimate} yields
\[
 \mathcal E_\mu(t)
 \le \mathcal E_\mu(0)
 +C\bigl(1+C^{\chi_*-1}\bigr)\int_0^t\mathcal E_\mu(\tau)\,\dd\tau\qquad 
\hbox{for {\it a.e.} }t\in[0,T].
\]
By the Grönwall inequality,
\[
 \mathcal E_\mu(t)
 \le \mathcal E_\mu(0)\exp\{C(1+C^{\chi_*-1})t\}
 \le C_T\mathcal E_\mu(0)
 \qquad \hbox{for {\it a.e.} }t\in[0,T],
\]
where $C_T=C_T(\mathcal E_\mu(0), \bar{\rho}, T)$.  This proves \eqref{finitetimerelative}.
\end{proof}

\medskip
\section{Existence of Solutions for CEREs with  General Pressure Laws}\label{existy}

In this section, we outline the approach to prove the existence of finite-energy weak solutions 
as defined in Definition \ref{definition} with pressure $p$ satisfying \eqref{1.2}--\eqref{1.5}.

\subsection{Approximate PDEs}
The analysis of solutions to \eqref{1.1} in the whole space $\mathbb{R}^n$ is complicated by several major difficulties, including cavitation, singularity formation, and shock development. 
To prove the global existence of finite-energy solutions with spherical symmetry, 
we develop a vanishing physical viscosity method through a family of approximate equations, 
referred to as the compressible Navier–Stokes–Riesz equations (CNSREs), with the form:

\begin{align}\label{0.2}
\begin{cases}
\displaystyle
\partial_t \rho+\nabla \cdot \M=0, \\[1mm]
\displaystyle
\partial_t \M+ \nabla \cdot \Big(\frac{\M\otimes\M}{\rho}\Big)
+\nabla p + \kappa \rho \nabla \mathcal{W}_\alpha
= \varepsilon \nabla \cdot \Big(\mu(\rho)D \big(\frac{\M}{\rho} \big) \Big)
+ \varepsilon \nabla \Big( \lambda(\rho) \nabla \cdot \big(\frac{\M}{\rho} \big) \Big),
\end{cases}
\end{align}
 where $D\big( \frac{\M}{\rho} \big)
 = \frac{1}{2} \big( \nabla (\frac{\M}{\rho})
  + \big(\nabla ( \frac{\M}{\rho}) \big)^\intercal \big)$
is the stress tensor,
$\varepsilon > 0$ can be seen as the inverse of the Reynolds number,
and $\mu(\rho)$ and $\lambda(\rho)$ are the shear and bulk viscosity coefficients respectively
satisfying
 \[
 \mu(\rho) \geq 0, \quad \mu(\rho) + n\lambda(\rho) \geq 0 \qquad \text{ for $\rho \geq 0$}.
 \]
Here the viscosity coefficients satisfy the BD relation:
$\lambda(\rho) = \rho \mu'(\rho) - \mu(\rho)$,
which is needed to derive a BD-type entropy estimate and thereby obtain 
control of the density derivatives; see Bresch–Desjardins \cite{Bresch02} and Bresch–Desjardins–Lin \cite{Bresch_2003}. For simplicity, we choose
$(\mu(\rho),\lambda(\rho)) = (\rho,0)$, 
which satisfies the required conditions. Formally,
as $\varepsilon \to 0$, solutions of \eqref{0.2} converge to solutions of \eqref{1.1}. 
We now introduce the notion of weak solutions
$(\rho^\varepsilon,\M^\varepsilon)(t,\boldsymbol{x})$ 
to \eqref{0.2} corresponding to given approximate initial data:
\begin{equation}\label{initialns}
    (\rho^\varepsilon, \M^\varepsilon)(0,\boldsymbol{x}) = (\rho^\varepsilon_0, \M^\varepsilon_0)(\boldsymbol{x}),
\end{equation}
as defined in \cite[Appendix B]{Carrillo2025}.

\begin{definition}\label{CNSREs}
A measurable vector function $(\rho^\varepsilon,\M^\varepsilon)(t,\boldsymbol{x})$
is a weak solution
of the Cauchy problem \eqref{0.2}--\eqref{initialns} 
with pressure $p$ satisfying \eqref{1.2}--\eqref{1.5} and $(\mu,\lambda) = (\rho,0)$
if $(\rho^\varepsilon,\M^\varepsilon)(t,\boldsymbol{x})$ satisfies
the following conditions{\rm :}
\begin{enumerate}
\item[(a)] $\rho^\varepsilon(t,\boldsymbol{x}) \geq 0$ {\it a.e.} and
$(\M^\varepsilon, \frac{\M^\varepsilon}{\sqrt{\rho^\varepsilon}})(t,\boldsymbol{x})
= \boldsymbol{0}$ {\it a.e.} on the vacuum states
$\{ (t,\boldsymbol{x}) \, : \, \rho^\varepsilon(t,\boldsymbol{x}) = 0 \}$,
    \begin{align*}
&\rho^\v\in L^{\infty}(0,T; L^1\cap L^{\gamma_2}(\mathbb{R}^n)),
\quad   \nabla\sqrt{\rho^\v}\in  L^{\infty}(0,T; L^2(\mathbb{R}^n)),\\
&\frac{\M^\v}{\sqrt{\rho^\v}}\in L^{\infty}(0,T; L^2(\mathbb{R}^n)),
    \end{align*}
for all $T>0$. Moreover, for $\alpha \in (0,n)$,
    \[
 \Phi_\alpha * \rho^\v\in L^\infty(0,T;L^{\frac{2n}{\alpha}}(\mathbb{R}^n) ),\quad
\nabla\Phi_\alpha * \rho^\v\in L^\infty(0,T;L^{\frac{2n}{\alpha + 2}}(\mathbb{R}^n)),
    \]
with, for $\alpha \in [n-1,n)$,
    \[
  \rho^\v \in L^\infty(0,T;C^{0,\beta}(\mathbb{R}^n)),
    \]
    for some $\beta \in (1 + \alpha - n,1)$; and, for $\alpha \in (-1,0]$,
    \[
  \Phi_\alpha * \rho^\v\in L^\infty(0,T;L^\infty_{\rm loc}(\mathbb{R}^n) ),\quad
\nabla\Phi_\alpha * \rho^\v\in L^\infty(0,T;L_{\rm loc}^{\frac{2n}{\alpha + 2}}(\mathbb{R}^n)).
    \]
\item[(b)] For any $\zeta \in C^1_0(\mathbb{R}_+ \times \mathbb{R}^n)$ and $t_2 \geq t_1 \geq 0$,
    \begin{equation*}
\int_{t_1}^{t_2} \int_{\mathbb{R}^n} (\rho^\v \zeta_t + \M^\v \cdot \nabla \zeta) \dd \boldsymbol{x} \dd t = \int_{\mathbb{R}^n} (\rho^\v \zeta)(t_2,\boldsymbol{x}) \dd \boldsymbol{x} - \int_{\mathbb{R}^n} (\rho^\v \zeta)(t_1,\boldsymbol{x}) \dd \boldsymbol{x}.
    \end{equation*}
\item[(c)] For any $\boldsymbol{\psi} \in (C^1_0(\mathbb{R}_+\times \mathbb{R}^n))^n$,
    \begin{align*}
 \begin{split}
 & \int_{\mathbb{R}_+} \int_{\mathbb{R}^n}
\Big( \M^\v \cdot \boldsymbol{\psi}_t
 + \frac{\M^\v}{\sqrt{\rho^\v}} \cdot \big( \frac{\M^\v}{\sqrt{\rho^\v}} \cdot \nabla\big) \boldsymbol{\psi} + p(\rho^\v)\nabla \cdot \boldsymbol{\psi} \Big)(t,\boldsymbol{x}) \dd \boldsymbol{x} \dd t + \int_{\mathbb{R}^n} \M^\v_0(\boldsymbol{x}) \cdot \boldsymbol{\psi}(0,\boldsymbol{x}) \dd \boldsymbol{x} \\
 & = - \v \int_{\mathbb{R}_+} \int_{\mathbb{R}^n}
  \Big( \frac{1}{2} \M^\v \cdot \big( \Delta \boldsymbol{\psi} + \nabla (\nabla \cdot \boldsymbol{\psi} ) \big)
  + \frac{\M^\v}{\sqrt{\rho^\v}} \cdot (\nabla \sqrt{\rho^\v} \cdot \nabla) \boldsymbol{\psi} + \nabla \sqrt{\rho^\v} \cdot \big( \frac{\M^\v}{\sqrt{\rho^\v}} \cdot \nabla \big) \boldsymbol{\psi} \Big) \dd \boldsymbol{x} \dd t \\
 & \quad\, + \int_{\mathbb{R}_+} \int_{\mathbb{R}^n} \kappa(\rho^\v \nabla \mathcal{W}^\v_\alpha \cdot \boldsymbol{\psi})(t,\boldsymbol{x}) \dd \boldsymbol{x} \dd t,\\
 & \text{where }
 \nabla \mathcal{W}^\v_\alpha(\boldsymbol{x}) =
 \begin{cases}
 \displaystyle
     (\nabla \Phi_\alpha * \rho^\v) (\boldsymbol{x}) & \text{for } \alpha \in (-1,n-1), \\[1mm]
     \displaystyle
     \int_{\mathbb{R}^n} \nabla \Phi_\alpha(\boldsymbol{x} - \boldsymbol{y}) \big(\rho^\v(\boldsymbol{y}) - \rho^\v(\boldsymbol{x})\big) \dd \boldsymbol{y}
     \quad& \text{for } \alpha \in [n-1,n).
 \end{cases}
 \end{split}
    \end{align*}
      \end{enumerate}
\end{definition}

To pass to the vanishing-viscosity limit as $\v \to 0$, 
we exploit the compensated compactness framework developed 
in \cite{Schrecker_2019,Chen_2024}.
Consider spherically symmetric solutions of the form
$\rho(t,\boldsymbol{x}) = \rho(t,r)$, $\mathcal{M}(t,\boldsymbol{x}) = m(t,r) \frac{\boldsymbol{x}}{|\boldsymbol{x}|}$ to reduce
the problem to a system of PDEs in the two variables $(t,r)$.
Under this ansatz, \eqref{0.2} takes the form:
\begin{align}\label{Eul}
    \begin{cases}
 \displaystyle
  \rho_t + (\rho u)_r + \frac{n-1}{r}\rho u = 0, \\
 \displaystyle
 (\rho u )_t + (\rho u^2 + p(\rho))_r + \frac{n-1}{r}\rho u^2 + \kappa \rho(\Phi_{\alpha} * \rho)_r = \varepsilon\Big( \rho \big (u_r + \frac{n-1}{r}u \big ) \Big)_r - \varepsilon\frac{n-1}{r}\rho_ru,
    \end{cases}
    \end{align}
 where $u(t,r) = \frac{m(t,r)}{\rho(t,r)}$. 
The limited regularity of solutions presents a major difficulty in the analysis.
To overcome this obstacle, we reformulate the approximate problem as a free-boundary problem. Take our domain to be 
$\Omega_T = \{(t,r) \, :\, a \leq r \leq b(t),  0 \leq t \leq T\}$, where
    \begin{equation}\label{1.16}
    \begin{cases}
      b' (t) = u(t,b(t)) &\quad \mbox{for $t>0$}, \\
      b(0) = b,
    \end{cases}
    \end{equation}
    and $a = b^{-1}$ with $b \gg 1$. With boundary conditions
    \begin{equation}\label{boundary}
    u(t,a) = 0, \quad \big ( p(\rho) - \varepsilon \rho (u_r + \frac{n-1}{r}u) \big)(t,b(t)) = 0  \qquad\,\, \text{ for } t >0.
    \end{equation}
    The initial conditions are prescribed by
    \begin{equation}\label{1.17}
 (\rho,u)|_{t=0}(r) = (\rho_0^{\v,b},u_0^{\v,b})(r) \qquad \text{ for } r \in [a,b],
    \end{equation}
    where $(\rho_0^{\v,b},u_0^{\v,b})(r)$ are defined in \cite[Appendix B]{Carrillo2025} and satisfy that there exists $C_{\v,b} > 0$ such that
    \begin{equation}\label{1.18}
  0 < C_{\v,b}^{-1} \leq \rho_0^{\v,b}(r) \leq C_{\v,b} < \infty.
    \end{equation}
   Owing to the fact that, under this free boundary formulation, 
   we can prove the existence of smooth solutions that have strictly positive density $\rho > 0$ 
   by virtue of \eqref{1.18}, so that the fluid velocity $u$ is well defined.
   Consequently, this formulation is equivalent to the original problem.
   We note that the boundary conditions \eqref{boundary} imposed on the free boundary correspond to 
   a stress-free boundary condition. This makes the formulation physically relevant, since the free
   boundary is transported by the fluid velocity and therefore evolves in concert with the fluid's motion.
   Throughout this paper, we denote solutions of the free boundary problem \eqref{Eul}-\eqref{1.17} by $(\rho^{\v,b},u^{\v,b})$.
   For notational simplicity, however, we suppress the superscript until Section \ref{weaksub} 
   whenever no confusion can arise.

  \begin{remark}
  It follows from $\eqref{Eul}_1$ and \eqref{boundary}, as in {\rm\cite{Chen2021}}, that
    \begin{equation*}
      \rho(t,b(t)) \leq \rho_0(b).
    \end{equation*}
    \end{remark}

 For later use, we define the initial energies:
 \begin{align*}
    &E_0^{\v,b} \coloneq
            \int_{\mathbb{R}^n} \bigg ( \frac{1}{2}\bigg | \frac{\M_0}{\sqrt{\rho_0^{\v,b}}} \bigg |^2 
            + \rho_0^{\v,b} e(\rho_0^{\v,b}) 
            + \frac{\kappa}{2}\rho_0^{\v,b}(\Phi_\alpha *\rho_0^{\v,b})\bigg)(\boldsymbol{x}) \dd \boldsymbol{x} < \infty,\\[2mm]
      &E^{\v,b}_1 = \v^2 \int^b_a \omega_n \Big| \big( \sqrt{\rho^{\v,b}_0} \big)_r \Big|^2 r^{n-1} \dd r.
    \end{align*}
It follows from the construction of $\rho^{\v,b}_0$ that the following equality for mass is still valid:
    \begin{equation*}
 \int^b_a \rho^{\v,b}_0(r) r^{n-1} \dd r = \frac{M}{\omega_n}.
    \end{equation*}
We can simply prove the following conservation of mass 
by using $\eqref{Eul}_1$ and \eqref{1.16}--\eqref{boundary}: 
    \begin{equation*}
 \frac{\d}{\d t} \int^{b(t)}_a \rho(t,r) r^{n-1} \dd r = (\rho u)(t,b(t))b(t)^{n-1} - \int^{b(t)}_a (\rho u r^{n-1})_r(t,r) \dd r = 0.
    \end{equation*}
Thus, we have 
\begin{equation*}
\int^{b(t)}_a \rho(t,r) r^{n-1} \dd r = \frac{M}{\omega_n}\qquad\,\mbox{for all $t \geq 0$},
\end{equation*}
Using this, we can define the Lagrangian coordinates $(\tau,x)$ as in \cite{Chen2021} by
    \begin{equation*}
 x(t,r) = \int^r_a \rho (t,y) y^{n-1} \dd y, \quad\,\, \tau = t.
    \end{equation*}
It is direct to show that this is a smooth bijective mapping from $[0, T] \times [a,b(t)]$ 
to $[0, T] \times [0, \frac{M}{\omega_n}]$. 
This is particularly useful because the latter domain is fixed. Then 
\begin{equation*}
\nabla_{(t,r)} x = (-\rho u r^{n-1},\rho r^{n-1}), \quad \nabla_{(t,r)} \tau = (1,0), 
\quad \nabla_{(\tau,x)} r = (u, \rho^{-1}r^{1-n}), \quad \nabla_{(\tau,x)} t = (1,0).
\end{equation*}
Under these changes of variables, the boundary conditions become
\begin{equation*}
u(\tau,0) = 0, \quad (p-\varepsilon\rho^2(r^{n-1}u)_x)(\tau,\frac{M}{\omega_n}) = 0 
\qquad\,\, \text{ for $\tau \in [0,T]$},
\end{equation*}
and system \eqref{Eul}
can be written in the form:
    \begin{align*}
    \begin{cases}
  \rho_\tau + \rho^2(r^{n-1}u)_x = 0,\\[1mm]
  u_\tau + r^{n-1}p_x + \kappa \rho r^{n-1}(\Phi_{\alpha} * \rho)_x = \varepsilon r^{n-1}(\rho^2(r^{n-1}u)_x)_x - (n-1)\varepsilon r^{n-2}\rho_xu.
    \end{cases}
    \end{align*}

\subsection{Entropy Analysis}
The sound speed $c(\rho)$ is given by $c(\rho):= \sqrt{p'(\rho)}$. 
We define the following function:
\[
    k(\rho) = \int^\rho_0 \frac{\sqrt{p'(s)}}{s} \dd s,
\]
which is fundamental in the entropy analysis. 
To obtain compactness when passing to both the free-boundary limit and the vanishing-viscosity limit, 
we make use of the theory, relating to entropy–entropy flux pairs, as introduced by Lax \cite{Lax}. 
A key ingredient is the establishment of a higher-integrability estimate 
for the velocity. More precisely, we seek to prove the integrability of the quantity $\rho|u|^3$.
To achieve this, we construct a special entropy–entropy flux pair that can control 
$\rho|u|^3$. In particular, we require a special entropy function 
$\hat{\eta}(\rho,u)$ of the form:
\begin{equation*}
    \hat{\eta}(\rho,u) =
    \begin{dcases}
  \frac{1}{2}\rho u^2 + \rho e(\rho) & \text{ for } u \geq k(\rho), \\
 -\frac{1}{2}\rho u^2 - \rho e(\rho) & \text{ for } u \leq -k(\rho),
    \end{dcases}
\end{equation*}
for $u \in (-\infty,-k(\rho)] \cup [k(\rho),\infty)$ and given as a solution of the Goursat problem:
\begin{equation}\label{Goursat}
    \begin{dcases}
 \eta_{\rho \rho} - k'(\rho)^2 \eta_{u u} = 0, \\
 \eta(\rho,u)|_{u = \pm k(\rho)} = \pm \big(\frac{1}{2}\rho u^2 + \rho e(\rho)\big),
    \end{dcases}
\end{equation}
for $u \in (-k(\rho),k(\rho))$. We can prove that this special entropy satisfies the following estimates:

\begin{lemma}[\text{\cite[Lemma 4.1]{Chen_2024}}]\label{specificentropy}
    The Goursat problem \eqref{Goursat} admits a unique solution $\hat{\eta} \in C^2(\mathbb{R}_+ \times \mathbb{R})$ such that
\begin{itemize}
 \item[{\rm(a)}] $|\hat{\eta}(\rho,u)| \leq C (\big | \frac{\mathcal{M}}{\sqrt{\rho}} \big |^2 + \rho^{\gamma(\rho)})$ for $(\rho,u) \in \mathbb{R}_+ \times \mathbb{R}$, where $\gamma(\rho) = \gamma_1$ if $\rho \in [0,\rho_*]$ and $\gamma(\rho) = \gamma_2$ if $\rho \in (\rho_*,\infty)$.
\item[{\rm(b)}] If $\hat{\eta}$ is regarded as a function of $(\rho,u)$, then
  \begin{equation*}
      |\hat{\eta}_\rho(\rho,u)| \leq C (|u|^2 + \rho^{2 \theta(\rho)}), 
      \,\,\,\, |\hat{\eta}_u(\rho,u)| \leq C (\rho|u| + \rho^{\theta(\rho) + 1}) 
     \qquad\mbox{for $(\rho,u) \in \mathbb{R}_+ \times \mathbb{R}$} 
\end{equation*}
    with $\theta(\rho) = \frac{\gamma(\rho) - 1}{2}$. 
    If $\hat{\eta}$ is regarded as a function of $(\rho,m)$, then
 \begin{equation*}
      |\hat{\eta}_\rho(\rho,m)| \leq C (|u|^2 + \rho^{2 \theta(\rho)}), 
     \,\,\,\, |\hat{\eta}_m(\rho,m)| \leq C (|u| + \rho^{\theta(\rho)})\qquad\mbox{for $(\rho,m) \in \mathbb{R}_+ \times \mathbb{R}$}.
       \end{equation*} 
              
\item[{\rm(c)}] If $\hat{\eta}_m$ is regarded as a function of $(\rho,u)$,
 \begin{equation*}
 |\hat{\eta}_{m \rho}(\rho,u)| \leq C \rho^{\theta(\rho) - 1}, \,\,\,\,
     |\hat{\eta}_{m u}(\rho,u)| \leq C
 \qquad\mbox{for $(\rho,u) \in \mathbb{R}_+ \times \mathbb{R}$}.
  \end{equation*}
  If $\hat{\eta}_m$ is regarded as a function of $(\rho,m)$,
\begin{equation*}
 |\hat{\eta}_{m \rho}(\rho,m)| \leq C \rho^{\theta(\rho) - 1}, \quad |\hat{\eta}_{m m}(\rho,u)| \leq C\rho^{-1} \qquad\mbox{for $(\rho,m) \in \mathbb{R}_+ \times \mathbb{R}$}.
 \end{equation*}
\item[{\rm(d)}] If $\hat{q}$ is the corresponding entropy flux determined by
 \begin{equation*}
 \nabla q(\rho,m) = \nabla \eta (\rho,m) \nabla
 \begin{pmatrix}
     m \\
     \frac{m^2}{\rho} + p(\rho)
 \end{pmatrix},
\end{equation*}
then $\hat{q} \in C^2(\mathbb{R}_+ \times \mathbb{R})$ and
 \begin{align*}
 &\hat{q}(\rho,u) = \frac{1}{2}\rho |u|^3 \pm \rho u (e(\rho) + \rho e'(\rho)) && \text{ for } \pm u \geq k(\rho), \\
  &|\hat{q}(\rho,u)| \leq C \rho^{\gamma(\rho) + \theta(\rho)} && \text{ for } |u| < k(\rho), \\
 &\hat{q}(\rho,u) \geq \frac{1}{2} \rho |u|^3 && \text{ for } |u| \geq k(\rho), \\
  &|\hat{q} - u \hat{\eta}| \leq C \rho^{\gamma(\rho)}|u| + \rho^{\gamma(\rho) + \theta(\rho)}
     && \text{ for } (\rho,u) \in \mathbb{R}_+ \times \mathbb{R}.
 \end{align*}
    \end{itemize}
\end{lemma}

    As in \cite{Chen_2000,Chen_2003,Chen_2024}, in order to apply the compensated compactness framework, fundamental to passing the viscosity limit, we need to study the entropy and entropy flux kernels 
    $\chi$ and $\sigma$, where $\chi$ satisfies
    \begin{equation}\label{entropy}
      \begin{cases}
 \chi_{\rho \rho} - k'(\rho)^2 \chi_{u u} = 0, \\
    \chi(0,u,s) = 0, \\
 \chi_\rho(0,u,s) = \delta_{u=s},
     \end{cases}
    \end{equation}
    in the sense of distributions. 
    Since the initial data are only supported at $(\rho,u) = (0,s)$, 
    the solution should only be supported in the domain of dependence given by
    \[
    \mathcal{K} : = \{ \rho \geq 0 \, : \, |s-u| \leq k(\rho) \}.
    \]
   Since the equation is invariant under the transformation $u \mapsto \pm (u - s)$, 
   then $\chi(\rho,u,s) = \chi(\rho,u-s,0)$ so that it is enough to study the case where $s=0$. 
   Equally, considering the Cauchy problem for the equation of
   the entropy flux kernel $\sigma$ given by
    \begin{equation*}
 \begin{cases}
     (\sigma -u \chi)_{\rho \rho} - k'(\rho)^2 (\sigma -u \chi)_{u u} = \frac{p''(\rho)}{\rho} \chi_u, \\
     (\sigma -u \chi)(0,u,s) = 0, \\
      (\sigma -u \chi)_\rho(0,u,s) = 0,
\end{cases}
    \end{equation*}
we see that the function $(\sigma - u \chi)(\rho,u,s) = (\sigma - u \chi)(\rho,u - s,0)$. 
Thus, we can also assume that $s = 0$ for $\sigma - u \chi$. In the $\gamma$-law case, we may find an explicit formula for the entropy and entropy flux pair as in \cite{Chen2021}. 
Given any $\psi \in C^2_{\rm c}(\mathbb{R})$, then the pair $(\eta^\psi,q^\psi)$ defined by
    \begin{align}\label{5.10}
	\begin{cases}
  \displaystyle
\eta^\psi(\rho,m) = \eta^\psi(\rho,\rho u) = \int_\mathbb{R} \chi(\rho;s-u) \psi(s) \dd s,\\[3mm]
  \displaystyle
	  q^\psi(\r,m) = q^\psi(\r, \r u) = \int_\mathbb{R} \sigma(\rho;s-u) \psi(s) \dd s,
	\end{cases}
\end{align}
is a weak entropy-entropy flux pair. 
When applying the compensated compactness argument as in \cite{Schrecker_2019} 
to pass the viscosity limit, 
we rely on the forms of the entropy and entropy flux kernels. 
In particular, we focus upon the analysis of the most singular parts in the corresponding expansions. 
To do this, we require the idea of a fractional derivative,
first introduced in \cite{Chen_1986,DCL_1985}.
We define the fractional derivative $\partial_s^\beta f$ for $\beta > 0$ of a function $f$ as
    \begin{equation*}
  \partial_s^\beta f(s) = \frac{1}{\Gamma(-\beta)} f*[s]_+^{-\beta - 1}.
    \end{equation*}
Taking the Fourier transform of equation \eqref{entropy} with respect to the velocity variable 
and fixing the Fourier variable $\xi$, we obtain a second-order ordinary differential equation 
of the form
    \begin{equation*}
     \begin{cases}
       \hat{\chi}_{\rho \rho} + k'(\rho)^2 \xi^2 \hat{\chi} = 0, \\
      \hat{\chi}(0,\xi) = 0, \\
     \hat{\chi}_\rho(0,\xi) = 1.
\end{cases}
    \end{equation*}
Such a problem is easier to analyze than \eqref{entropy}. 
By looking at the expansion of the Fourier transform of the entropy kernel as in \cite{Chen_2000,Chen_2003,Chen_2024} and applying the inverse Fourier transform, 
we arrive at the following lemma:

\begin{lemma}[\text{\cite[Theorem 2.2]{Chen_2000} and \cite[Lemma 4.6]{Chen_2024}}]
The entropy kernel has the following expansion:
\[
\chi(\rho,u) = a_\sharp(\rho) G_{\lambda_1}(\rho,u) + a_\flat(\rho) G_{\lambda_1 + 1}(\rho,u) + g_1(\rho,u),
\]
with
\[
G_{\lambda}(\rho,u) = [k(\rho)^2 - u^2]^\lambda_+,
\]
and the coefficients $a_\sharp(\rho)$ and $a_\flat(\rho)$ given by
\begin{align*}
\displaystyle &a_\sharp(\rho) := M_{\lambda_1} k(\rho)^{-\lambda_1} k'(\rho)^{-\frac{1}{2}}, \\            \displaystyle &a_\flat(\rho) := -\frac{1}{4(\lambda_1 + 1)}k(\rho)^{-\lambda_1 - 1} 
k'(\rho)^{-\frac{1}{2}} \int^\rho_0 k(s)^{\lambda_1} k'(s)^{-\frac{1}{2}} a_\sharp''(s) \dd s,
\end{align*}
where  $\lambda_1=\frac{3-\gamma_1}{2(\gamma_1-1)}$ and
\[
M_{\lambda_1} := \Big ( \frac{2\lambda_1}{\sqrt{2\lambda_1 + 1}} 
\int^1_{-1} (1-z^2)^{\lambda_1} \dd z \Big )^{-1}.
\]
Furthermore, $\supp (\chi) \subseteq \{ (\rho,u) \, : \, |u| \leq k(\rho) \}$ 
with $\chi(\rho,u) > 0$ in $\{ (\rho,u) \, : \, |u| < k(\rho) \}$. 
For any $\rho_\text{\rm max} > 0$, 
the remainder $g_1(\rho,u)$ and its fractional derivative $\partial^{\lambda_1 + 1}_u g_1(\rho,u)$ 
are H\"older continuous in $(\rho,u)$ for any $0 \leq \rho \leq \rho_\text{\rm max}$ 
with the property that there exists $C(\rho_\text{max}) > 0$ such that
\[
|g_1(\rho,u)| \leq C(\rho_\text{\rm max} )G_{\lambda_1 + 1 + \alpha_0}(\rho,u)
\]
for any $0 \leq \rho \leq \rho_\text{\rm max}$ and some $\alpha_0 \in (0,1)$. 
Moreover, for any $0 \leq \rho \leq \rho_\text{\rm max}$,
 \[
|a_\sharp(\rho)| + \rho^{1-2\theta_1} |a_\sharp'(\rho)| + \rho^{2-2\theta_1} |a_\sharp''(\rho)| 
+ |a_\flat(\rho)| + \rho |a_\flat'(\rho)| + \rho^2 |a_\flat''(\rho)| \leq C(\rho_\text{\rm max}),
\]
where $\theta_1=\frac{\gamma_1-1}{2}$.
Furthermore, there exists $C > 0$ depending on $\rho^* > 0$ such that
 \[
 \|\chi(\rho,\cdot)\|_{L^\infty_u} \leq C\rho
 \qquad\mbox{for $\rho \geq \rho^*$}.
 \]
\end{lemma}

Moreover, we obtain a similar lemma for the expansion of the entropy flux kernel.
\begin{lemma}[\text{\cite[Theorem 2.3]{Chen_2000}}]
 The entropy flux kernel has the following expansion{\rm:}
 \[
  (\sigma - u\chi)(\rho,u) = - u (b_\sharp(\rho) G_{\lambda_1}(\rho,u) + b_\flat(\rho) G_{\lambda_1 + 1}(\rho,u)) + g_2(\rho,u),
 \]
 with
 \[
  G_{\lambda}(\rho,u) = [k(\rho)^2 - u^2]^\lambda_+,
 \]
 and the coefficients $b_\sharp(\rho)$ and $b_\flat(\rho)$ given by
  \begin{align*}
  &b_\sharp(\rho):= M_{\lambda_1} \rho k(\rho)^{-\lambda_1 - 1} k'(\rho)^{\frac{1}{2}}, \\
  &\displaystyle b_\flat(\rho):=\frac{1}{4(\lambda_1 + 1)} k(\rho)^{-(\lambda_1 + 2)} k'(\rho)^{-\frac{1}{2}} \\[2mm]
  &\qquad\quad \,\,\,\times\Big (\displaystyle  \int^\rho_0 k(s)^{\lambda_1 + 1} k'(s)^{-\frac{1}{2}} b_\sharp''(s) \dd s 
  + \rho k'(\rho) \int^\rho_0  k(s)^{\lambda_1} k'(s)^{-\frac{1}{2}} a_\sharp''(s) \dd s\\[2mm]
 &\qquad\qquad \quad\,\,\,\displaystyle  - \int^\rho_0 s k(s)^{\lambda_1} k'(s)^{\frac{1}{2}} a_\sharp''(s) \dd s \Big).
 \end{align*}
For any $\rho_\text{\rm max} > 0$, the remainder $g_2(\rho,u)$ and its fractional derivative $\partial^{\lambda_1 + 1}_u g_2(\rho,u)$ are H\"older continuous in $(\rho,u)$ for any 
 $0 \leq \rho \leq \rho_\text{\rm max}$ with the property that 
there exists $C(\rho_\text{\rm max}) > 0$ such that
 \[
  |g_2(\rho,u)| \leq C(\rho_\text{\rm max} )G_{\lambda_1 + 1 + \alpha_0}(\rho,u),
 \]
 for any $0 \leq \rho \leq \rho_\text{\rm max}$ and some $\alpha_0 \in (0,1)$. 
 Moreover, for any $0 \leq \rho \leq \rho_\text{\rm max}$,
 \[
|b_\sharp(\rho)| + \rho^{1-2\theta_1} |b_\sharp'(\rho)| + \rho^{2-2\theta_1} |b_\sharp''(\rho)| + |b_\flat(\rho)| + \rho |b_\flat'(\rho)| + \rho^2 |b_\flat''(\rho)| \leq C(\rho_\text{\rm max}).
  \]
  Furthermore, there exists $C > 0$ depending on $\rho^* > 0$ such that
 \[
  \|(\sigma - u\chi)(\rho,\cdot)\|_{L^\infty_u} \leq C\rho^{1+\theta_2}\qquad
 \mbox{for $\rho \geq \rho^*$},
  \]
 where $\theta_2=\frac{\gamma_2-1}{2}$.
 \end{lemma}

Employing these estimates on the entropy and entropy flux kernels pair, 
we obtain the following estimates.

\begin{lemma}[\text{\cite[Lemmas 4.5, 4.7 and 4.10]{Chen_2024}}]\label{entropyan}
 Let $\psi \in C^2_{\rm c}(\mathbb{R})$. 
 For any weak entropy pair $(\eta^\psi,q^\psi)$ defined in \eqref{5.10}, 
there exists a constant $C_\psi > 0$ depending on $\rho^* > 0$ and $\psi$ 
 such that
 \begin{enumerate}
 \item[\rm (i)] When $\rho \in [0,2\rho^*]$,
  \[
 |\eta^\psi(\rho,u)| + |q^\psi(\rho,u)| \leq C_\psi \rho.
 \]
 Regarding $\eta^\psi$ as a function of $(\rho,m)$, then
 \[
 |\eta^\psi_m(\rho,m)| + \rho |\eta_{mm}^\psi(\rho,m)| \leq C_\psi, \qquad |\eta^\psi_\rho(\rho,m)| \leq C_\psi(1+\rho^{\theta_1}).
  \]
Regarding $\eta_m^\psi$ as a function of $(\rho,u)$, then
 \[
   |\eta^\psi_{mu}(\rho,m)| + \rho^{1-\theta_1}| \eta_{m\rho}^\psi(\rho,m)| \leq C_\psi.
 \]
 \item[\rm (ii)] When $\rho \geq \rho^*$, then
 \begin{align*}
     &|\eta^\psi(\rho,u)| + |\eta_u^\psi(\rho,u)| + |\eta_{uu}^\psi(\rho,u)| \leq C_\psi \rho,\\
     &|q^\psi(\rho,u)| + \rho|\eta_\rho^\psi(\rho,u)| + \rho^{2-\theta_2}|\eta_{\rho \rho}^\psi(\rho,u)| 
     \leq C_\psi \rho^{1 + \theta_2}
 \end{align*}
  Regarding $\eta^\psi$ as a function of $(\rho,m)$, then
 \[
 |\eta^\psi_\rho(\rho,m)| + \rho^{\theta_2}| \eta_{m}^\psi(\rho,m)| + \rho^{1 + \theta_2}| \eta_{mm}^\psi(\rho,m)| \leq C_\psi.
 \]
 Regarding $\eta_m^\psi$ as a function of $(\rho,u)$, then
\[
  |\eta^\psi_{mu}(\rho,m)| + \rho^{1-\theta_2}| \eta_{m\rho}^\psi(\rho,m)| \leq C_\psi.
   \]
  \end{enumerate}
    \end{lemma}

    As in \cite{Schrecker_2019}, we may obtain a lemma for the exact form of the singularities for the entropy and entropy flux kernels.

    \begin{lemma}[\text{\cite[Lemma 4.14]{Chen_2024}}]\label{singularity}
  The fractional derivatives $\partial_u^{\lambda_1} \chi$ and $\partial_u^{\lambda_1} (\sigma - u \chi)$ admit the following expansions{\rm:}
 \begin{align*}
 \partial_u^{\lambda_1} \chi(\rho,u - s) = & \sum_{\pm} \Big ( A_{1,\pm}(\rho) \delta (s-u \pm k(\rho) ) + A_{2,\pm} H(s-u \pm k(\rho)) \Big ) \\
 & + \sum_{\pm} \Big ( A_{3,\pm}(\rho) PV (s-u \pm k(\rho) ) + A_{4,\pm} Ci(s-u \pm k(\rho)) \Big ) \\
  & + r_\chi(\rho,u - s),
 \end{align*}
 \begin{align*}
 \partial_u^{\lambda_1} (\sigma - u \chi)(\rho,u - s) = & \sum_{\pm}(s-u) \Big ( B_{1,\pm}(\rho) \delta (s-u \pm k(\rho) ) + B_{2,\pm} H(s-u \pm k(\rho)) \Big ) \\
 & + \sum_{\pm}(s-u) \Big ( B_{3,\pm}(\rho) PV (s-u \pm k(\rho) ) + B_{4,\pm} Ci(s-u \pm k(\rho)) \Big ) \\
 & + \sum_{\pm} \Big ( B_{5,\pm}(\rho) H (s-u \pm k(\rho) ) + B_{6,\pm} Ci(s-u \pm k(\rho)) \Big ) \\
  & + r_\sigma(\rho,u - s),
\end{align*}
where $\delta$ is the Dirac distribution, $H$ the Heaviside function, $PV$ the principal value distribution and $Ci$ the cosine integral
\[
 Ci(s) := - \int^\infty_{|s|} \frac{\cos y}{y} \dd y.
 \]
Moreover, the remainder terms $r_\chi$ and $r_\sigma$ are H\"older continuous functions.
Furthermore, there exists a positive constant $C>0$ depending only on $\gamma_1,\gamma_2,\rho_*$, 
and $\rho^*$ such that, for $\rho \geq \rho^*$,
\begin{align*}
 &\sum_{j=1,\pm}^4 |A_{j,\pm}(\rho)| + \sum_{j=1,\pm}^6 |B_{j,\pm}(\rho)| \leq C \rho^{\frac{1}{2} - \frac{\theta_2}{2}},\\
 &\| r_\chi(\rho,\cdot) \|_{C^{\alpha_1}(\mathbb{R})} \leq C\rho, \quad \| r_\sigma(\rho,\cdot) \|_{C^{\alpha_1}(\mathbb{R})} \leq C\rho^{1 + \theta_2}
  \qquad\mbox{for some $\alpha_1 \in (0,\alpha_0]$}.
  \end{align*}  
 \end{lemma}

\subsection{Asymptotic Behavior of the Internal Energy}

As in \cite{Carrillo2025}, in the attractive case $(\kappa = 1)$, obtaining uniform (in $\varepsilon$, $\delta$) energy estimate for the kinetic and potential energies requires controlling the potential energy in terms of 
the internal energy. 
In the polytropic case, for $\gamma \in (\frac{2n}{2n -\alpha},\frac{n + \alpha}{n}]$, 
such a control can be established provided that the total mass is below a critical threshold.
The key ingredient is the HLS inequality, which allows the potential energy to be bounded in terms of the
$L^1$ and $L^\gamma$ norms of the density. 
In the polytropic case, these quantities correspond naturally to the mass and internal energy of the solution, respectively. 
For general pressure laws, however, the situation is more delicate.
Since the pressure may exhibit different asymptotic behaviors as $\rho \to 0$ and $\rho \to \infty$,
the loss of homogeneity prevents a direct application of the HLS inequality 
to bound the potential energy solely in terms of the mass and internal energy. 
Instead, the potential energy can be controlled by the $L^1$ and $L^{\gamma_2}$ norms of the density.
Using the assumption $\gamma_2\le \gamma_1$, one can show that, for any $\beta>0$, 
there exists a constant $C=C(\beta)>0$ such that $\rho^{\gamma_2} \leq \rho(\beta + e(\rho))$.
This observation allows us to recover the required energy control. 
We summarize the result in the following lemma:

\begin{lemma}\label{bigsmall}
Suppose that $p:[0,\infty) \to \mathbb{R}$ satisfies \eqref{Press} and \eqref{1.2}--\eqref{1.5}.
Then, for $\beta > 0$, there exists $C=C(\beta)>0$ such that
\begin{equation}\label{bigandsmallbehav}
\displaystyle
\| \rho \|_{L^{\gamma_2}} \leq C \Big ( \int^\infty_0 \rho (\beta + e(\rho)) r^{n-1} \dd r \Big )^\frac{1}{\gamma_2}. 
\end{equation}
\end{lemma}

\begin{proof}
Notice that \eqref{Press} and \eqref{1.2}--\eqref{1.3} imply that 
$p(\rho) > 0$ for all $\rho > 0$. 
By \eqref{1.2}--\eqref{1.3}, there exists $0< \tilde{\rho}_* \leq \rho_*$ such that, 
for all $0< \rho < \tilde{\rho}_*$, 
$|\mathcal{P}_1(\rho)| \leq \frac{1}{2}$. 
Since $p$ is a continuous positive function, 
there exists $C>0$ such that $\rho^{\gamma_1} \leq C p(\rho)$ for $\tilde{\rho}_* \leq \rho \leq \rho^*$. 
Tying both of these together, there exists $C>0$ such that $\rho^{\gamma_1} \leq C p(\rho)$ 
for $0 < \rho \leq \rho^*$. 
Furthermore, since $\rho^{\gamma_1} \leq C \rho e(\rho)$ for $0 < \rho \leq \rho^*$ 
and $\gamma_1 \geq \gamma_2$, we can obtain the bound: 
$$
\rho^{\gamma_2} \leq C\rho(\beta + e(\rho))\qquad
\mbox{for $0 < \rho \leq \rho^*$}. 
$$
Similarly, from \eqref{1.4}--\eqref{1.5}, it follows that there exists $C>0$ 
such that $\rho^{\gamma_2} \leq C \rho e(\rho)$ for $\rho^* \leq \rho$. 
Then the result follows.
\end{proof}

Next, we want to optimize the value of $\beta$ to minimize the estimate of the potential energy. 
More specifically, we study the asymptotic behavior of the quantity
    \[
    \rho^{\frac{\alpha(\gamma_2-1)}{(2n-\alpha)\gamma_2 - 2n}}
    \big(\beta + e(\rho)\big)^{-\frac{\alpha}{(2n-\alpha)\gamma_2 - 2n}}   \qquad\mbox{for $\beta >0$}. 
    \]
For $\rho < \rho_*$, we have 
\begin{align}\label{1.7}
     e(\rho) = \int^\rho_0 \frac{p(s)}{s^2} \dd s 
  = \kappa_1 \int^\rho_0 s^{\gamma_1 - 2}\big(1 + \mathcal{P}_1(s)\big) \dd s = \frac{\kappa_1}{\gamma_1 - 1} \rho^{\gamma_1 - 1} + \kappa_1 \int^\rho_0 s^{\gamma_1 - 2} \mathcal{P}_1(s) \dd s.
    \end{align}
Focusing on the second term on the RHS of \eqref{1.7} and using \eqref{1.3}, we obtain that, for $\rho < \rho_*$,
    \begin{align}\label{1.8}
 \Big| \int^\rho_0 s^{\gamma_1 - 2} \mathcal{P}_1(s) \dd s \Big| 
\leq C_\# \int^\rho_0 s^{2(\gamma_1 - 1) - 1} \dd s = \frac{C_\#}{2(\gamma_1 - 1)}\rho^{2(\gamma_1 - 1)}.
    \end{align}
Then it follows from  \eqref{1.8} and
$\gamma_2 \leq \gamma_1$ that, for $\rho < \rho_*$,
    \begin{align}\label{1.9}
 \frac{1}{\rho^{\gamma_2 - 1}} \Big| \int^\rho_0 s^{\gamma_1 - 2} \mathcal{P}_1(s) \dd s \Big| 
\leq \frac{C_\#}{2(\gamma_1 - 1)}\rho^{2(\gamma_1 - 1) - (\gamma_2 - 1)} 
 \rightarrow 0 \qquad \text{ as $\rho \to 0$}.
    \end{align}
Combining \eqref{1.7} with \eqref{1.9}, we have
\begin{align}\label{1.10}
        \begin{split}
& \lim_{\rho \to 0} \rho^{\frac{\alpha(\gamma_2-1)}{(2n-\alpha)\gamma_2 - 2n}}(\beta + e(\rho))^{-\frac{\alpha}{(2n-\alpha)\gamma_2 - 2n}} \\
  & = \lim_{\rho \to 0} \Big(\frac{1}{\rho^{\gamma_2 - 1}} 
\big(\beta + \kappa_1 \int^\rho_0 s^{\gamma_1 - 2} \mathcal{P}_1(s) \dd s \big) 
 + \frac{\kappa_1}{\gamma_1 - 1} \rho^{\gamma_1 - \gamma_2}
\Big)^{-\frac{\alpha}{(2n-\alpha)\gamma_2 - 2n}} = 0.
 \end{split}
\end{align}
Similarly, for $\rho \geq \rho^*$, we obtain
    \begin{align}\label{1.11}
  e(\rho) = \frac{\kappa_2}{\gamma_2 - 1} \big(\rho^{\gamma_2 - 1} - (\rho^*)^{\gamma_2 - 1}\big) 
 + \kappa_2 \int^\rho_{\rho^*} s^{\gamma_2 - 2} \mathcal{P}_2(s) \dd s 
+ \int^{\rho^*}_0 \frac{p(s)}{s^2} \dd s.
    \end{align}
It is direct to show that the third integral is well-defined and bounded. 
Using \eqref{1.5}, we can prove that, for $\rho \geq \rho^*$,
\begin{align}\label{1.12}
\Big| \int^\rho_{\rho^*} s^{\gamma_2 - 2} \mathcal{P}_2(s) \dd s \Big| 
\leq C^\# \int^\rho_{\rho^*} s^{\gamma_2 - \varepsilon - 2}  \dd s 
\leq \frac{C^\#}{\gamma_2 - \varepsilon - 1}
\big( \rho^{\gamma_2 - \varepsilon - 1} - (\rho^*)^{\gamma_2 - \varepsilon - 1}\big).
\end{align}
Using \eqref{1.12}, we obtain that, for $\rho \geq \rho^*$,
    \begin{align}\label{1.13}
 \frac{1}{\rho^{\gamma_2 - 1}} \Big| \int^\rho_{\rho^*} s^{\gamma_2 - 2} \mathcal{P}_2(s) \dd s \Big| 
       \leq \frac{C^\#}{\gamma_2 - \varepsilon - 1} 
       \Big( \rho^{- \varepsilon} - \frac{(\rho^*)^{\gamma_2 - \varepsilon - 1}}{\rho^{\gamma_2 -1}} \Big) 
       \rightarrow 0\qquad \text{ as $\rho \to \infty$}.
    \end{align}
Thus, combining \eqref{1.11} with \eqref{1.13} yields
\begin{align}\label{1.14}
& \lim_{\rho \to \infty} \rho^{\frac{\alpha(\gamma_2-1)}{(2n-\alpha)\gamma_2 - 2n}}
  \big(\beta + e(\rho)\big)^{-\frac{\alpha}{(2n-\alpha)\gamma_2 - 2n}}\nonumber \\
& = \lim_{\rho \to \infty} \Big( \frac{\kappa_2}{\gamma_2 - 1} + \frac{1}{\rho^{\gamma_2 - 1}} 
\big( \kappa_2 \int^\rho_{\rho^*} s^{\gamma_2 - 2} \mathcal{P}_2(s) \dd s + \int^{\rho^*}_0 \frac{p(s)}{s^2} \dd s - \frac{\kappa_2}{\gamma_2 - 1} (\rho^*)^{\gamma_2 - 1} \big) \Big)^{-\frac{\alpha}{(2n-\alpha)\gamma_2 - 2n}} \nonumber\\
& = \Big( \frac{\kappa_2}{\gamma_2 - 1} \Big)^{-\frac{\alpha}{(2n-\alpha)\gamma_2 - 2n}}.
\end{align}
Since $\rho \mapsto \rho^{\frac{\alpha(\gamma_2-1)}{(2n-\alpha)\gamma_2 - 2n}}
(\beta + e(\rho))^{-\frac{\alpha}{(2n-\alpha)\gamma_2 - 2n}}$ is continuous 
for $\rho \in (0,\infty)$, 
it follows from \eqref{1.10} and \eqref{1.14} that, for $\beta > 0$,
    \begin{align*}
 C_{\text{max}}(\beta) := \max_{\rho \geq 0} 
 \big\{ \rho^{\frac{\alpha(\gamma_2-1)}{(2n-\alpha)\gamma_2 - 2n}}(\beta + e(\rho))^{-\frac{\alpha}{(2n-\alpha)\gamma_2 - 2n}} \big\} \in [ ( \frac{\kappa_2}{\gamma_2 - 1})^{-\frac{\alpha}{(2n-\alpha)\gamma_2 - 2n}}, \,\infty )
    \end{align*}
    is well-defined.

\subsection{Energy Estimates}
The energy estimates are fundamental in obtaining uniform estimates for the approximate solutions.

\begin{lemma}[Energy Estimates]\label{Energy}
When $\alpha \in (-1,n - 1)$, any smooth solution $(\rho,u)(t,r)$ of the free boundary 
problem \eqref{Eul}--\eqref{1.17} satisfies the following energy estimate{\rm:}
    \begin{align*}
    &\int^{b(t)}_a \rho \big(\frac{1}{2}u^2 + e(\rho) + \frac{\kappa}{2}(\Phi_{\alpha} * \rho) \big)r^{n-1} \dd r 
    + \varepsilon \int^t_0\int^{b(s)}_a \rho \big(u_r^2 + (n-1)\frac{u^2}{r^2} \big) r^{n-1} \dd r\dd s \\ 
    &+ (n-1)\varepsilon\int^t_0 (\rho u^2)(s,b(s))b(s)^{n-2} \dd s 
    = \int^b_a \rho_0 \big(\frac{1}{2}u_0^2 + e(\rho_0) + \frac{\kappa}{2} (\Phi_{\alpha} * \rho_0)\big)r^{n-1} \dd r.
    \end{align*}
Then we have the following estimates{\rm:}\\[2mm] 
\noindent{Case} {\rm 1:} $\alpha \in (0,n-1)$ and $\kappa = -1$ with $\gamma_2 > 1$,
\begin{align*}
\begin{split}
      &\int^{b(t)}_a \rho \big(\frac{1}{2}u^2 + e(\rho) - \frac{1}{2} (\Phi_{\alpha} * \rho) \big)r^{n-1} \dd r
      + \varepsilon \int^t_0\int^{b(s)}_a \rho \big(u_r^2 + (n-1)\frac{u^2}{r^2} \big) r^{n-1} \dd r\dd s \\ 
      &+ (n-1)\varepsilon\int^t_0 (\rho u^2)(s,b(s))b(s)^{n-2} \dd s 
     = \int^b_a \rho_0 \big(\frac{1}{2}u_0^2 + e(\rho_0) - \frac{1}{2} (\Phi_{\alpha} * \rho_0)\big)r^{n-1} \dd r 
= E^{\varepsilon,b}_0.
\end{split}
\end{align*}

\smallskip
\noindent{Case} {\rm 2:} $\alpha \in (0,n-1)$ and $\kappa = 1$ with 
$\gamma_2 \in (\frac{2n}{2n - \alpha}, \frac{n + \alpha}{n}]$ 
and $M < M^{\varepsilon,b,\alpha}_{\rm c}(\gamma)$, there exists $\beta_0 > 0$ such that
    \begin{align*}
    \begin{split}
     &\int^{b(t)}_a \rho \big(\frac{1}{2}u^2 + C_\gamma e(\rho) \big)r^{n-1} \dd r 
      + \varepsilon \int^t_0\int^{b(s)}_a \rho \big(u_r^2 + (n-1)\frac{u^2}{r^2} \big) r^{n-1} \dd r\dd s \\ 
      & + (n-1)\varepsilon\int^t_0 (\rho u^2)(s,b(s))b(s)^{n-2} \dd s \\ 
     & \leq \int^b_a \rho_0 \big(\frac{1}{2}u_0^2 + e(\rho_0)  \big)r^{n-1} \dd r 
     \leq C(M,E^{\varepsilon,b}_0).
     \end{split}
     \end{align*}
where the positive constant $C_\gamma > 0$ is defined as
\[
C_\gamma := \begin{cases}
1 - B_{\beta_0}M^{\frac{n - \alpha}{\alpha}} \quad & \mbox{for $\gamma = \frac{n + \alpha}{n}$},\\[1mm]            \frac{\alpha - n(\gamma_2-1)}{\alpha} 
& \mbox{for $\gamma \in (\frac{2n}{2n - \alpha}, \frac{n + \alpha}{n})$}.
\end{cases}
\]
    
\smallskip    
\noindent{Case} {\rm 3:} $\alpha \in (0,n-1)$ and $\kappa = 1$ with $\gamma_2 > \frac{n + \alpha}{n}$,
    \begin{align*}
    \begin{split}
    &\int^{b(t)}_a \frac{1}{2} \rho \big(u^2 + e(\rho) \big)r^{n-1} \dd r 
     + \varepsilon \int^t_0\int^{b(s)}_a \rho \big(u_r^2 + (n-1)\frac{u^2}{r^2} \big) r^{n-1} \dd r\dd s \\ 
     & + (n-1)\varepsilon\int^t_0 (\rho u^2)(s,b(s))b(s)^{n-2} \dd s \leq C(E^{\varepsilon,b}_0,M).
     \end{split}
    \end{align*}
\smallskip 
\noindent{Case} {\rm 4:} $\alpha \in (-1,0]$ and $\kappa \in \{ -1, 1 \}$ with $\gamma_2 > 1$,
\begin{align*}
\begin{split}
&\int^{b(t)}_a \frac{1}{2} \rho \big(u^2 + e(\rho) + k_\alpha(1+r^2) \big)r^{n-1} \dd r 
+ \varepsilon \int^t_0\int^{b(s)}_a \rho \big(u_r^2 + (n-1)\frac{u^2}{r^2} \big) r^{n-1} \dd r\dd s \\ 
& + (n-1)\varepsilon\int^t_0 (\rho u^2)(s,b(s))b(s)^{n-2} \dd s \leq C(M,E_0^{\varepsilon,b},T).
\end{split}
\end{align*}
\end{lemma}

All cases of the proof, except for {\it Case 2}, follow from the same argument 
as \cite[Lemma 3.4]{Carrillo2025} after replacing $\gamma$ by $\gamma_2$ and 
using \eqref{bigandsmallbehav}. 
For {\it Case 2}, the proof combines the approach developed for general pressure laws 
in \cite[Lemma 5.1]{Chen_2024} with the estimates for the Riesz potential 
in \cite[Lemma 3.4]{Carrillo2025}. 

\smallskip
The uniform estimates established in Lemma \ref{Energy} yield the following corollary:

\begin{corollary}[Uniform Estimates]\label{uniform}
Under the conditions of {\rm Lemma \ref{Energy}}, the following uniform estimates in $(\varepsilon,b)$
{\rm (}for $\varepsilon$ small enough and $b$ large enough{\rm )} hold{\rm :}

\smallskip
\noindent{Case} {\rm 1:} $\kappa = -1$ with $\gamma_2 > 1$ for $\alpha \in (0,n-2]$ 
or $\gamma_2 > \frac{2n}{2n - \alpha}$ for $\alpha \in (n-2,n-1)$,
\begin{align}\label{ueq1}
\int_{\mathbb{R}^n} \rho \big(1 + u^2 + e(\rho) + \rho^{\gamma_2 - 1} - \Phi_\alpha * \rho\big)
\dd \boldsymbol{x} \leq C(M,E_0);
\end{align}

\noindent{Case} {\rm 2:} $\alpha \in (0,n-1)$ and $\kappa = 1$ with $\gamma_2 > \frac{2n}{2n - \alpha}$,
\begin{align*}
\int_{\mathbb{R}^n} \rho \big(1 + u^2 + e(\rho) + \rho^{\gamma_2 - 1} \big) \dd \boldsymbol{x}
   \leq C(M,E_0);
\end{align*}

\noindent{Case} {\rm 3:} $\alpha \in (-1,0]$ and $\kappa \in\{-1,1\}$ with $\gamma_2 > 1$,
\begin{align}\label{ueq3}
\int_{\mathbb{R}^n} \rho \big(1 + u^2 + e(\rho) + \rho^{\gamma_2 - 1} + k_\alpha(1 + |\boldsymbol{x}|^2) \big)
\dd \boldsymbol{x} \leq C(M,E_0,T),
    \end{align}
where we have used the convention that $(\rho, u)$ is zero outside of $\Omega_t$.
In particular, the following uniform estimates in $(\varepsilon,b)$
{\rm (}for $\varepsilon$ small enough and $b$ large enough{\rm )} hold{\rm :} For {\rm Cases 1--2},
\begin{align*}
\|(\Phi_\alpha*\rho)(t)\|_{L^{\frac{2n}{\alpha}}(\mathbb{R}^n)}
+\|(\nabla \Phi_\alpha*\rho)(t)\|_{L^\frac{2n}{\alpha + 2}(\mathbb{R}^n)} \leq C(M,E_0);
    \end{align*}
and for {\rm Case 3}, for any $K \Subset \mathbb{R}^n$, there exists $C(M,E_0,T,K)>0$ such that
\begin{align*}
\|(\Phi_\alpha*\rho)(t)\|_{L^\infty(K)}
+\|(\nabla \Phi_\alpha*\rho)(t)\|_{L^\frac{2n}{\alpha + 2}(K)} \leq C(M,E_0,T,K).
\end{align*}
\end{corollary}

\begin{proof}
Estimates \eqref{ueq1}--\eqref{ueq3} follow from Lemmas \ref{bigsmall}--\ref{Energy} 
and the fact that $\int^{b(t)}_a \rho r^{n-1} \dd r =  \frac{M}{\omega_n}$,
and
$E_0^{\v,b} \to E_0$ as $b \to \infty$ and $\varepsilon \to 0$ as in \cite[Lemma B.8]{Carrillo2025}. The second part of the corollary follows in the same way as \cite[Corollary 3.6]{Carrillo2025} 
after replacing $\gamma$ by $\gamma_2$ and using \eqref{bigandsmallbehav}.
\end{proof}

\subsection{Expanding the Domain}
To guarantee that the approximate solutions converge to global solutions of CNSREs 
in the limit $b\to \infty$, we must verify that the domain $\Omega^T$ exhausts the whole space $\mathbb{R}^n$.

\begin{lemma}[Expanding of the domain $\Omega^T$]\label{expand}
Given $T>0$ and $\v\in(0,\v_0]$, there exists a positive constant $C_1(M,E_0,T,\v)>0$ such that,
if $b\geq \max\{C_1(M,E_0,T,\v),\mathfrak{B}(\v)\}$ {\rm (}where $\mathfrak{B}(\v)$ is defined in {\rm \cite[Lemma B.8]{Carrillo2025}}{\rm )},
	\begin{align*}
	b(t)\geq \frac{b}{2}\qquad\,\, \mbox{for $t\in[0,T]$}.
	\end{align*}
\end{lemma}

This lemma can be achieved by following the same argument as in \cite[Lemma 5.5]{Chen_2024}.

Therefore, as $b \to \infty$, $\lim_{b \to \infty} \inf_{t \in [0,T]} b(t) = \infty$, as we desire.

\subsection{BD Entropy Estimates}
We now turn to the derivation of BD entropy estimates for the approximate solutions. 
These estimates provide bounds, uniform in $b$, on the derivatives of the density, 
which are essential for applying the Aubin–Lions lemma and obtaining compactness of
of $(\rho^{\v,b}, \rho^{\v,b}u^{\v,b})$ as $b \to \infty$.
This compactness will later be used to establish convergence to a global weak solution 
$(\rho^\v, \rho^\v u^\v)$ of \eqref{0.2}.
To derive the BD entropy estimates, we require the viscosity coefficients to 
satisfy the Bresch–Desjardins relation:
$\lambda(\rho) = \rho \mu'(\rho) - \mu(\rho)$. 
As in \cite{Carrillo2025}, we establish two distinct BD-type entropy estimates.
For the first estimate, we assume that $\alpha \leq n-2$ and $\kappa = -1$ (except in the critical case 
$\alpha = n-2$). The condition $\alpha \leq n-2$ ensures that $\Phi_\alpha * \rho$ 
is twice differentiable, allowing the integration-by-parts argument to be carried out rigorously. 
However, this procedure generates the boundary term
$\kappa a^{n-1} (\rho(\Phi_\alpha * \rho)_r)(t,a)$, which has an unfavorable sign in the attractive case 
$\kappa = 1$. The only exception occurs when $\alpha=n-2$
for which this term vanishes identically.  
In general, the boundary contribution is uncontrollable and therefore precludes the direct application 
of the standard BD entropy argument.
To overcome this difficulty, we derive a second BD-type entropy estimate based on a Grönwall-type argument, which avoids the problematic boundary contribution and remains applicable in the attractive regime.

\begin{lemma}[BD-type Entropy Estimate I]\label{oldBD}
Under the conditions of Lemma {\rm \ref{Energy}}, for $\kappa = -1$ with $\alpha \in (-1,n-2]$ 
or $\kappa = 1$ with $\alpha = n-2$, any given $T>0$, $\v\in(0,\v_0]$, 
$b \geq \max\{C_1(M,E_0,T,\varepsilon),\mathfrak{B}(\varepsilon)\}$ {\rm (}as defined in {\rm Lemma \ref{expand}} 
and {\rm \cite[Lemma B.8]{Carrillo2025}}{\rm )}, the following estimate holds{\rm:}
\begin{align*}
\begin{split}
	&\frac{\varepsilon^2}{4}\int_a^{b(t)} |(\sqrt{\rho(t,r)})_r|^2\,r^{n-1} \dd r
	+ 4\varepsilon \int_0^t\int_a^{b(s)} p'(\rho)|(\sqrt{\rho(s,r)})_r|^2\,r^{n-1}\dd r\dd s\\
	&\quad+\frac{1}{n}p(\rho(t,b(t)))b(t)^n
	+\frac{1}{n\varepsilon}\int_0^t p(\rho(s,b(s)))p'(\rho(s,b(s))) b(s)^n\dd s \\
	&\leq C(E_0,M,T) + \varepsilon \kappa \int_0^t ((\Phi_{\alpha} * \rho)_r\rho)(s,a)a^{n-1} \dd s \qquad \text{ for all } t \in [0,T].
 \end{split}
\end{align*}
\end{lemma}

The proof follows from the same argument as \cite[Lemma 3.9]{Carrillo2025} after 
replacing $\gamma$ by $\gamma_2$ and using \eqref{bigandsmallbehav}.

\begin{lemma}[BD-Type Entropy Estimate II]
Under the conditions of {\rm Lemma {\rm \ref{Energy}}}, for $\alpha \in (-1,n-1)$,
$\gamma_2 \geq \frac{3n}{3n - 2(1+\alpha)}$, any given $T > 0$, $\v\in(0,\v_0]$,
and $b \geq \max\{C_1(M,E_0,T,\varepsilon), \mathfrak{B}(\varepsilon)\}$
{\rm (}as defined in {\rm Lemma \ref{expand}} and {\rm \cite[Lemma B.8]{Carrillo2025}}{\rm )},
the following uniform estimate holds for $t \in [0,T]${\rm :}
\begin{align*}
\begin{split}
	&\frac{\varepsilon^2}{4}\int_a^{b(t)} \big|(\sqrt{\rho(t,r)})_r\big|^2\,r^{n-1} \dd r
	+\frac{4 a_0\varepsilon}{\gamma} \int_0^t\int_a^{b(s)} p'(\rho)|(\sqrt{\rho(s,r)})_r|^2\,r^{n-1}\dd r\dd s\\
	&\quad+\frac{1}{n}p(\rho(t,b(t)))b(t)^n
	+\frac{1}{n\varepsilon}\int_0^t p(\rho(s,b(s)))p'(\rho(s,b(s))) b(s)^n\dd s \\
	&\leq C(E_0,M,T).
 \end{split}
\end{align*}
\end{lemma}

The proof follows from the same argument as \cite[Lemma 3.11]{Carrillo2025} after replacing $\gamma$ by $\gamma_2$ 
and using \eqref{bigandsmallbehav}.

\subsection{Higher Integrability}
To pass to the limit $\varepsilon \to 0$ and thereby obtain a finite-energy solution  $(\rho,m)$
of CEREs \eqref{1.1} from weak solutions $(\rho^\v,\rho^\v u^\v)$ of CNSREs \eqref{0.2}, 
we employ the compensated compactness framework developed in \cite{Chen_2024};
see also \cite{chen2009vanishing,Schrecker_2019}. 
A key requirement for this approach is the availability of uniform local higher-integrability 
estimates for both the density and the velocity. 
The derivation of these estimates relies crucially on the entropy analysis 
established in Lemma~\ref{specificentropy}.

\begin{lemma}[Higher Integrability of the Density and Velocity]\label{higher}
Let $(\rho,u)$ be the smooth solution of \eqref{Eul}. Then, under the assumptions of {\rm Theorem {\rm \ref{existence}}},
for $T>0$, $\v\in(0,\v_0]$, and $b \geq \max\{C_1(M,E_0,T,\varepsilon), \mathfrak{B}(\varepsilon)\}$
{\rm (}as defined in {\rm Lemmas \ref{expand}} and {\rm \cite[Lemma B.8]{Carrillo2025}}{\rm )},
the following estimate holds{\rm :}
    \begin{align*}
    \int^T_0 \Big( \int^{b(t)}_d \rho (p(\rho) + \rho^{\gamma_2} )(t,r) \dd r + \int^{D}_d \rho|u|^3(t,r) 
     \dd r \Big) \dd t \leq C(d,D,M,E_0,T)
    \end{align*}
for any $d \in (2b^{-1},b(t))$, $D \in (d,b(t))$, and any $t \in [0,T]$. 
\end{lemma}
The estimate of the first term in the proof follows from the same argument 
as \cite[Lemma 3.12]{Carrillo2025}. 
Furthermore, as in Lemma \ref{bigsmall}, we see that there exists $C>0$ such that
\[
    \int_K \rho^{\gamma_2 + 1}(t,r) \dd r \leq C \int_K \rho(1 + p(\rho))(t,r) \dd r.
\]
Thus, we can infer the estimate on the second term. 
The estimate on the final term follows from a combination of the argument 
for general pressure laws in \cite[Lemma 5.6]{Chen_2024} 
and those for estimates of the Riesz potential in \cite[Lemma 3.13]{Carrillo2025}.

\subsection{Weak Solutions to CNSREs}\label{weaksub}
Having established the uniform energy estimates, BD-type entropy estimates, 
and local higher-integrability estimates for the velocity, 
we are now in a position to apply the Aubin–Lions lemma to obtain compactness of the sequence
$(\rho^{\v,b},\rho^{\v,b} u^{\v,b})$ in suitable local spaces. 
This compactness enables us to pass to the limit $b\to\infty$ and thereby obtain global weak solutions
$(\rho^{\v},\rho^{\v} u^{\v})$ to CNSREs. 
Viewing the approximate solutions as functions of the radial variable, 
the local convergence is obtained on arbitrary compact subsets
$K \Subset (0,\infty)$. 
By Lemma~\ref{expand}, for sufficiently large $b$, we see that
 $K \subset (a,b(t))$ for all $t > 0$. 
 Therefore, the limiting process is well-defined, 
 and the resulting convergence yields global weak solutions to CNSREs.

\begin{lemma}\label{lem5.2}
Under the assumptions of {\rm Theorem {\rm \ref{existence}}}, for fixed $\v\in(0,\v_0]$
{\rm (}as defined in {\rm \cite[Lemma B.8]{Carrillo2025}}{\rm )}, there exists functions $\rho^{\v}(t,r)\ge 0$ almost everywhere
and $m^{\v}(t,r)$ such that, as $b\rightarrow\infty$ $($up to a sub-sequence$)$,
\begin{enumerate}
    \item[{\rm(a)}] 
    $
	(\sqrt{\rho^{\v,b}}, \rho^{\v,b})\longrightarrow (\sqrt{\rho^{\v}}, \rho^{\v})$  {\it a.e.} and strongly in  $C(0,T; L^q_{\rm loc})$ for all $q \in [1,\infty)$,
    where $L^q_{\rm loc}$ denotes  $L^q(K)$ for $K\Subset (0,\infty)$.

    \vspace{0.2cm}

    \item[{\rm(b)}] The sequence $p(\rho^{\v,b})$ is uniformly bounded in $L^\infty(0,T; L^q_{\rm loc})$ for all $q\in[1,\infty]$ and
    \begin{align*}
	p(\rho^{\v,b})\longrightarrow p(\rho^{\v}) \qquad  \mbox{strongly in $L^q(0,T; L^q_{\rm loc})$ for all $q\in[1,\infty)$}.
	\end{align*}

    \vspace{0.2cm}

    \item[{\rm(c)}] $ m^{\v,b}=\rho^{\v,b} u^{\v,b}\longrightarrow m^{\v}(t,r)$ {\it a.e.} and strongly in $L^2(0,T; L^q_{\rm loc})$ for all $q\in[1,\infty)$.

    \vspace{0.2cm}

    \item[{\rm(d)}]There exists
	a function $u^{\v}(t,r)$ such that $m^{\v}(t,r)=\rho^{\v}(t,r) u^{\v}(t,r)$ a.e. and $\,u^\v(t,r)=0$ a.e.
	on $\{(t,r)\,:\,\rho^{\v}(t,r)=0\}$. Moreover, as $b\rightarrow\infty$ $($up to a sub-sequence$)$,
	\begin{align*}
	& m^{\v,b} \longrightarrow m^{\v}=\rho^{\v} u^{\v} \qquad  \mbox{strongly in $L^2(0,T; L^q_{\rm loc})$ for $q\in[1,\infty)$},\\
	& \frac{m^{\v,b}}{\sqrt{\rho^{\v,b}}}\longrightarrow \sqrt{\rho^{\v}} u^{\v}=:\frac{m^\v}{\sqrt{\rho^\v}}
\qquad  \mbox{strongly in $L^2(0,T; L^2_{\rm loc})$}.
\end{align*}
\end{enumerate}
\end{lemma}

The proof follows from combining the approaches of Lemma 4.1, Corollary 4.2, and Lemmas 4.3--4.4 in \cite{Chen2021}. With the local convergence, we can show that the mass equation is satisfied weakly 
for the approximate solutions $(\rho^\v,m^\v) = (\rho^\v,\rho^\v u^\v)$.

\begin{lemma}
Under the assumptions of {\rm Lemma {\rm \ref{lem5.2}}}, let $0\leq t_1<t_2\leq T$, and let $\zeta(t,\boldsymbol{x})\in C^1_{\rm c}([0,T]\times\mathbb{R}^n)$
be any smooth function with compact support.
Then
\begin{align*}
	\int_{\mathbb{R}^n} \rho^{\v}(t_2,\boldsymbol{x}) \zeta(t_2,\boldsymbol{x})\, \dd\boldsymbol{x}
	=\int_{\mathbb{R}^n} \rho^{\v}(t_1,\boldsymbol{x}) \zeta(t_1,\boldsymbol{x})\, \dd\boldsymbol{x}
	+\int_{t_1}^{t_2} \int_{\mathbb{R}^n} \big(\rho^{\v} \zeta_t + \M^{\v}\cdot\nabla\zeta\big)\,\dd\boldsymbol{x}\dd t.
\end{align*}
Moreover, the total mass is conserved{\rm:}
\begin{equation*}
	\int_{\mathbb{R}^n} \rho^\v(t,\boldsymbol{x})\, \dd\boldsymbol{x}=\int_{\mathbb{R}^n} \rho_0^\v(\boldsymbol{x})\, \dd\boldsymbol{x}=M\qquad\, \mbox{for $t\geq0$}.
\end{equation*}
\end{lemma}

The proof follows from the same argument as in \cite{Chen2021}. 
Note that, even with local convergence, 
the mass of the approximate solutions is conserved in the limit. 
Therefore, this allows us to prove global strong convergence of the density in some appropriate norm.
\begin{lemma}
Under the assumptions of {\rm Lemma {\rm \ref{lem5.2}}}, as $b\rightarrow\infty$ $($up to a sub-sequence$)$,
	\begin{equation*}
	\rho^{\v,b} \longrightarrow \rho^{\v}  \quad
 \text{ strongly in $C(0,T; L^q(\mathbb{R}^n))$ for any $q \in [1,\gamma_2)$}.
	\end{equation*}
That is, the convergence of $\rho^{\v,b}$ is global sub-sequentially in $L^q$.
\end{lemma}

The proof is largely analogous to that of \cite{Carrillo2025}, except that it relies on the 
uniform $L^{\gamma_2}$--estimate for the density obtained in Lemma~\ref{uniform}. 
A significant additional difficulty arises from the nonlocal nature of the Riesz potential. 
In contrast to \cite{Chen_2024,Chen2021}, where the potential is governed by a Poisson equation 
and admits a local representation, the Riesz interaction requires a more refined compactness 
argument. In particular, the strong convergence of the density 
is essential for passing to the limit in the potential term as $b\to \infty$.

\begin{lemma}\label{lem5.1}
Under the assumptions of {\rm Lemma {\rm \ref{lem5.2}}}, the following statements hold as $b\rightarrow\infty$ {\rm(}up to a sub-sequence{\rm)}{\rm :}
\begin{align*}
&\Phi_\alpha*\rho^{\v,b} \longrightharpoonup \Phi_\alpha*\rho^\v \quad
	\mbox{weak-$\ast$ in $L^\infty(0,T; W^{1,\frac{2n}{\alpha +2}}_{\rm loc}(\mathbb{R}^n))$ and weakly in $L^2(0,T; W^{1,\frac{2n}{\alpha +2}}_{\rm loc}(\mathbb{R}^n))$},\\
	&(\Phi_\alpha*\rho^{\v,b})_r(t,r)r^{n-1} \longrightarrow (\Phi_\alpha*\rho^\v)_r(t,r)r^{n-1} \quad
	\mbox{ in $C_{\rm loc}([0,T]\times[0,\infty))$},  \\[4mm]
    &(\Phi_\alpha*\rho^{\v,b})(t,r)r^{n-1} \longrightarrow (\Phi_\alpha*\rho^\v)(t,r)r^{n-1} \quad
	\mbox{ in $C_{\rm loc}([0,T]\times[0,\infty))$}, \\[3mm]
    & \int_0^\infty \big|\rho^{\v,b} (\Phi_\alpha*\rho^{\v,b}) - \rho^{\v} (\Phi_\alpha*\rho^{\v})\big|(t,r)\,
    r^{n-1} \dd r \longrightarrow 0. 
\end{align*}
Moreover, when $\alpha \in (0,n-1)$,
\begin{align*}
	\|(\Phi_\alpha*\rho^\v)(t)\|_{L^{\frac{2n}{\alpha}}(\mathbb{R}^n)}
    +\|(\nabla \Phi_\alpha*\rho^\v)(t)\|_{L^\frac{2n}{\alpha + 2}(\mathbb{R}^n)}\leq C(M,E_0)\qquad\,\mbox{for $t\geq0$};
\end{align*}
when $\alpha \in (-1,0]$, for any $K \Subset \mathbb{R}^n$,
\begin{align*}
	\|(\Phi_\alpha*\rho^\v)(t)\|_{L^\infty(K)}
    +\|(\nabla \Phi_\alpha*\rho^\v)(t)\|_{L^\frac{2n}{\alpha + 2}(K)}
    \leq C(M,E_0,T,K)\qquad\,\mbox{for $t\geq0$}.
\end{align*}
\end{lemma}

The proof follows from the same argument as \cite[Lemma 3.9]{Carrillo2025} after 
replacing $\gamma$ by $\gamma_2$ and using \eqref{bigandsmallbehav}. 
Hence, combining the estimates in the Lemmas \ref{Energy}--\ref{uniform} and \ref{oldBD}--\ref{higher} 
and with 
Lemmas \ref{lem5.2}--\ref{lem5.1},
we obtain the following estimates on the approximate solutions.

\begin{lemma}\label{limest}
Under the assumptions of {\rm Lemma {\rm \ref{lem5.2}}}, for any $T>0$,
there exists a global spherically symmetric solution $(\rho^\v, m^\v)=(\rho^\v, \rho^\v u^\v)$
of CNSREs \eqref{0.2} in the sense of {\rm Definition \ref{CNSREs}} satisfying
\begin{enumerate}
    \item[{\rm(a)}] $\rho^{\v}(t,r)\geq0 \,\,\,\, a.e.$

    \vspace{0.2cm}

    \item[{\rm(b)}] $ \displaystyle u^\v(t,r)=0, \,\,\big(\frac{m^\v}{\sqrt{\rho^\v}}\big)(t,r)=\sqrt{\rho^\v}(t,r)u^\v(t,r)=0 \quad\, a.e.\,\, \mbox{on $\{(t,r)\,:\, \rho^\v(t,r)=0\}$}.$

    \vspace{0.2cm}

    \item[{\rm(c)}] $ \displaystyle \int^\infty_0 \rho^{\v}(t,r) r^{n-1} \dd r = \frac{M}{\omega_n} \qquad \mbox{for all $t\in[0,T].$}$

    \vspace{0.3cm}
     \item[{\rm(d)}] For all $t\in[0,T]$,
    \begin{align*}
    &\int_0^\infty  \Big(\frac12\rho^{\v} |u^{\v}|^2+ \rho^{\v}e(\rho^\v) - \frac{\kappa}{2}\rho^\v (\Phi_\alpha * \rho^\v)\Big)(t,r)\,r^{n-1}\dd r
    +\v \int_0^\infty\int_0^\infty (\rho^{\v} |u^{\v}|^2)(s,r) r^{n-3}\,\dd r \dd s\\[2mm]
     &\quad\leq C(M,E_0).
     \end{align*}
     
     \vspace{0.3cm}
    \item[{\rm(e)}] For all $t\in[0,T]$,
   \begin{align*}
 & \int_0^\infty  \Big(\frac12 \Big|\frac{m^{\v}}{\sqrt{\rho^\varepsilon}} \Big|^2
    + \rho^{\v}e(\rho^\v) + \frac{\kappa}{2} \rho^\v (\Phi_\alpha * \rho^\v)\Big)(t,r)\,r^{n-1}\dd r \\[1mm]
 & \quad \leq \int_0^\infty  \Big(\frac12 \Big|\frac{m^{\v}_0}{\sqrt{\rho^\varepsilon_0}} \Big|^2 + \rho^{\v}_0 e(\rho^{\v}_0)
  + \frac{\kappa}{2} \rho^\v_0 (\Phi_\alpha * \rho^\v_0)\Big)(r)\,r^{n-1}\dd r.
  \end{align*}

 \vspace{0.3cm}
    \item[{\rm(f)}]  For all $t\in[0,T]$, 
    \begin{align*}
&\v^2 \int_0^\infty \big|(\sqrt{\rho^{\v}(t,r)})_r\big|^2\,r^{n-1}\dd r
 +\v\int_0^t\int_{0}^\infty p'(\rho^{\v})|(\sqrt{\rho^{\v}(s,r)})_r|^2\,r^{n-1}\dd r\dd s\\[1mm]
 &\quad\leq C(M,E_0,T).
\end{align*}

\vspace{0.3cm}
\item[{\rm(g)}] For $K \Subset (0,\infty)$,
\begin{align*}
\int_0^T\int_K  \rho^{\v} \big(p(\rho^{\v})+(\rho^{\v})^{\gamma_2} + |u^{\v}|^3 \big)(t,r)\, r^{n-1}\dd r \dd t
\leq C(K,M,E_0,T).
\end{align*}
\end{enumerate}
\end{lemma}

The convergence of the potential allows us to prove that the momentum equation is satisfied weakly for the approximate solutions $(\rho^\v,m^\v) = (\rho^\v,\rho^\v u^\v)$.

\begin{lemma}
Under the assumptions of {\rm Lemma {\rm \ref{lem5.2}}}, let $\bp(t,\boldsymbol{x})\in \left(C^2_0([0,T]\times \mathbb{R}^n)\right)^n$ be any smooth function with compact support
so that $\bp(T,\boldsymbol{x})=\boldsymbol{0}$. Then
\begin{align*}
	&\int_{\mathbb{R}_+^{n+1}} \Big\{\M^{\v} \cdot\partial_t\bp +\frac{\M^{\v}}{\sqrt{\rho^{\v}}} \cdot \big(\frac{\M^{\v}}{\sqrt{\rho^{\v}}}\cdot \nabla\big)\bp
	+p(\rho^{\v}) \nabla \cdot \bp -\kappa\rho^\v \nabla(\Phi_\alpha * \rho^\v)\cdot \bp\Big\}\,\dd\boldsymbol{x} \dd t\\
	&\quad +\int_{\mathbb{R}^n} \M_0^\v\cdot \bp(0,\boldsymbol{x})\,\dd\boldsymbol{x}
	\\
	&=-\v\int_{\mathbb{R}_+^{n+1}}
	\Big\{\frac{1}{2}\M^{\v}\cdot \big(\Delta \bp+\nabla(\nabla \cdot\,\bp) \big)
	+ \frac{\M^{\v}}{\sqrt{\rho^{\v}}} \cdot \big(\nabla\sqrt{\rho^{\v}}\cdot \nabla\big)\bp
    + \nabla\sqrt{\rho^{\v}}  \cdot \big(\frac{\M^{\v}}{\sqrt{\rho^{\v}}}\cdot \nabla\big)\bp\Big\}\,\dd\boldsymbol{x}\dd t\nonumber\\
	&=\sqrt{\v}\int_{\mathbb{R}_+^{n+1}}
	\sqrt{\rho^{\v}} \Big\{V^{\v}  \frac{\boldsymbol{x}\otimes\boldsymbol{x}}{r^2}
	+\frac{\sqrt{\v}}{r}\frac{m^\v}{\sqrt{\rho^\v}}\big(I_{n\times n}-\frac{\boldsymbol{x}\otimes\boldsymbol{x}}{r^2}\big)\Big\}: \nabla\bp\, \dd\boldsymbol{x}\dd t,
\end{align*}
where $V^{\v}(t,\boldsymbol{x})\in L^2(0,T; L^2(\mathbb{R}^n))$ is a function such that
$$
\displaystyle\int_0^T\int_{\mathbb{R}^n} |V^{\v}(t,\boldsymbol{x})|^2\, \dd\boldsymbol{x}\dd t\leq C(E_0,M)
\quad \mbox{for some $C(E_0,M)>0$ independent of $T>0$}.
$$
\end{lemma}

The proof follows from the same argument as that of \cite[Lemma 4.10]{Chen2021}.

\subsection{Compensated Compactness Framework}
To pass to the limit $\varepsilon \to 0$, a compactness argument is required in order to obtain strong convergence. 
In the limit $b \to \infty$, the uniform (in $b$) estimates on the derivatives of the density were crucial for applying the Aubin–Lions lemma and establishing compactness. However, the same approach cannot be used in the vanishing-viscosity limit, since the derivative estimates obtained from the BD-type entropy inequalities are not uniform in $\varepsilon$.
Consequently, a compensated compactness framework is required, in which the entropy analysis 
plays a fundamental role. 
In particular, the entropy estimates established in Lemma~\ref{entropyan} allow us 
to prove the following compactness result:

\begin{lemma}[$W_{\rm loc}^{-1,p}$--Compactness]\label{compacty}
Under the assumptions of {\rm Lemma {\rm \ref{lem5.2}}}, let $(\eta, q)$ be a weak entropy pair defined in \eqref{5.10}
for any compactly supported smooth function $\psi(s)$ in $\mathbb{R}$. Then
\begin{align*}
	\partial_t\eta(\rho^\v,m^\v)+\partial_rq(\rho^\v,m^\v) \quad \mbox{is compact in $ W^{-1,p}_{\rm loc}(\mathbb{R}^2_+)$} \text{ for any } p \in [1,2).
\end{align*}
\end{lemma}

The result follows by combining the argument developed in \cite[Lemma 7.1]{Chen_2024} with the treatment of the nonlocal potential as in \cite[Lemma 4.11]{Carrillo2025}.

\begin{remark}
It is worth noting that, if the technical assumption that either
$\gamma_2 \leq \gamma_1$ or $\gamma_1 \in (1,3)$ can be relaxed so as to permit 
$\gamma_2>3$, then the following stronger result holds{\rm :}
    \[
 \partial_t\eta(\rho^\v,m^\v)+\partial_rq(\rho^\v,m^\v) \quad \mbox{is compact in $ H^{-1}_{\rm loc}(\mathbb{R}^2_+)$}.
    \]
\end{remark}

Now, applying the compensated compactness framework in \cite[\S 8]{Chen_2024} 
(also see \cite[\S 4]{Schrecker_2019}),
we can show that there exist radial functions $(\rho,m)(t,r)$ such that
$(\rho^\v,\rho^\v u^\v) \to (\rho,m)$ {\it a.e.} in $\mathbb{R}^2_+$. 
The idea behind the approach is the following:
\begin{enumerate}
    \item[\rm (i)] For the half plane 
    $\mathbb{H} := \{ (\rho,u)\in \mathbb{R}^2_+ \, : \, \rho>0 \}$,
    we define the subset of continuous functions:
    \[
    \qquad\quad\,\, \overline{C}(\mathbb{H}) := \{ \phi \in C(\overline{\mathbb{H}})\, : \, \phi(0,u) \text{ is constant 
    and map } (\rho,u) \mapsto \lim_{s \to \infty} \phi(s\rho,su) \in C(\mathbb{S}^1 \cap \mathbb{H}) \}.
    \]
    It can be shown that there exists a compactification $\overline{\mathcal{H}}$ of $\mathbb{H}$
    such that the weak-* topology is induced by $C(\overline{\mathcal{H}})$. 
    There is also an isometric isomorphism $\iota : \overline{C}(\mathbb{H}) \to C(\overline{\mathcal{H}})$. 
    Let $V$ be the weak-* closure of the vacuum set $\{(0,u) \, : \, u \in \mathbb{R}_+ \}$.

    \vspace{0.3cm}
    \item[\rm (ii)] The fundamental theorem of Young measures \cite[Theorem 2.4]{Alberti_2001} and
    \cite[Proposition 4.1]{Schrecker_2019} implies that there exists a Young measure 
    $\nu_{(t,r)} \in \mathcal{P}(\overline{\mathcal{H}})$ such that, 
    for all $\phi \in C(\overline{\mathbb{H}})$ vanishing on $\partial \mathbb{H}$, 
    then $\phi$ is integrable with respect to $\nu_{(t,r)}$ for {\it a.e.} $(t,r) \in \mathbb{R}^2_+$ and
    \[
 \phi(\rho^\v(t,r),u^\v(t,r)) \longrightharpoonup \int_{\mathbb{H}} \phi(\rho,u) \dd \nu_{(t,r)}(\rho,u) \quad \text{ in } L^1_{\rm loc}(\mathbb{R}^2_+).
    \]
    Moreover, by taking $\phi(\rho,u) = (\rho,\rho u)$, we see that the sequence $(\rho^\v,\rho^\v u^\v)$ converges {\it a.e.} to $(\rho,m)$ if and only if
    \[
 \nu_{(t,r)}(\rho,u) = \delta_{(\rho(t,r),m(t,r))} \quad \text{ for {\it a.e.} } (t,r) \in \{ (t,r) \in \mathbb{R}^2_+ \, : \, \rho(t,r) > 0 \},
    \]
    and $\supp (\nu_{(t,r)}) \subseteq V$ for 
    {\it a.e.} $(t,r) \in \{ (t,r) \in \mathbb{R}^2_+ \, : \, \rho(t,r) = 0 \}$.
    That is, we need to show that $\nu_{(t,r)}(\rho,u)$ is supported in $V$ or at a single point in $\mathbb{H}$ for {\it a.e.} $(t,r) \in \mathbb{R}^2_+$.

    \vspace{0.3cm}
    \item[\rm (iii)] It follows from the div-curl lemma from \cite[Theorem]{Conti_2011}, Lemma \ref{compacty}, 
    and the estimates on the entropy and entropy flux kernels in Lemma \ref{entropyan} 
    to prove the commutator relation:
    \[
     \overline{\chi(s_1) \sigma(s_2) - \chi(s_2) \sigma(s_1)} = \overline{\chi(s_1)} \, \overline{\sigma(s_2)} - \overline{\chi(s_2)} \, \overline{\sigma(s_1)} \qquad \text{ for all } s_1,s_2 \in \mathbb{R},
    \]
    where $\chi(s) := \chi(\cdot,\cdot,s)$, $\sigma(s) := \sigma(\cdot,\cdot,s)$, 
    and $\overline{f} = \int_{\overline{\mathcal{H}}} f \dd \nu_{(t,r)}$.

    \vspace{0.3cm}
    \item[\rm (iv)] The commutator relation along with the representation formula 
    for the $\partial^{\lambda_1 + 1}$ derivatives of the entropy and entropy flux kernels 
    in Lemma \ref{singularity} can be used to show the support of 
    the Young measure $\nu_{(t,r)}$ is either a subset of $V$ or a single point in $\mathbb{H}$.
\end{enumerate}

\subsection{Finite-Energy Solutions to CEREs}
With the {\it a.e.} convergence of the density and momentum now established, 
the estimates in Lemma~\ref{limest} allow us to conclude the following lemma:

\begin{lemma}
    Under the assumptions of {\rm Theorem \ref{existence}}, as $\v \rightarrow 0$ {\rm(}up to a sub-sequence{\rm):}
    \begin{enumerate}
    \item[{\rm(a)}] $
	\rho^{\v} \to \rho$  {\it a.e.} and strongly in  $L^q_{\rm loc}$ for all $q \in [1,\gamma_2 + 1)$, where $L^q_{\rm loc}$ denotes $L^q([0,T] \times K)$ for any $T > 0$ and $K\Subset (0,\infty)$.

    \vspace{0.3cm}

    \item[{\rm(b)}] $ m^{\v}=\rho^{\v} u^{\v}\longrightarrow m(t,r)$ {\it a.e.} and strongly in $L^q_{\rm loc}$ for all $q\in[1,\frac{3(\gamma_2 + 1)}{\gamma_2 + 3})$.

    \vspace{0.3cm}

    \item[{\rm(c)}] $\displaystyle \frac{m^{\v}}{\sqrt{\rho^{\v}}} = \sqrt{\rho^{\v}} u^{\v} \longrightarrow \frac{m}{\sqrt{\rho}}$ {\it a.e.} and strongly in $L^2_{\rm loc}$.

    \vspace{0.3cm}

    \item[{\rm(d)}] $\Phi_\alpha*\rho^{\v} \longrightharpoonup \Phi_\alpha*\rho \quad
	\mbox{weak-$\ast$ in $L^\infty(0,T; W^{1,\frac{2n}{\alpha +2}}_{\rm loc}(\mathbb{R}^n))$ and weakly in $L^2(0,T; W^{1,\frac{2n}{\alpha +2}}_{\rm loc}(\mathbb{R}^n))$}$,

    \vspace{0.3cm}

    \item[{\rm(e)}] $\displaystyle \int^T_0 \int_0^\infty \big|\rho^{\v} (\Phi_\alpha*\rho^{\v})-\rho (\Phi_\alpha*\rho)\big|(t,r)\, r^{n-1} \dd r \dd t \longrightarrow 0$.
\end{enumerate}
Moreover, when $\alpha \in (0,n-1)$,
\begin{align*}
	\|(\Phi_\alpha*\rho)(t)\|_{L^{\frac{2n}{\alpha}}(\mathbb{R}^n)}
    +\|(\nabla \Phi_\alpha*\rho)(t)\|_{L^\frac{2n}{\alpha + 2}(\mathbb{R}^n)}\leq C(M,E_0)\qquad\,\mbox{for $t\geq0$};
\end{align*}
when $\alpha \in (-1,0]$, for any $K \Subset \mathbb{R}^n$,
\begin{align*}
	\|(\Phi_\alpha*\rho)(t)\|_{L^\infty(K)}
    +\|(\nabla \Phi_\alpha*\rho)(t)\|_{L^\frac{2n}{\alpha + 2}(K)}
    \leq C(M,E_0,T,K)\qquad\,\mbox{for $t\geq0$}.
\end{align*}
\end{lemma}

The proof combines the arguments of \cite[\S 9]{Chen_2024} with those of Lemmas~\ref{lem5.1} and \ref{lem5.2}. Proceeding as in \cite[\S 5]{Chen2021}, we deduce that $(\rho,m)(t,r)$ is a global finite-energy solution 
of CEREs \eqref{1.1} in the sense of Definition~\ref{definition}, 
thereby proving Theorem~\ref{existence}.

\appendix

\section{Convolution Inequalities}
We now present several convolution inequalities that have been used in the main text of this paper.

\begin{lemma}[Hardy-Littlewood-Sobolev (HLS) Inequality]\label{simplehls}
	Let $\alpha \in (0,n)$ and $1 < p,r < \infty$ such that
	\[
	\frac{1}{p} + \frac{\alpha}{n} = 1 + \frac{1}{r}.
	\]
Then there exists a constant $C>0$ such that,
    for all $f \in L^p(\mathbb{R}^n)$,
	\begin{equation*}
	\big\| |\cdot |^{-\alpha} * f \big\|_{L^r(\mathbb{R}^n)} \leq C\|f\|_{L^p(\mathbb{R}^n)}.
	\end{equation*}    
    When $r = \frac{2n}{\alpha}$ and $p=\frac{2n}{2n-\alpha}$, the constant $C=C_{n,\alpha}$ given by
    \[
    C_{n,\alpha} := \pi^\frac{\alpha}{2} \frac{\Gamma(\frac{n - \alpha}{2})}{\Gamma( n - \frac{\alpha}{2})} 
    \bigg( \frac{\Gamma(\frac{n}{2})}{\Gamma(n)} \bigg)^{\frac{\alpha - n}{n}}.
    \]
\end{lemma}

\begin{lemma}[Variation of HLS Inequality \text{\cite[Theorem 3.1]{Calvez_2017}}]\label{HLSineq}
Suppose that $\alpha \in (0,n)$ and $p \geq \frac{2n}{2n - \alpha}$.
Then there exists $C_* = C_*(\alpha,p,n) > 0$ such that, for any $f \in L^1(\mathbb{R}^n) \cap L^p(\mathbb{R}^n)$,
\begin{equation}\label{HLS}
\Big| \int_{\mathbb{R}^n} \int_{\mathbb{R}^n} f(\boldsymbol{x}) |\boldsymbol{x} - \boldsymbol{y}|^{-\alpha} f(\boldsymbol{y}) \dd \boldsymbol{x} \dd \boldsymbol{y} \Big| 
\leq C_* \| f \|_{L^1}^\frac{p (2n - \alpha) - 2n}{n(p - 1)} \| f \|_{L^p}^\frac{\alpha p}{n(p - 1)}.
\end{equation}
In particular, when $p = \frac{n + \alpha}{n} > \frac{2n}{2n - \alpha}$, for any
$f \in L^1(\mathbb{R}^n) \cap L^{\frac{n + \alpha}{n}}(\mathbb{R}^n)$,
\[
\Big| \int_{\mathbb{R}^n} \int_{\mathbb{R}^n} f(\boldsymbol{x}) |\boldsymbol{x} - \boldsymbol{y}|^{-\alpha} f(\boldsymbol{y}) \dd \boldsymbol{x} \dd \boldsymbol{y} \Big|
\leq C_* \| f \|_{L^1}^\frac{n - \alpha}{n} \| f \|_{L^{\frac{n + \alpha}{n}}}^\frac{n + \alpha}{n}.
\]
Furthermore, $C_*$ is the sharp constant that satisfies \eqref{HLS} and is strictly bounded above by the standard HLS constant $C_{n,\alpha}$, i.e., $C_* < C_{n,\alpha}$.
\end{lemma}

\begin{lemma}[Weak Young's Convolution Inequality]\label{precised}
	Let $1 < q < \infty$ and $p \geq \frac{2q}{2q-1}$. 
    Then there exists $C_*^h> 0$ such that, for any $h \in L^q_{\rm w}(\mathbb{R}^n)$ 
    and $f,g \in L^1(\mathbb{R}^n) \cap L^p(\mathbb{R}^n)$,
	\[
	\Big| \int_{\mathbb{R}^n} \int_{\mathbb{R}^n} f(\boldsymbol{x}) h(\boldsymbol{x} - \boldsymbol{y}) g(\boldsymbol{y}) \dd \boldsymbol{x} \dd \boldsymbol{y} \Big| 
    \leq C_*^h (\|f\|_{L^1}\|g\|_{L^1})^\frac{p(2q-1)-2q}{2q(p-1)} (\|f\|_{L^p}\|g\|_{L^p})^\frac{p}{2q(p-1)}.
	\]
\end{lemma}

\begin{lemma}[Riesz's Rearrangement Inequality \cite{Lieb_2001}]
    Let $f, g, h:\,\mathbb{R}^n \to \mathbb{R}_+$ be measurable functions. Define
    \[
    I(f,g,h) := \int_{\mathbb{R}^n} \int_{\mathbb{R}^n} f(\boldsymbol{x}) g(\boldsymbol{x} - \boldsymbol{y}) h(\boldsymbol{y}) \dd \boldsymbol{x} \dd \boldsymbol{y}.
    \]
    For a given measurable function $\phi : \mathbb{R}^n \to \mathbb{R}_+$, define the radially decreasing rearrangement $\phi^{\#}$ of $\phi$ as
    \begin{equation*}
    \phi^{\#} (\boldsymbol{x}) := \int^\infty_0 \mathds{1}_{\{\boldsymbol{y} \, : \, \phi(\boldsymbol{y}) > t\}^{\#}} (\boldsymbol{x}) \dd t
    \end{equation*}
    with
    \[
    A^{\#} = \{ \boldsymbol{x} \, : \, \omega_n |\boldsymbol{x}|^n \leq |A|\}
    \]
    for any Lebesgue measurable set $A \subset \mathbb{R}^n$. Then
    \begin{equation}\label{rearrangeineq}
    I(f,g,h) \leq I(f^{\#},g^{\#},h^{\#}).
    \end{equation}
    If $g$ is a strictly symmetric-decreasing function, then the equality holds in \eqref{rearrangeineq}
    if and only if there exists some $\boldsymbol{z} \in \mathbb{R}^n$ such that $f(\boldsymbol{x}) = f^{\#}(\boldsymbol{x} - \boldsymbol{z})$ and $h(\boldsymbol{x}) = h^{\#}(\boldsymbol{x} - \boldsymbol{z})$.
\end{lemma}

\section{Estimates on the Radial Nonlocal Potential}
Let $f \in L^1(\mathbb{R}^n) \cap L^q_{\rm loc}(\mathbb{R}^n \setminus \{\boldsymbol{0}\})$
be a radial function, where $q \in (\frac{n}{n - 1 - \alpha}, \infty)$.
Then $(\Phi_\alpha*f)(\boldsymbol{x})$ is a radial function so that we are able to represent
it as a convolution of some potential with the radial function.
Like in \cite{BALAGUE20135}, we can represent the convolution by
\begin{equation*}
    (\Phi_\alpha*f)(\boldsymbol{x}) = \int^{\infty}_{0} K(|\boldsymbol{x}|,\eta) f(\eta) \eta^{n-1}
    \dd \eta,
\end{equation*}
where
\begin{equation*}
    K(r,\eta) := \begin{cases}
  \displaystyle
 \int_{\mathbb{S}^{n-1}} \frac{|r\boldsymbol{e_1}-\eta \boldsymbol{y}|^{-\alpha}}{-\alpha} \dd S (\boldsymbol{y}) & \mbox{ for $\alpha \neq 0$}, \\[4mm]
\displaystyle
 \int_{\mathbb{S}^{n-1}} \log(|r\boldsymbol{e_1}-\eta \boldsymbol{y}|) \dd S (\boldsymbol{y}) \quad &\mbox{ for $\alpha = 0$}, \\
\end{cases}
\end{equation*}
with $\boldsymbol{e_1} = (1,0,\cdots,0)$.
We now consider the estimates of the potential term and the derivative of the potential term 
for radial functions.

\begin{lemma}[\text{\cite[Lemma 3.1]{Carrillo2025}}]\label{potentially}
Let $\alpha \in (-1,n-1)$.
Given a radial function $f \in L^1(\mathbb{R}^n) \cap L^q_{\rm loc}(\mathbb{R}^n \setminus \{\boldsymbol{0}\})$ for some $q \in (\frac{n}{n - 1 - \alpha}, \infty)$ such that $\Phi_\alpha*f$
and $\nabla\Phi_\alpha*f$ are well-defined,
then $(\Phi_\alpha*f)(\boldsymbol{x})$ is continuously differentiable on $\mathbb{R}^n \setminus \{\boldsymbol{0}\}$, and
    \begin{equation*}
 (\nabla \Phi_\alpha*f)(\boldsymbol{x}) \cdot \frac{\boldsymbol{x}}{|\boldsymbol{x}|}
 = (\Phi_\alpha*f)_r(|\boldsymbol{x}|)
 = \int^{\infty}_0 \omega(|\boldsymbol{x}|,\eta) f(\eta)\,\eta^{n-1}\dd \eta
\end{equation*} is a radial function, where
\[
 \omega(r,\eta) := \int_{\mathbb{S}^{n-1}} \nabla \Phi_\alpha(r\boldsymbol{e_1} - \eta \boldsymbol{y}) \cdot \boldsymbol{e_1} \dd S(\boldsymbol{y}).
\]
Moreover, there exists $C = C(\alpha,n) > 0$ such that
\begin{align*}
\begin{split}
\begin{cases}
    \displaystyle
    |K(r,\eta)| \leq C (r\eta)^{-\frac{\alpha}{2}} & \mbox{ for $\alpha \in(-1,n-1)$}, \\[1mm]
    \displaystyle
    |K(r,\eta)| \leq C(r\eta)^{-\frac{n-1}{2}} |r-\eta|^{n-1-\alpha} & \mbox{ for $\alpha \in (n-1,n)$},
\end{cases}
\end{split}
\end{align*}
and
\begin{align*}
\begin{split}
\begin{cases}
\displaystyle
|\omega(r,\eta)| \leq C (r\eta)^{-\frac{\alpha + 1}{2}} & \mbox{ for $\alpha \in(-1,n-2)$}, \\[1mm]
 \displaystyle
 \omega(r,\eta) = \frac{\omega_n \mathds{1}_{(0,r)}(\eta)}{r^{n-1}}
 & \mbox{ for $\alpha = n-2$}, \\[1mm]
 \displaystyle
|\omega(r,\eta)| \leq C(r\eta)^{-\frac{n-1}{2}} |r-\eta|^{n-2-\alpha}
 & \mbox{ for $\alpha \in (n-2,n-1)$},
\end{cases}
\end{split}
\end{align*}
    where $\mathds{1}_{(0,r)}$ is the indicator function on interval $(0,r)$.
\end{lemma}

We give the definition of fractional Sobolev spaces for completeness.

\begin{definition}[Fractional Sobolev Spaces $W^{s,p}$]\label{fraccy}
Let $\Omega \subseteq \mathbb{R}^n$ be an open set, 
$s \in (0,1)$, and $1 \leq p < \infty$.
Then the fractional Sobolev space $W^{s,p}(\Omega)$ is defined by
    \[
 W^{s,p}(\Omega) := \Big\{ u \in L^p(\Omega) \, : \, \frac{|u(\boldsymbol{x}) - u(\boldsymbol{y})|}{|\boldsymbol{x} - \boldsymbol{y}|^{\frac{n}{p}+s}} \in L^p(\Omega \times \Omega) \Big\}.
    \]
    We define the semi-norm $[\,\cdot\,]_{W^{s,p}(\Omega)}$ by
    \[
 [u]_{W^{s,p}(\Omega)} := \Big( \int_{\Omega} \int_{\Omega} \frac{|u(\boldsymbol{x}) - u(\boldsymbol{y})|^p}{|\boldsymbol{x} - \boldsymbol{y}|^{n+sp}} \dd \boldsymbol{x} \dd \boldsymbol{y} \Big)^\frac{1}{p},
    \]
    and the norm of $W^{s,p}(\Omega)$ by
    \[
 \|u\|_{W^{s,p}(\Omega)} := \Big ( \|u\|_{L^p(\Omega)}^p + [u]_{W^{s,p}(\Omega)}^p \Big )^\frac{1}{p}.
    \]
\end{definition}

\begin{lemma}[Compact Embeddings of $W^{s,p}$ \text{\cite[Corollary 7.2]{Di_Nezza_2012}}]\label{compactfraccy}
Let $\Omega \subseteq \mathbb{R}^n$ be a bounded extension domain for $W^{s,p}$, $s \in (0,1)$, 
and $1 \leq p < \infty$ such that $sp < n$. 
For $q \in [1,p^*)$ with $p^* = \frac{np}{n-sp}$ as the fractional critical exponent, then
\[
    W^{s,p}(\Omega) \Subset L^q(\Omega).
\]
\end{lemma}

\bigskip
\bigskip
\noindent{\bf Acknowledgments.}
JAC, GQC, and DF acknowledge support by EPSRC grant EP/V051121/1. JAC was also partially supported by the Advanced Grant Nonlocal-CPD (Nonlocal PDEs for Complex Particle Dynamics: Phase Transitions, Patterns and Synchronization) of the European Research Council Executive Agency (ERC) under the European Union Horizon 2020 research and innovation programme (grant agreement No. 883363). GQC was also partially supported by EPSRC grant EP/V008854. DF was also partially supported by the Fundamental Research Funds for the Central Universities No. 2233100021 and No. 2233300008.
For the purpose of open access, the authors have applied a CC BY public copyright license to any Author Accepted Manuscript (AAM) version
arising from this submission.

\bigskip
\medskip
\noindent{\bf Conflict of Interest:} The authors declare that they have no conflict of interest.
The authors also declare that this manuscript has not been previously published,
and will not be submitted elsewhere before your decision.

\bigskip
\noindent{\bf Data availability:} Data sharing is not applicable to this article as no datasets were generated or analyzed during the current study.

\end{document}